\pdfoutput=1

\documentclass{amsart}%

\pdfoutput=1

\usepackage[T1]{fontenc}
\usepackage[mathscr]{eucal}% so-called Euler fonts, allows \mathscr
\usepackage{amssymb}% allows special arrows like >-> e.g.
\usepackage{booktabs} %pretty tables
\usepackage[usenames,dvipsnames]{xcolor} %changed to xcolor; color was giving ``incompatible color definition'' warnings
\usepackage[normalem]{ulem}% allows \sout to strike-out, cross-out text.
\usepackage{amsthm}% allows theorem* , e.g.
\usepackage{bbold}% allows \unit \mathbb{1}
\usepackage{comment}% allows multiple lines to be commeted out
\usepackage[perpage]{footmisc} %resets footnotes per page
\usepackage{enumitem}% allows resuming enumerate
\usepackage{amsmath}% allows pmatrix
\usepackage{tikz}
\usepackage{tikz-cd}
\usetikzlibrary{arrows}

\usepackage{etoolbox} %this is needed for the hack removing address indentation

\usepackage[unicode]{hyperref} %added unicode; was getting warnings otherwise

\definecolor{dark-red}{rgb}{0.5,0.15,0.15}
\definecolor{dark-blue}{rgb}{0.15,0.15,0.6}
\definecolor{dark-green}{rgb}{0.15,0.6,0.15}

\hypersetup{
    colorlinks, linkcolor=Blue,
    citecolor=Blue, urlcolor=Blue
}

\usepackage{comment}

\usepackage[nameinlink,capitalise,noabbrev]{cleveref}
\usepackage{microtype}

\numberwithin{equation}{section}% makes equat numb contain the section
\usepackage[all]{xy}
\xyoption{line}
\newdir{ >}{{}*!/-10pt/\dir{>}}
\usepackage{graphicx}
\usepackage{mathtools}

\newtheorem{ThmAlpha}{Theorem}

\newtheorem{Thm}[equation]{Theorem}
\newtheorem*{Thm*}{Theorem}
\newtheorem*{MainThm*}{Main Theorem}
\newtheorem{Prop}[equation]{Proposition}
\newtheorem{Lem}[equation]{Lemma}
\newtheorem{Cor}[equation]{Corollary}

\newtheorem*{Que*}{Question}
\newtheorem*{Goal*}{Goal}

\theoremstyle{remark}
\newtheorem{Def}[equation]{Definition}

\newtheorem{Not}[equation]{Notation}
\newtheorem{Exa}[equation]{Example}

\newtheorem{Cons}[equation]{Construction}

\newtheorem{Rem}[equation]{Remark}
\newtheorem{War}[equation]{Warning}

\tikzset{
    labelrotatebelow/.style={anchor=north, rotate=90, inner sep=1.0mm}
}
\tikzset{
    labelrotateabove/.style={anchor=south, rotate=90, inner sep=1.0mm}
}

\newcommand{\nc}{\newcommand}
\nc{\dmo}{\DeclareMathOperator}

\renewcommand{\emptyset}{\varnothing}

\nc{\Beren}[1]{{\color{MidnightBlue}#1}}
\nc{\Drew}[1]{{\color{Orange}#1}}
\nc{\Tobi}[1]{{\color{Green}#1}}
\nc{\Natalia}[1]{{\color{Yellow}#1}}
\nc{\Dout}[1]{\Drew{\sout{#1}}}
\nc{\Bout}[1]{\Beren{\sout{#1}}}
\nc{\Tout}[1]{\Tobi{\sout{#1}}}
\nc{\Nout}[1]{\Natalia{\sout{#1}}}
\nc{\Greg}[1]{{\color{magenta}#1}}

\usepackage{todonotes}

\usepackage{pdflscape}

\nc{\overbar}[1]{\mkern 1.5mu\overline{\mkern-1.5mu#1\mkern-1.5mu}\mkern 1.5mu}

\usepackage[a]{esvect}
\nc{\weaklyfinite}{weakly closed}
\nc{\finite}{closed}
\nc{\BCdual}[1]{{#1}^*}
\nc{\LCore}{\mathrm{LCore}}
\nc{\Stovicek}{\v{S}\v{t}ov\'{i}\v{c}ek}
\nc{\ftriple}{f_{\natural}}
\nc{\unitC}{\unit_{\cat C}}%
\nc{\unitD}{\unit_{\cat D}}%
\nc{\Pone}{{\mathbb{P}^1}}
\nc{\InvSupp}[1]{\Supp_{\cat T}^{-1}(#1)}%this is pretty bad notation; beren
\nc{\InvCosupp}[1]{\Cosupp_{\cat T}^{-1}(#1)}
\nc{\closureP}{\overbar{\{\cat P\}}}
\nc{\closureQ}{\overbar{\{\cat Q\}}}
\nc{\singP}{\{\cat P\}}
\nc{\singQ}{\{\cat Q\}}
\nc{\singm}{\{\frak m\}}
\dmo{\Inj}{Inj}
\dmo{\tfib}{tfib}
\dmo{\tcof}{tcof}
\dmo{\Aut}{Aut}
\dmo{\tofib}{tofib}
\dmo{\surj}{surj}
\dmo{\Excs}{Exc}
\dmo{\Homog}{Homog}
\dmo{\PSh}{PSh}
\dmo{\Epis}{Epi}
\dmo{\Nat}{Nat}
\newcommand{\Exc}[1]{\Excs_{#1}}
\newcommand{\Epi}[1]{\Epis_{\le #1}}
\dmo{\KInjdmo}{K}
\dmo{\Dbdmo}{mod}
\dmo{\sur}{sur}
\nc{\KInj}[1]{\KInjdmo(\Inj #1)}
\nc{\Dbmod}[1]{\Der^b(\Dbdmo #1)}
\dmo{\Viss}{vis}%lol
\nc{\Vis}{\Viss}
\nc{\vis}{\Vis}
\nc{\kappaaux}{g}
\nc{\kappaCh}{{\kappaaux(\cat C_h)}}
\nc{\kappam}{{\kappaaux({\frak m})}}
\nc{\kappaP}{{\kappaaux(\cat P)}}
\nc{\kappaQ}{{\kappaaux(\cat Q)}}
\nc{\kappaCP}{{\kappaaux_{\cat C}(\cat P)}}
\nc{\kappaDP}{{\kappaaux_{\cat D}(\cat P)}}
\nc{\kappaCQ}{{\kappaaux_{\cat C}(\cat Q)}}
\nc{\kappaDQ}{{\kappaaux_{\cat D}(\cat Q)}}
\nc{\kappaphiB}{{\kappaaux(\phi(\cat B))}}
\nc{\kappaphiQ}{{\kappaaux(\varphi(\cat Q))}}
\nc{\halfplus}{{\scriptscriptstyle\top}}
\dmo{\Sub}{Sub}
\nc{\SpEn}{\cat S_{E(n)}}
\nc{\SpEnf}{\cat S_n}
\nc{\Lcomp}{L^{\mathrm{com}}} %I made this and the next one commands because I'm unsure of the choice of notation
\nc{\Ucomp}{U^{\mathrm{com}}}
\nc{\bbullet}{{\scriptscriptstyle\hspace{-1pt}\bullet}}
\nc{\bullett}{{\scriptscriptstyle\bullet}\hspace{-1pt}}
\nc{\LF}{L\hspace{-0.2ex}F}
\dmo{\StMod}{StMod}
\dmo{\Proj}{Proj}
\nc{\SpG}{\Sp^G}
\nc{\EG}{\bbE_G}
\nc{\DEG}{\Der(\EG)}
\dmo{\Perf}{Perf}
\nc{\DE}{\Der(\bbE)}
\nc{\Prst}{{\cat P}\mathrm{r^{st}}}
\nc{\Mack}[2]{\mathrm{Mack}_{#1}(#2)}
\nc{\SC}{S\cat C}
\dmo{\fin}{{fin}}
\dmo{\DM}{DM}
\dmo{\fp}{fp}
\nc{\DMQ}{\DM_Q}
\dmo{\DerKal}{DMack}
\dmo{\coh}{coh}
\dmo{\Der}{D}
\dmo{\DMot}{DMot}
\dmo{\rmH}{H}
\dmo{\piu}{\underline{\pi}}
\dmo{\Sphere}{\mathbb{S}}
\nc{\HA}{{\rmH \hspace{-0.2em}\bbA}}
\nc{\HZ}{{\rmH \hspace{-0.2em}\bbZ}}
\nc{\HZbar}{{\rmH \hspace{-0.2em}\underline{\bbZ}}}
\nc{\HbbF}{{\rmH \hspace{-0.15em}\mathbb{F}}}
\nc{\Fp}{{\bbF_{\hspace{-0.1em}p}}}
\nc{\HFp}{{\rmH \hspace{-0.15em}\bbF_{\hspace{-0.1em}p}}}
\nc{\HQ}{{\rm H \bbQ}}
\nc{\DHZG}{\Der(\HZ_G)}
\nc{\DHZH}{\Der(\HZ_H)}
\nc{\DHZK}{\Der(\HZ_K)}
\nc{\DHZGN}{\Der(\HZ_{G/N})}
\nc{\DHZGG}{\Der(\HZ_{G/G})}
\nc{\DHZCp}{\Der(\HZ_{C_p})}
\nc{\DHZGprime}{\Der(\HZ_{G'})}
\nc{\DHZ}{\Der(\HZ)}
\nc{\frakp}{\mathfrak{p}}
\nc{\frakq}{\mathfrak{q}}
\nc{\frakS}{\mathfrak{S}}
\nc{\frakT}{\mathfrak{T}}
\nc{\Z}{\mathbb{Z}}
\nc{\F}{\mathbb{F}}
\nc{\SSG}{\text{sSet}_*^G}
\nc{\sSet}{\text{sSet}}

\dmo{\csupp}{csupp}
\dmo{\Con}{Conj}
\dmo{\Id}{Id}
\dmo{\rmK}{\textrm{\rm K}}
\dmo{\Spc}{Spc}
\dmo{\thick}{thick}
\dmo{\thickid}{thickid}
\nc{\thicko}[1]{\thickid\langle #1 \rangle}
\nc{\thickt}[1]{\thick_\otimes\langle #1 \rangle}
\dmo{\cone}{cone}
\dmo{\End}{End}
\dmo{\Derperf}{D_{perf}}
\dmo{\Mor}{Mor}
\dmo{\id}{id}
\dmo{\incl}{incl}
\dmo{\Img}{Im}
\dmo{\im}{im}
\dmo{\Ker}{Ker}
\dmo{\ind}{ind}
\dmo{\CoInd}{coind}
\dmo{\GH}{GH}
\dmo{\idem}{e}
\dmo{\res}{res}
\dmo{\infl}{infl}
\dmo{\Derqc}{D_{qc}}
\nc{\DbcohX}{\Der^b(\coh X)}
\dmo{\triv}{triv}
\dmo{\Tel}{Tel} %telescope
\dmo{\grMod}{grMod}%
\dmo{\Mod}{Mod}%
\dmo{\opname}{op}
\dmo{\SH}{SH}% ground name for cat of spectra
\dmo{\smallb}{b}% ground exponent for ``bounded''
\dmo{\Spec}{Spec}
\dmo{\supp}{supp}
\dmo{\Supp}{Supp}
\dmo{\crosseffec}{cr}
\newcommand{\crosseffect}{{\crosseffec}}
\dmo{\thofib}{tofib}
\dmo{\hofib}{hofib}
\dmo{\cosupp}{cosupp}
\dmo{\Cosupp}{Cosupp}
\nc{\SHc}{{\SH^c}}
\nc{\SHp}{{\SH_{(p)}}}
\nc{\SHcp}{{\SH^c_{(p)}}}
\nc{\SHG}{\SH(G)}
\nc{\SHGp}{\SH(G)_{(p)}}
\nc{\SHGc}{\SHG^c}
\nc{\SHGcp}{\SHG^c_{(p)}}
\nc{\quadtext}[1]{\quad\textrm{#1}\quad}
\nc{\qquadtext}[1]{\qquad\textrm{#1}\qquad}
\nc{\adj}{\dashv}
\nc{\adjto}{\rightleftarrows}
\nc{\bbL}{\mathbb{L}}
\nc{\bbS}{\mathbb{S}}
\nc{\bbA}{\mathbb{A}}
\nc{\bbE}{\mathbb{E}}
\nc{\bbN}{\mathbb{N}}
\nc{\bbQ}{\mathbb{Q}}
\nc{\bbZ}{\mathbb{Z}}
\nc{\bbF}{\mathbb{F}}
\nc{\cat}[1]{\mathscr{#1}}%or: \nc{\cat}[1]{\mathcal{#1}}
\nc{\ie}{{\sl i.e.}, }
\nc{\into}{\mathop{\rightarrowtail}}
\nc{\inv}{^{-1}}
\nc{\isoto}{\mathop{\overset{\sim}\to}}
\nc{\isotoo}{\mathop{\overset{\sim}\too}}
\nc{\onto}{\mathop{\twoheadrightarrow}}
\nc{\too}{\mathop{\longrightarrow}\limits}
\nc{\mapstoo}{\longmapsto}
\nc{\adh}[1]{\overline{#1}}% adherence
\nc{\adhpt}[1]{\adh{\{#1\}}}% adherence of a pt
\nc{\aka}{{a.\,k.\,a.}\ }
\nc{\calF}{\mathcal{F}}
\nc{\eg}{{\sl e.\,g.}}
\nc{\hook}{\hookrightarrow}

\nc{\ideal}[1]{\langle #1\rangle}
\dmo{\red}{red}
\usepackage{makecell}
\dmo{\Hom}{Hom}
\nc{\Homcat}[1]{\Hom_{\cat #1}}
\nc{\iHom}{\mathcal{H}\mathrm{om}}
\nc{\ihom}[1]{\mathsf{hom}(#1)}
\nc{\ihomC}[1]{\mathsf{hom}_{\cat C}(#1)}
\nc{\ihomD}[1]{\mathsf{hom}_{\cat D}(#1)}
\nc{\ihomsub}[2]{\mathsf{hom}_{#1}(#2)}
\usepackage{longtable}
\nc{\Mid}{\,\big|\,}
\nc{\MMod}{\,\text{-}\Mod}%
\nc{\GrMMod}{\,\text{-}\grMod}%
\nc{\op}{^{\opname}}
\nc{\oto}[1]{\overset{#1}\to}
\nc{\otoo}[1]{\overset{#1}{\,\too\,}}
\nc{\sminus}{\!\smallsetminus\!}
\nc{\poplus}[1]{^{\oplus #1}}%
\nc{\potimes}[1]{^{\otimes #1}}% tensor power
\nc{\sbull}{{\scriptscriptstyle\bullet}}%\mathbf{\cdot}}%{}}
\nc{\SET}[2]{\big\{\,#1\Mid#2\,\big\}}
\nc{\SpcK}{\Spc(\cat K)}% most used
\nc{\then}{\Rightarrow}
\nc{\unit}{\mathbb{1}}% unit for \otimes
\nc{\xra}{\xrightarrow}
\nc{\phigeom}[1]{\widetilde{\Phi}^{#1}}
\dmo{\Oname}{O}
\dmo{\proper}{proper}% for proper subgroups
\dmo{\lenormal}{\unlhd}
\dmo{\fib}{fib}
\dmo{\cofib}{cofib}
\dmo{\lnormal}{\lhd}
\nc{\normal}{\trianglelefteq}%\lhd
\nc{\Op}{\Oname^p}% O^p for maximal p-normal subgroup
\nc{\Oq}{\Oname^q}% as above for p=q
\newcommand{\Sp}{{\mathscr{S}p}}

\dmo{\Ho}{Ho}
\dmo{\CB}{CB}
\dmo{\Fin}{Fin}
\dmo{\add}{add}
\dmo{\Fun}{Fun}
\dmo{\Ext}{Ext}
\dmo{\CAlg}{CAlg}
\dmo{\CMon}{CMon}
\dmo{\CC}{\cat C} %beren: I changed these, but left the O
\dmo{\DD}{\cat D}
\dmo{\OO}{\mathcal{O}}
\dmo{\Map}{Map}
\dmo{\Span}{Span}
\dmo{\Tot}{Tot}
\dmo{\N}{N}
\dmo{\Cat}{Cat}
\dmo{\colim}{colim}
\dmo{\hocolim}{hocolim}
\dmo{\Ch}{Ch}
\dmo{\A}{\mathbb{A}^{eff}}
\nc{\AGeff}{\mathbb{A}_G^{\mathrm{eff}}}
\nc{\BGeff}{\mathcal{B}_G^{\mathrm{eff}}}
\nc{\BG}{{\mathcal{B}_G}}
\nc{\NBGeff}{{\N}{\BGeff}}
\dmo{\Ab}{Ab}
\nc{\Smith}{\mathsf{Smith}}
\nc{\Floyd}{\mathsf{Floyd}}
\nc{\blue}{\beth^{\mathrm{geom}}}
\dmo{\Set}{Set}
\dmo{\ev}{ev}
\dmo{\Spcl}{Spcl}
\nc{\Funadd}{\Fun_{\add}}
\dmo{\proj}{proj}
\dmo{\cof}{cof}
\nc{\cPd}{\cat P_{\hspace{-.1em}d}}
\nc{\cPm}{\cat P_{\hspace{-.1em}m}}
\nc{\Chp}{\mathsf{Ch}_p}
\nc{\Pp}{\mathsf{P}_p}

\dmo{\Coideal}{Coideal}
\dmo{\gen}{gen}
\nc{\auxcoidealsymb}{\vartriangleleft}
\dmo{\Loc}{Loc}
\dmo{\Ind}{Ind}
\dmo{\Coloc}{Coloc}
\dmo{\Locideal}{Locid}
\dmo{\Colocideal}{Colocid}
\nc{\LOCO}{\Locideal}
\nc{\COLOCO}{\Colocideal}
\nc{\Loco}[1]{\LOCO\langle #1 \rangle}
\nc{\Coloco}[1]{\COLOCO\langle #1 \rangle}
\nc{\LambdaP}{\Lambda^{\cat P}} %beren: I've added this command here because we might need to make some spacing changes to make the typesetting less ugly
\nc{\LambdaQ}{\Lambda^{\cat Q}} %beren: I've added this command here because we might need to make some spacing changes to make the typesetting less ugly
\nc{\GammaP}{\Gamma_{\cat P}} %beren: I've added this command here because we might need to make some spacing changes to make the typesetting less ugly
\nc{\GammaQ}{\Gamma_{\cat Q}} %beren: I've added this command here because we might need to make some spacing changes to make the typesetting less ugly
\nc{\LambdaW}{\Lambda^{\hspace{-0.3ex}W}} %beren: I've added this command here because we might need to make some spacing changes to make the typesetting less ugly
\nc{\GammaW}{\Gamma_{\hspace{-0.3ex}W}} %beren: I've added this command here because we might need to make some spacing changes to make the typesetting less ugly
\nc{\gW}{g_W}
\nc{\gP}{g_{\cat P}}
\nc{\gQ}{g_{\cat Q}}
\nc{\cC}{{\cat C}}
\nc{\cT}{{\cat T}}
\nc{\cD}{{\cat D}}

\nc{\mT}{\kern-0.5em\mod\kern-0.1em\text{-}\cat{T}^c}
\nc{\mTc}{\kern-0.5em\mod\kern-0.1em\text{-}\cat{T}^c}
\nc{\MTc}{\Mod\kern-0.1em\text{-}\cat{T}^c}
\nc{\MT}{\Mod\kern-0.1em\text{-}\cat{T}}
\newcounter{enum-resume-hack}
\usepackage{wasysym}%
\usepackage{caption}
\Crefname{Thm}{Theorem}{Theorems}
\Crefname{Prop}{Proposition}{Propositions}

\usepackage{makecell}
\usepackage{tablefootnote}
\usepackage{arydshln}
\usepackage{booktabs}
\usepackage{quiver}
\usepackage[all]{xy}
\newdir{ >}{{}*!/-10pt/\dir{>}}

\makeatletter
\providecommand*{\twoheadrightarrowfill@}{%
  \arrowfill@\relbar\relbar\twoheadrightarrow
}
\providecommand*{\twoheadleftarrowfill@}{%
  \arrowfill@\twoheadleftarrow\relbar\relbar
}
\providecommand*{\xtwoheadrightarrow}[2][]{%
  \ext@arrow 0579\twoheadrightarrowfill@{#1}{#2}%
}
\providecommand*{\xtwoheadleftarrow}[2][]{%
  \ext@arrow 5097\twoheadleftarrowfill@{#1}{#2}%
}
\makeatother

\DeclareMathOperator{\type}{type}
\DeclareMathOperator{\height}{height}

\nc{\cL}{\mathcal{L}}
\nc{\cP}{\mathcal{P}}

\nc{\tblue}{\beth^{\mathrm{Tate}}}

\nc{\cA}{\mathcal{A}}
\nc{\cF}{\mathcal{F}}
\nc{\Fnt}{\cF_{\mathrm{nt}}}
\nc{\Fall}{\cF_{\mathrm{all}}}
\nc{\Ftriv}{\cF_{\mathrm{triv}}}

\nc{\bE}{\underline{E}}
\newcommand{\num}[1]{[#1]}
\DeclareRobustCommand{\stirling}{\genfrac\{\}{0pt}{}}
\usepackage{adjustbox}
\Crefname{Thm}{Theorem}{Theorems}
\Crefname{Prop}{Proposition}{Propositions}
\Crefname{Lem}{Lemma}{Lemmas}
\Crefname{Cor}{Corollary}{Corollaries}
\Crefname{Exa}{Example}{Examples}

\begin{document}

%\title{The spectrum of Goodwillie calculus}
%\title{The Balmer spectrum of Goodwillie calculus}
%\title{The Balmer spectrum of functor calculus}
%\title{The Balmer spectrum in Goodwillie calculus}
%\title{The Balmer spectrum in functor calculus}
%\title{The Balmer spectrum of excisive functors}
\title{The spectrum of excisive functors}
%\title{The tensor triangular geometry of functor categories}
%\title{The tensor triangular geometry of Goodwillie calculus}
%\title{The tensor triangular geometry of functor calculus}
%\title{Functor calculus via tensor triangular geometry}
%\title{The tensor triangular geometry of excisive functors}
%\title{The tt-geometry of excisive functors}

\author{Gregory Arone}
\author{Tobias Barthel}
\author{Drew Heard}
\author{Beren Sanders}
\date{\today}

\makeatletter
\patchcmd{\@setaddresses}{\indent}{\noindent}{}{}
\patchcmd{\@setaddresses}{\indent}{\noindent}{}{}
\patchcmd{\@setaddresses}{\indent}{\noindent}{}{}
\patchcmd{\@setaddresses}{\indent}{\noindent}{}{}
\makeatother

\address{Gregory Arone, Department of Mathematics, Stockholm University, SE-106 91 Stockholm, Sweden}
\email{gregory.arone@math.su.se}

\address{Tobias Barthel, Max Planck Institute for Mathematics, Vivatsgasse 7, 53111 Bonn, Germany}
\email{tbarthel@mpim-bonn.mpg.de}
\urladdr{https://sites.google.com/view/tobiasbarthel/home}

\address{Drew Heard, Department of Mathematical Sciences, Norwegian University of Science and Technology, Trondheim}
\email{drew.k.heard@ntnu.no}
\urladdr{https://folk.ntnu.no/drewkh/}

\address{Beren Sanders, Mathematics Department, UC Santa Cruz, 95064 CA, USA}
\email{beren@ucsc.edu}
\urladdr{http://people.ucsc.edu/$\sim$beren/}

\begin{abstract}
We prove a thick subcategory theorem for the category of \mbox{$d$-excisive} functors from finite spectra to spectra. This generalizes the Hopkins--Smith thick subcategory theorem (the $d=1$ case) and the \mbox{$C_2$-equivariant} thick subcategory theorem (the $d=2$ case). We obtain our classification theorem by completely computing the Balmer spectrum of compact $d$-excisive functors. A key ingredient is a non-abelian blueshift theorem for the generalized Tate construction associated to the family of non-transitive subgroups of products of symmetric groups. Also important are the techniques of tensor triangular geometry and striking analogies between functor calculus and equivariant homotopy theory. In particular, we introduce a functor calculus analogue of the Burnside ring and describe its Zariski spectrum \`{a} la Dress. The analogy with equivariant homotopy theory is strengthened further through two applications: We explain the effect of changing coefficients from spectra to $\HZ$-modules and we establish a functor calculus analogue of transchromatic Smith--Floyd theory as developed by Kuhn--Lloyd. Our work offers a new perspective on functor calculus which builds upon the previous approaches of Arone--Ching and Glasman.%
\end{abstract}

\vspace*{-1em}%
\maketitle
{
\hypersetup{linkcolor=black}
\tableofcontents
}

\newpage
\section{Introduction}

Goodwillie's calculus of functors \cite{Goodwillie90,Goodwillie91,Goodwillie03} provides  a powerful  tool  for studying functors between categories of homotopical origin. It approximates a homotopical functor $F$  via a tower of simpler functors
    \[
        F \to \ldots \to P_dF \to P_{d-1}F \to \ldots P_2F \to P_1F
    \]
analogous to the Taylor tower in ordinary calculus which approximates a smooth function $f\colon\mathbb{R}\to\mathbb{R}$ by polynomials of increasing degree $d$. In the calculus of functors, polynomials of degree~$d$ are replaced by  so-called \emph{$d$-excisive functors} which are homotopical functors satisfying a certain weak form of excision. The functor $P_d F$ is the best approximation of $F$ by a $d$-excisive functor.

A prominent domain for this theory is the study of (homotopical) functors  $\Sp \to \Sp$ from the category of spectra to itself. For technical reasons, we restrict attention to functors that preserve filtered colimits, which can be equivalently regarded  as  functors $\Sp^c \to \Sp$ from finite spectra to spectra. A fundamental problem  then  is to understand the structure  of the category $\Exc{d}(\Sp^c,\Sp)$ of $d$-excisive functors $\Sp^c \to \Sp$; in other words, to understand what ``polynomials'' are in this~context.

For example, a linear functor $\Sp^c\to\Sp$ has the form $X\mapsto C\otimes X$ where~$C\in \Sp$. In other words, a linear functor is essentially the same thing as a homology theory and we have an equivalence $\Exc{1}(\Sp^c,\Sp)\simeq \Sp$. In a similar vein, the functor $X\mapsto X^{\otimes d}$ is $d$-excisive, and one can show that $\Exc{d}(\Sp^c, \Sp)$ is the smallest category that contains each of the functors $X\mapsto X, X^{\otimes 2}, \ldots, X^{\otimes d}$ and is closed under colimits and finite limits~\cite[Corollary 6.1.4.15]{HALurie}. Nevertheless, the structure  of $\Exc{d}(\Sp^c,\Sp)$ is considerably more complicated for $d>1$. As in ordinary calculus, one can define the derivatives $\partial_1 F,\ldots,\partial_d F \in \Sp$  of a $d$-excisive functor $F\colon \Sp^c \to \Sp$ but, in contrast to ordinary polynomials, a $d$-excisive functor  is far from being determined by its derivatives when $d>1$. Previous attempts to understand the category $\Exc{d}(\Sp^c,\Sp)$ have focused on describing the extra structure possessed by the derivatives~\cite{AroneChing15,AroneChing16} or by the cross-effects~\cite{Glasman18pp} of a $d$-excisive functor.

Here we take a different approach to analyzing
$\Exc{d}(\Sp^c,\Sp)$ 
by regarding it as a tensor triangulated (``tt'') category under Day convolution. Our main result is a  computation of the \emph{Balmer spectrum} of the  subcategory of compact objects in $\Exc{d}(\Sp^c,\Sp)$ for any $d \ge 1$. This leads to a complete classification of  compact $d$-excisive functors \emph{up to tt-equivalence}. Two  compact functors $F$ and $G$ are \mbox{tt-equivalent} if they can be built from each other  using triangles, retracts and Day convolution with arbitrary compact functors; that is, $F$ and $G$ are tt-equivalent if they generate the same thick tensor-ideal.

We proceed to describe  our classification theorem. For simplicity, let us restrict attention to $p$-local spectra $\Sp_{(p)}$ for a given prime~$p$. The $\type(x) \in \bbN_{\infty} = \{0,1,2,\ldots\} \cup \{\infty\}$ of a $p$-local finite spectrum $x$ is the minimal $h$ with $K(p,h)_*(x) \neq 0$ if $x$ is nonzero and $\infty$ otherwise; here $K(p,h)$ is height $h$ Morava $K$-theory at $p$. This invariant extends to  functors $F\colon \Sp \to \Sp_{(p)}$ by setting
    \[
        \type(F) \colon \bbN \to \bbN_{\infty}, \quad l \mapsto \type(\partial_lF).
    \]
If $F$ is $d$-excisive, then we may restrict the domain to $\{1,\ldots, d\}$.  We now can state the $p$-local version of our main result (\cref{thm:tt-ideal-classification}):

\begin{MainThm*}\label{thmmain}
    The assignment $F \mapsto \type(F)$ induces a bijection
        \[
            \begin{tikzcd}[ampersand replacement=\&,column sep=small]
                {\begin{Bmatrix}
                    \text{tt-equivalence classes } \langle F \rangle:\\
                    F \in \Exc{d}(\Sp^c,\Sp_{(p)}) \text{ compact}
                \end{Bmatrix}}
                    \& 
                {\begin{Bmatrix}
                    \text{functions } f\colon \{1,\ldots,d\} \to \bbN_{\infty} \text{ such that }  \\
                    f(k) \leq \delta_p(k,l) + f(l) \text{ if } p-1\! \mid\! k-l \geq 0
                \end{Bmatrix}}
                    \arrow["\sim", from=1-1, to=1-2]
            \end{tikzcd}
        \]
    where $\delta_p(k,l)$ is an explicit function depending on $l$ and the  base $p$ expansion of $k$.
\end{MainThm*}

To put this result into perspective, let us consider two special cases. If $d=1$, then $\Exc{1}(\Sp^c,\Sp) \simeq \Sp$ and the above bijection is the content of the thick subcategory theorem of Devinatz, Hopkins and Smith \cite{DevinatzHopkinsSmith88,HopkinsSmith98} which we however take as an input to our theorem.  When $d=2$, there is an equivalence  $\Exc{2}(\Sp^c,\Sp)\simeq \Sp_{C_2}$ with the category of genuine $C_2$-equivariant spectra. In this case, our theorem recovers the classification of  compact $C_2$-spectra due to Balmer and Sanders \cite{BalmerSanders17}. Generalizing from cyclic groups of order $p$ to arbitrary finite abelian groups $G$, the analogous classification problem for  compact $G$-spectra has subsequently been settled by Barthel, Hausmann, Naumann, Nikolaus, Noel and Stapleton in \cite{BHNNNS19}, but  the general case of a non-abelian finite group $G$ remains elusive.

The theory we develop here in our study of $d$-excisive functors will demonstrate that the categories $\Exc{d}(\Sp^c,\Sp)$ behave a great deal like categories of $\Sigma_d$-equivariant spectra indexed on the family of non-transitive subgroups of $\Sigma_d$. In light of the ambiguity left in the tt-classification of  compact $G$-spectra for general $G$, we view the complete classification furnished by our main theorem as somewhat surprising. It requires precise control over the chromatic behaviour of generalized Tate constructions with respect to families of non-transitive subgroups of products of symmetric groups, i.e., non-abelian analogs of the blueshift theorems of %
\cite{Kuhn04} and~\cite{BHNNNS19}.

\subsection*{Further details} We will now say more about the techniques that go into proving the main theorem. As mentioned above, our starting point is the observation that $\Exc{d}(\Sp^c,\Sp)$ is a tensor triangulated category via (localized) Day convolution, and our first task is to show that it has excellent structural properties, making it amenable to the ample toolkit of tt-geometry. The following is a summary of results obtained in \Cref{prop:generators-dualizable-self-dual,prop:compact_generation,cor:rigid-compact}.

\begin{ThmAlpha}\label{thmx:ttcat}
    The category of $d$-excisive functors $\Exc{d}(\Sp^c,\Sp)$ is a compactly generated presentably symmetric monoidal stable $\infty$-category whose compact and dualizable objects coincide. A set of compact generators is given by the functors
    $
    P_dh_{\bbS}(i)%, \quad X \mapsto P_d\left(\left(\Sigma^{\infty}\Map_{\Sp}(\mathbb{S},X) \right)^{\otimes i}\right)
    $
    for $1 \le i \le d$, where $h_{\bbS}(i)(X) = (\Sigma^\infty\Omega^\infty X)^{\otimes i}$. % and $P_d$ is Goodwillie's $d$-excisive approximation.
    Moreover, these generators are self-dual separable commutative algebras.
\end{ThmAlpha}

We deduce most of this result from global statements about $\Fun(\Sp^c,\Sp)$ by applying the (smashing) localization functor $P_d\colon \Fun(\Sp^c,\Sp)\to\Exc{d}(\Sp^c,\Sp)$. This leads in particular to a strong compatibility among the given sets of compact generators for varying $d$. Indeed,  we observe that for $1 \le i \le d-1$ the generators $P_dh_{\bbS}(i)$ of $\Exc{d}(\Sp^c,\Sp)$ are mapped by $P_{d-1}$ to the generators $P_{d-1}h_{\bbS}(i)$ of $\Exc{d-1}(\Sp^c,\Sp)$. One is then led to wonder if $P_{d-1}$ is the \emph{finite} localization that kills~$P_dh_{\bbS}(d)$. We verify in \Cref{thm:finite-localization-p-n} that this is indeed the case: 

\begin{ThmAlpha}\label{thmx:openpiece}
    The induced functor $P_{d-1}\colon \Exc{d}(\Sp^c,\Sp) \to \Exc{d-1}(\Sp^c,\Sp)$ is a finite localization whose ideal of acyclics is generated by $P_dh_{\bbS}(d)$. In particular, the category $\Homog_{d}(\Sp^c,\Sp) \subseteq \Exc{d}(\Sp^c,\Sp)$ of $d$-homogeneous functors is the localizing tensor-ideal generated by $P_dh_{\bbS}(d)$.
\end{ThmAlpha}

The combination of \cref{thmx:ttcat} and \cref{thmx:openpiece} puts us in the position to study $\Exc{d}(\Sp^c,\Sp)$ via tt-geometry and, in particular, to reason inductively over $d$. Recall that tt-geometry as initiated by Balmer \cite{Balmer05a,BalmerICM} views a rigidly-compactly generated tt-category $\cat T$ as a sheaf of local tt-categories over a naturally associated spectral space $\Spc(\cat T^c)$, the Balmer spectrum of the full subcategory $\cat T^c$  of compact objects in $\cat T$. The spectrum (whose points are the prime ideals of $\cat T^c$) plays a dual role. On the one hand, it reveals the geometric structure of $\cat T$
and brings a geometric set of methods to its analysis.
On the other hand, its topology encodes the classification of thick tensor-ideals of $\cat T^c$. This leads us to:

\begin{Goal*}
    Compute the spectrum of prime ideals $\Spc(\Exc{d}(\Sp^c,\Sp)^c)$ for all $d \geq 1$.
\end{Goal*}

The first step is to determine its underlying set. We show that the Goodwillie derivatives $\partial_i \colon \Exc{d}(\Sp^c,\Sp) \to \Sp$ for $1 \le i \le d$ form a collection of symmetric monoidal and jointly conservative tt-functors (\Cref{lem:geometric-derivatives} and \Cref{lem:conservativity}). Therefore, one can find prime ideals of $\Exc{d}(\Sp^c,\Sp)^c$ by pulling back prime ideals of $\Sp^c$ along $\partial_i$. The prime ideals of $\Sp^c$ are well known by the tt-geometric form of the thick subcategory theorem \cite{HopkinsSmith98} as recalled in \Cref{fig:balmer-sp-of-sp}. In this way we obtain in  \cref{thm:spec-as-a-set} a description of $\Spc(\Exc{d}(\Sp^c,\Sp)^c)$ as a set:

\begin{ThmAlpha}\label{thmx:spcset}
    Every prime ideal in $\Spc(\Exc{d}(\Sp^c,\Sp)^c)$ is of the form  
        \[
            \cPd(\num{i},p,h) = \SET{F \in \Exc{d}(\Sp^c,\Sp)^c}{K(p,h-1)_*\partial_i(F) = 0}
        \]
    for some triple $(i,p,h)$ consisting of an integer $1 \le i \le d$, a prime number $p$ or $p = 0$, and a chromatic height $1 \le h \le \infty$. Moreover, we have $\cPd(\num{i},p,h) = \cPd(\num{j},p',h')$ if and only if $\num{i} = \num{j}$ and $h = h'$ and if $h = h' >1$, then also $p = p'$. 
\end{ThmAlpha}

It then remains to determine the topology of the space $\Spc(\Exc{d}(\Sp^c,\Sp)^c)$, which is a significantly more difficult problem. It turns out that it is enough to determine all the inclusions $\cat P_d(\num{i},p,h) \subseteq \cat P_d(\num{j},p',h')$ among the primes  (\Cref{prop:spc_posetreduction}). One general technique to get information about these inclusions is the \emph{comparison map} introduced by Balmer \cite{Balmer10b}. This is an inclusion-reversing continuous map
    \begin{equation}\label{eq:intro_comparisonmap}
        \rho \colon \Spc(\cat K) \to \Spec(\End_{\cat K}(\unit))
    \end{equation}
from the Balmer spectrum of any tt-category $\cat K$ to the Zariski spectrum of the endomorphism ring of the unit in $\cat K$. In order to apply this in our setting,  we introduce in \cref{sec:goodwillie-burnside} a commutative ring $A(d)$ for each $d \ge 1$ that we call the Goodwillie--Burnside ring. It is a combinatorially defined ring constructed from the category of finite sets of cardinality at most $d$ and surjections. More specifically, as an abelian group, it is free on $d$ generators $x_1,\ldots,x_d$ and its multiplication is determined by the formula
    \[
        x_ix_j = \sum_{1 \le l \le d}\mu(i,j,l)x_l,
    \]
where $\mu(i,j,l)$ counts the number of \emph{good subsets} (\Cref{def:good-subsets}) of $\num{i} \times \num{j}$ of cardinality $l$, where $\num{i}$ is the set with $i$ elements. Its relation to the category of $d$-excisive functors is expressed by the following result (\Cref{thm:burnside-ring-isomorphism}): 

\begin{ThmAlpha}\label{thmx:goodwillieburnsidering}
    The endomorphism ring $\pi_0\End(P_dh_{\bbS})$ of the unit in $\Exc{d}(\Sp^c,\Sp)$ is isomorphic to the Goodwillie--Burnside ring $A(d)$ introduced above. 
\end{ThmAlpha}

In other words, \cref{thmx:goodwillieburnsidering} supplies a functor calculus analogue of the computation of the Burnside ring in equivariant homotopy theory  due to Segal \cite{Segal1971} and tom Dieck \cite{tomDieck1979}.  We are then able to give  (in \cref{thm:burnside-goodwillie-spectrum} and \cref{prop:comparison-map}) a complete description of the Zariski spectrum of $A(d)$ and the comparison map
\begin{equation}
\rho\colon \Spc(\Exc{d}(\Sp^c,\Sp)^c) \to
\Spec(A(d)).
\end{equation} 
This severely restricts which inclusions can possibly occur, and altogether, we show that to determine the topology, it suffices to compute the minimal $\beth \ge 0$, the so-called \emph{geometric blueshift}, such that
    \begin{equation}\label{eq:intro_inclusions}
	\cPd(\num{k},p,h+\beth) \subseteq \cPd(\num{l},p,h)
    \end{equation}
whenever $p-1 \mid k-l \ge 0$ and for any $1 \leq h \leq \infty$. 

To facilitate this computation, we proceed inductively on the ambient degree $d$ with the base $d=1$ case reducing to the classical thick subcategory theorem. At the level of spectra, \cref{thmx:openpiece} induces an open embedding
        \begin{equation}\label{eq:open-embedding}
            \Spc(P_{d-1}) \colon \Spc(\Exc{d-1}(\Sp^c,\Sp)^c) \hookrightarrow \Spc(\Exc{d}(\Sp^c,\Sp)^c)
        \end{equation}
which in particular implies that it suffices to consider the case $k = d$ in \eqref{eq:intro_inclusions}. The values of the corresponding geometric blueshift numbers are stated in the following theorem, which forms the heart of the description of the topology of $\Spc(\Exc{d}(\Sp^c,\Sp)^c)$.

\begin{ThmAlpha}\label{thmx:spctop}
    Let $p,q$ be prime numbers, $1\le k,l\le d$ integers, and suppose $1\leq h,h' \leq \infty$. There is an inclusion $\cPd(\num{k},p,h') \subseteq \cPd(\num{l},q,h)$ if and only if the following three conditions hold:
	\begin{equation*}
		\begin{split}
			\text{(a)}&\;\; p-1 \mid k-l \ge 0;\\
			\text{(b)}&\;\; h' \ge h +\delta_p(k,l);\\
			\text{(c)}&\;\; \text{if $h > 1$, then $p=q$,}\\
		\end{split}
		\text{ where }
		\begin{split}
	 \delta_p(k,l) \coloneqq \begin{cases}
			0 &  \text{if } k=l;\\
			1 &  \text{if } p-1\mid k-l > 0 \text{ and } l \ge s_p(k);\\
			2 & \text{if } p-1\mid k-l > 0 \text{ and } l < s_p(k).
            %; \\ \infty & \text{otherwise}.
		\end{cases}\end{split}
\end{equation*}	
    Here $s_p(k)$ denotes the sum of the coefficients of the base $p$ expansion of $k$.
\end{ThmAlpha}

In particular, this result shows that, somewhat surprisingly, the geometric blueshift~$\beth$ is \emph{never} greater than 2. Its proof occupies the bulk of \cref{part:III} and is summarized in the following subsection. It completes our computation of the Balmer spectrum of compact $d$-excisive functor for all $d \ge 1$. The $d=3$ case is depicted in \cref{fig:exc3} below.

%    \begin{figure}[h!]
    \begin{figure}[b]
            \includegraphics{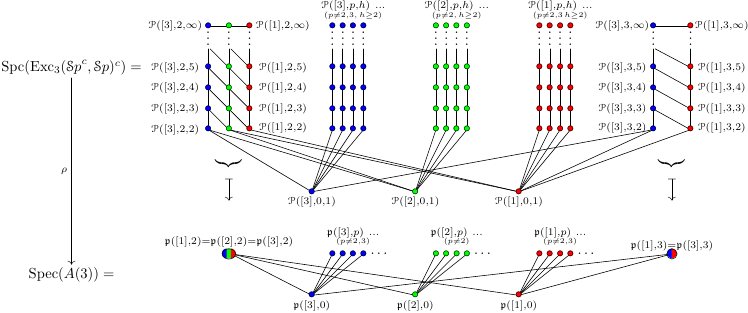}\caption{The Balmer spectrum of 3-excisive functors, along with the comparison map to the Zariski spectrum of $A(3)$.}\label{fig:exc3}
    \end{figure}

\subsection*{Outline of the proof of \cref{thmx:spctop}}

The tt-geometric classification of compact $d$-excisive functors stated at the beginning of the introduction is equivalent to \cref{thmx:spctop}; the translation between the two formulations follows from general tt-geometric principles. In particular, the tt-classification theorem is extracted from \cref{thmx:spctop} in \cref{thm:tt-ideal-classification}; see also \cref{cor:plocal-tt-ideal-classification}.

The proof of \Cref{thmx:spctop} takes a number of steps, whose overall strategy is modelled on the work of Balmer and Sanders \cite{BalmerSanders17} in equivariant homotopy theory. First of all, the condition that $p - 1 \mid k-l$ already appears in the Zariski spectrum of the Goodwillie--Burnside ring while the requirement that $k-l \ge 0$ follows from a straightforward computation;  thus it remains to determine the values of $\delta_p(k,l)$, as explained above. This geometric blueshift function measures the chromatic shift that  occurs in inclusions among prime ideals, and it is essentially controlled by two interacting pieces of structure, to be explained in more detail below:
    \begin{itemize}
        \item the chromatic blueshift behaviour of the Tate-derivatives; and
        \item the combinatorics of $p$-power partitions. 
    \end{itemize}
In order to explain our strategy, recall from \eqref{eq:open-embedding} that $P_{d-1}$ induces an open embedding on spectra. Reinterpreted tt-geometrically, this results in a categorical open-closed decomposition of $\Exc{d}(\Sp^c,\Sp)$, which on objects decategorifies to the Kuhn--McCarthy pullback square \cite{McCarthyRy2001Dual,Kuhn04} for $d$-excisive functors $F$:
    \begin{equation}\label{eq:intro_square}
        \begin{tikzcd}[ampersand replacement=\&]
        {P_dF(-)} \& {(\partial_dF \otimes (-)^{\otimes d})^{h\Sigma_d}} \\
        {P_{d-1}F(-)} \& {(\partial_dF \otimes (-)^{\otimes d})^{t\Sigma_d}.}
        \arrow[from=1-1, to=2-1]
        \arrow[from=1-1, to=1-2]
        \arrow[from=2-1, to=2-2]
        \arrow[from=1-2, to=2-2]
        \end{tikzcd}
    \end{equation}
Here, the Tate construction appearing in the lower right corner governs the gluing data between $(d-1)$-excisive and $d$-homogeneous functors needed to reconstruct a \mbox{$d$-excisive} functor. The idea is then to argue inductively on $d$, where the induction step requires us to glue $\Spc(\Sp^c)$ and $\Spc(\Exc{d-1}(\Sp^c,\Sp)^c)$ via the Tate construction.

In order to understand the gluing, we first reduce to the $p$-local situation for a given $p$ and then consider the ($p$-local) \emph{Tate-derivatives} defined as 
\[
    \partial_lt_di_d\colon \Sp_{(p)} \to \Sp_{(p)}, \quad A \mapsto \partial_l\left(X \mapsto (A \otimes X^{\otimes d})^{t\Sigma_d}\right)
\]
for any $l<d$. These functors exhibit a chromatic height-shifting behaviour---a phenomenon generally known as \emph{blueshift}---which determines what we call \emph{elementary} inclusions among prime ideals.

\begin{ThmAlpha}\label{thmx:blueshift}
    The $p$-local Tate-derivatives $\partial_lt_di_d$ are contractible unless $d>l$ and there exists a partition of $d$ by powers of $p$ of length $l$. If such a partition exists, then the $p$-local Tate-derivatives exhibit a blueshift by 1, i.e., the kernel of the functor 
        \[
			\begin{tikzcd}
            \partial_l\left ((L_{p,h}^f\bbS \otimes (-)^{\otimes d})^{t\Sigma_d}\right ) \colon \Sp_{(p)}^{c} \ar[r] & \Sp
			\end{tikzcd}
        \]
    consists precisely of the type $h$ finite $p$-local spectra. 
\end{ThmAlpha}

The proof of \cref{thmx:blueshift} proceeds by using work of Arone--Ching \cite{AroneChing15} to translate the problem to stable equivariant homotopy theory. More precisely, we express the Tate-derivatives in terms of generalized Tate constructions with respect to families $\Fnt$ of non-transitive subgroups of symmetric groups: 
    \[
        \partial_l\left ((L_{p,h}^f\bbS \otimes (-)^{\otimes d})^{t\Sigma_d}\right ) \simeq (L_{p,h}^f\bbS)^{t\Fnt(\lambda)} \coloneqq \Phi^{\Fnt(\lambda)}(\underline{\infl_e^{\Sigma_\lambda}L_{p,h}^f\bbS}).
    \]
We are then able to determine the corresponding blueshift behaviour using the main theorem of previous work \cite{BHNNNS19} on the topology of Balmer spectrum of finite abelian groups, see \cref{thm:tateblueshift}, which in turn extended Kuhn's blueshift theorem~\cite{Kuhn04}. Since it involves generalized Tate constructions based on symmetric groups, we may view \cref{thmx:blueshift} as an instance of \emph{non-abelian blueshift}. We note that, in general, determining the blueshift behaviour for arbitrary families of subgroups of symmetric groups would suffice to compute the Balmer spectrum of compact \mbox{$G$-spectra} for all finite groups $G$, a problem that remains open for every~$\Sigma_d$ with~$d \geq 8$.

Returning to the problem of computing the blueshift function $\delta_p(k,l)$, \cref{thmx:blueshift} accounts only for part of the possible relations; indeed, in general there can (and will!) be chains of inclusions
    \[
        \cPd(\num{k},p,h) = \cPd(\num{k_1},p,h_1) \subseteq \cPd(\num{k_2},p,h_2) \subseteq \ldots \subseteq \cPd(\num{k_j},p,h_j) = \cPd(\num{l},p,h')
    \]
in which only the adjacent inclusions are captured by the Tate-derivatives. One outcome of our analysis is that all inclusions among prime ideals are formed by such chains, so it remains to find the minimal shift among all such chains, which is precisely the value of $\delta_p(k,l)$. This is then translated into a combinatorial problem on $p$-power partitions, which we are able to solve using some elementary number theory (\cref{prop:ppowerchains}).

\subsection*{Parerga and paralipomena}

We conclude our summary of the paper by mentioning some applications and further directions.

\subsubsection*{Change of categories}
In this paper we focus on functors from the category of spectra to itself. But the methods of functor calculus apply much more broadly. Indeed, one can apply functor calculus to study the $\infty$-category of functors $\Fun(\cat C, \cat D)$ where $\cat C$ and $\cat D$ are general $\infty$-categories, perhaps satisfying some mild technical assumptions. It is natural to wonder how our methods work for other categories. In order to use tt-geometric techniques, we should require our categories to be symmetric monoidal and stable; in addition, the theory is particularly effective and well-developed in the rigidly-compactly generated case.

It is relatively straightforward to extend our methods from $\Exc{d}(\Sp^c, \Sp)$ to the study of $\Exc{d}(\Sp^c, \cat D)$, where $\cat D$ is a  rigidly-compactly generated symmetric monoidal stable $\infty$-category. Indeed, we show in \Cref{prop:setspc_coefficients} that, as a set, the primes are always obtained by pulling back primes along the derivatives  as in the case $\cat D=\Sp$. In contrast, it is likely that new ideas are needed to replace $\Sp^c$ in the source with a more general category. In particular \Cref{lem:partial-splits-inflation} and \Cref{lem:injectivity-of-partial-k} seem to depend on the source being $\Sp^c$. 

As a proof of concept, when the target is $\Mod_{\HZ}$, we  completely determine the topology of $\Spc(\Exc{d}(\Sp^c, \Mod_{\HZ})^c)$ in \Cref{thm:spc_integral}, a computation analogous to the work of Patchkoria--Sanders--Wimmer in equivariant homotopy theory \cite{PatchkoriaSandersWimmer22}. In this case, all primes are of the form $\cPd^{\Z}(\num{k},\mathfrak p)$ for $1 \le k \le d$ and $\mathfrak p \in \Spec(\bbZ)$ and the topology is given by:

\begin{ThmAlpha}\label{thmx:integralcoefficients}
    Let $d \geq k,l \geq 1$ be integers and consider two prime ideals $\mathfrak p, \mathfrak q \in \Spec(\Z)$. Then there is an inclusion $\cPd^{\Z}(\num{k},\mathfrak p) \subseteq \cPd^{\Z}(\num{l},\mathfrak q)$ if and only if one of the following two conditions is satisfied:
        \begin{enumerate}
            \item $\mathfrak p = (p)$ for some prime $p$, $\mathfrak q = (p)$ or $\mathfrak q = (0)$, and $p-1 \mid k-l \geq 0$;
            \item $\mathfrak p = (0) = \mathfrak q$ and $k=l$.
        \end{enumerate}
    Moreover, $\Spc(\Exc{d}(\Sp^c,\Mod_{\HZ})^c)$ is noetherian, so that the topology is determined by these inclusions. Base-change $\Sp\to\Mod_{\HZ}$ induces a map
    \[
        \Spc(\Exc{d}(\Sp^c,\Mod_{\HZ})^c) \to \Spc(\Exc{d}(\Sp^c,\Sp)^c)
    \] 
    which is a homeomorphism onto its image. It maps $\cPd^{\Z}(\num{k},(p))$ to $\cPd(\num{k},p,\infty)$ and maps $\cPd^{\Z}(\num{k},(0))$ to $\cPd(\num{k},0,1)$. Finally, $\Exc{d}(\Sp^c,\Mod_{\HZ})$ is stratified and costratified over its spectrum in the sense of \cite{bhs1} and satisfies the telescope conjecture.
\end{ThmAlpha}
We refer the reader to \Cref{fig:zarizki_a3z} (in \Cref{sec:applications}) for an illustration of the case $d = 3$. More generally, it is possible to deduce an abstract description of the resulting spectrum of the category $\Spc(\Exc{d}(\Sp^c,\cat D)^c)$ for any rigidly-compactly generated symmetric monoidal stable $\infty$-category $\cat D$ from \cref{thmx:spctop}, but we will not pursue this here.

Another important example of change of coefficients is that of functors from spectra to telescopically localized spectra, studied in depth by Kuhn in \cite{Kuhn04}. Using the aforementioned telescopic blueshift theorem proven in the same paper, Kuhn deduced from \eqref{eq:intro_square} that the Taylor tower of any $F \in \Exc{d}(\Sp^c,\Sp)$ splits completely after telescopic localization. From this perspective, our main result is a far-reaching delocalization of Kuhn's splitting theorem.

\subsubsection*{Smith--Floyd theory in functor calculus}

The geometric blueshift occurring between inclusions of prime ideals in \eqref{eq:intro_inclusions}, as determined in \cref{thmx:spctop}, is equivalent to the implication:
    \begin{equation}\label{eq:intro_smithfloyd}
        K(p,h-1+\beth)_*\partial_kF = 0 \implies K(p,h-1)_*\partial_lF = 0 
    \end{equation}
whenever $F \in \Exc{d}(\Sp^c,\Sp)^c$. Instead of a conditional vanishing result, it is desirable to have quantitative control over the transchromatic relation between different derivatives.

In equivariant homotopy theory, the analogous problem is to find relations between the dimensions of Morava $K$-theories of geometric fixed points $\Phi^H(x),\Phi^K(x)$ of a given finite $G$-spectrum $x$, for various subgroups $H,K$ of $G$. On the one hand, since the geometric $H$-fixed points of a $G$-suspension spectrum of a $G$-space $X$ coincide with the $G$-suspension spectrum of the $H$-fixed points of $X$, at height~$\infty$ this is the topic of classical Smith \cite{Smith41} and Floyd \cite{Floyd52} theory for mod $p$ homology. On the other hand, non-equivariantly, Ravenel \cite{Ravenel84} showed that the dimension of $K(h)_*(x)$ provides an upper bound for the dimension of $K(n)_*(x)$ for all $n\leq h$ and all finite spectra $x$; this is a transchromatic version of Floyd theory for the trivial action. Simultaneously generalizing these theorems, Hausmann and Kuhn--Lloyd~\cite{KuhnLloyd2024} reinterpreted the results of \cite{BalmerSanders17,BHNNNS19} as transchromatic Smith--Floyd theory.

In light of the strong analogy between excisive functors $\Sp^c \to \Sp$ and stable equivariant homotopy theory, we can view the sought-after quantitative version of~\eqref{eq:intro_smithfloyd} as a functor calculus analogue of transchromatic Smith--Floyd theory. As one application of our main theorem, we show that this story indeed carries over to the calculus context, see \cref{cor:floydinequality}:

\begin{ThmAlpha}\label{thmx:smith}
    For any compact $F \in \Exc{d}(\Sp^c,\Sp)^c$, the following inequality holds:
        \begin{equation}\label{eq:intro_goodwilliefloyd}
            \dim_{K(p,n)_*}K(p,n)_*(\partial_kF) \geq \dim_{K(p,h)_*}K(p,h)_*(\partial_lF)
        \end{equation}
    whenever $p-1 \mid k-l \geq 0$ and $n \geq h +\delta_p(k,l)$.
\end{ThmAlpha}

We refer to \cref{prop:smithfloyd_equivalence} for the precise equivalence between transchromatic Smith theory \eqref{eq:intro_smithfloyd} and transchromatic Floyd theory \eqref{eq:intro_goodwilliefloyd} in $\Exc{d}(\Sp^c,\Sp)^c$.

\subsubsection*{Spectral Mackey functors}

As witnessed repeatedly above, the abstract tt-geometry of $\Exc{d}(\Sp^c,\Sp)$ closely resembles that of $\Sp_G$ for $G$ a finite group. The underlying reason for why the two stories are parallel is given in the setting of spectral Mackey functors on epiorbital categories developed by Barwick and Glasman \cite{Barwick17,Glasman15pp}. Specifically, $G$-spectra can be modelled as spectral Mackey functors on the category of finite $G$-sets \cite{Barwick17}, while $d$-excisive functors can be modelled as spectral Mackey functors on the free coproduct completion of the category of finite sets of cardinality at most $d$ and surjections \cite{Glasman18pp}. In fact, when $d = 2$, this induces the aforementioned equivalence between $2$-excisive functors and $C_2$-spectra, as the indexing categories are equivalent; however, the stories bifurcate as soon as $d > 2$.
For the benefit of the reader conversant in equivariant homotopy theory, we will often indicate when a construction in Goodwillie calculus is the analog of a well-known construction in equivariant homotopy theory; for example, the Goodwillie derivatives correspond to geometric fixed points. We have also tabulated the most prominent instances of this correspondence in \Cref{appendix}. 

\subsubsection*{Spaces vs spectra}
One aspect of this work that we find intriguing is that it is particularly adapted to the study of functors from spectra to spectra --- and, more generally, to functors between rigidly-compactly generated symmetric monoidal stable $\infty$-categories --- but it does not readily apply to functors from {\it spaces} to spectra, because 
the category of such functors is not rigidly-compactly generated. This is in contrast to previous approaches to classification of excisive functors, where functors from spaces to spectra tended to be easier to understand than functors from spectra to spectra.

\subsection*{Frequently used notations}\label{ssec:notation}
Here is some notation that will be used throughout:
\begin{itemize}
\item For any $\infty$-category $\cat C$, $\cat C^c$ denotes the category of compact objects in $\cat C$.
\item Coproducts and smash products of spectra are denoted by $\oplus$ and $\otimes$ respectively. The smash product of pointed \emph{spaces} is denoted by $\wedge$. Day convolution is denoted by $\circledast$. Internal homs are denoted $\ihom{-,-}$.
\item In a tensor triangulated category $\cat T$, the localizing ideal (resp.~localizing subcategory) generated by a subcategory $\cat E \subset \cat T$ is denoted $\Locideal\langle\cat E\rangle$ (resp.~$\Loc\langle\cat E\rangle$) while the thick ideal generated by $\cat E$ is denoted $\thickid\langle \cat E\rangle$.
\item If $\cat C$ and $\cat D$ are presentably symmetric monoidal stable $\infty$-categories, then a \emph{geometric functor} $F\colon \cat C \to \cat D$ is a symmetric monoidal functor which preserves all colimits.
A \emph{geometric equivalence} is a geometric functor which is an equivalence.
    \item For an integer $n\ge 0$, the standard set with $n$ elements will be denoted by $\num{n}\coloneqq \{1, \ldots, n\}$
    \item $\surj(m, n)$ will denote the set of surjective functions from $\num{m}$ to $\num{n}$. Accordingly, $|\surj(m, n)|$ denotes the number of surjections from $\num{m}$ to $\num{n}$.
    \item We write $\bbN_{\infty} = \{0,1,2,\ldots\} \cup \{\infty\}$.
\end{itemize}

\subsection*{Acknowledgements}\label{ssec:thanks}

We thank Torgeir Aamb\o, Kaif Hilman, Nick Kuhn, Marius Verner Bach Nielsen, Irakli Patchkoria, and Maxime Ramzi for helpful discussions. We are also grateful to Achim Krause for helpful discussions and in particular for pointing out \cref{lem:summationformula}. Moreover, we would like to thank Niall Taggart and an anonymous referee for useful comments on an earlier version of this paper.

TB is supported by the European Research Council (ERC) under Horizon Europe (grant No.~101042990) and would like to thank the Max Planck Institute for its hospitality. DH is supported by grant number TMS2020TMT02 from the Trond Mohn Foundation. 
The authors would also like to thank the Hausdorff Research Institute for Mathematics for its hospitality and support during the trimester program `Spectral Methods in Algebra, Geometry, and Topology', funded by the Deutsche Forschungsgemeinschaft under Germany's Excellence Strategy – EXC-2047/1 – 390685813.

\clearpage
\part{Categorical foundations}

\section{The tensor triangulated category of \texorpdfstring{$d$}{d}-excisive functors}

We begin with  a treatment of Goodwillie calculus for functors from spectra to spectra and introduce the tensor triangulated category of $d$-excisive functors. Most, if not all, of the definitions are due to Goodwillie \cite{Goodwillie03}. For an $\infty$-categorical treatment of Goodwillie's work we refer the reader to \cite[Chapter 6]{HALurie}. A survey of Goodwillie calculus can be found in \cite{AroneChing20}.

\subsection*{\texorpdfstring{$d$}{d}-cubes and \texorpdfstring{$d$}{d}-excisive functors}

\begin{Def}
    Let $\cat C$ be an  $\infty$-category with finite limits and finite colimits. The category of $d$-cubes in  $\cat C$ is the functor category $\Fun(\cat P(\num{d}),\cat C)$ where $\cat P(\num{d})$ denotes the poset of subsets of a finite set of cardinality $d$.  A $d$-cube $\cat X$ is said to be
    \begin{enumerate}[label=(\roman*)]
    \item \emph{cartesian}, if the canonical map 
    \[
        \cat X(\varnothing) \to \lim_{\varnothing \ne S \subseteq \num{d}} \cat X(S)
    \]
    is an equivalence, and
    \item \emph{cocartesian}, if the canonical map 
    \[
        \colim_{S \subsetneq \num{d}} \cat X(S) \to \cat X(\num{d})
    \]
    is an equivalence.
    \end{enumerate}
    Finally, a $d$-cube $\cat X$ is \emph{strongly cocartesian} if it is left Kan extended from subsets of cardinality at most 1 (equivalently, any 2-face in the cube is a pushout).
\end{Def}

\begin{Exa}
    When $d = 2$ the conditions above reduce to the usual notions of pushout and pullback squares. A 3-cube is of the form
    \[\begin{tikzcd}[column sep=1em,row sep=1em,ampersand replacement=\&]
    	\& {\cat X_1} \&\& {\cat X_{12}} \\
    	{\cat X_{\emptyset}} \&\& {\cat X_2} \\
    	\& {\cat X_{13}} \&\& {\cat X_{123}} \\
    	{\cat X_3} \&\& {\cat X_{23}}
    	\arrow[from=2-1, to=4-1]
    	\arrow[from=4-1, to=4-3]
    	\arrow[from=1-2, to=1-4]
    	\arrow[from=1-2, to=3-2]
    	\arrow[from=3-2, to=3-4]
    	\arrow[from=1-4, to=3-4]
    	\arrow[from=4-1, to=3-2]
    	\arrow[from=2-1, to=1-2]
    	\arrow[from=2-3, to=1-4]
    	\arrow[from=4-3, to=3-4]
     	\arrow[crossing over, from=2-1, to=2-3]
     	\arrow[crossing over, from=2-3, to=4-3]
    \end{tikzcd}\]
    It is strongly cocartesian if every face is a pushout or, equivalently, if it is left Kan extended from the diagram
    \[\begin{tikzcd}[ampersand replacement=\&]
    	\& {\cat X_1} \\
    	{\cat X_{\emptyset}} \&\& {\cat X_2} \\
    	\\
    	{\cat X_3}
    	\arrow[from=2-1, to=4-1]
    	\arrow[from=2-1, to=2-3]
    	\arrow[from=2-1, to=1-2]
    \end{tikzcd}\]
\end{Exa}

\begin{Rem}\label{rem:pushout=pullback}
    When $\cat C$ is a \emph{stable} $\infty$-category, a cubical diagram in $\cat C$ is cocartesian if and only if it is cartesian. 
    For $d=2$ this is essentially the definition of a stable $\infty$-category, see \cite[Proposition 1.1.3.4]{HALurie},
    and for $d>2$ it is proved by induction on $d$, see \cite[Proposition~1.2.4.13]{HALurie}.
\end{Rem}

\begin{Def}\label{def:fun}
    A functor $F \colon \Sp^c \to \Sp$ is \emph{reduced} if it preserves the zero object. We let $\Fun(\Sp^c,\Sp)$ denote the stable $\infty$-category of \emph{reduced} functors from $\Sp^c$ to $\Sp$.  
\end{Def}

\begin{Rem}\label{Rem:reduction}
    The category $\Fun(\Sp^c,\Sp)$ is a full subcategory of the category of \emph{all} functors from $\Sp^c$ to $\Sp$. In fact, $\Fun(\Sp^c,\Sp)$ is strongly reflexive, i.e., it is presentable, stable under equivalence in the category of all functors, and the inclusion admits a left adjoint; see \cite[Remark 1.4.2.4]{HALurie}. Furthermore, for any functor~$F$ there is a natural equivalence $F\simeq \overline{F}\times F(0)$ where $\overline{F}(X)\coloneqq\fib(F(X)\to F(0))$ is the reduced part of $F$. It follows that there is no real loss of information in focusing on reduced functors. 
\end{Rem}

\begin{Def}\label{def:dexcisive}
    Let $d \geq 1$ be an integer and let $F \colon \Sp^c \to \Sp$ be a functor from finite spectra to spectra.
    We say that 
    $F$ is \emph{$d$-excisive} if it takes strongly cocartesian~\mbox{$(d+1)$-cubes} in $\Sp^c$ to cartesian (or, equivalently, cocartesian by \Cref{rem:pushout=pullback}) \mbox{$(d+1)$-cubes} in $\Sp$. We let $\Exc{d}(\Sp^c,\Sp) \subseteq \Fun(\Sp^c,\Sp)$ denote the full stable subcategory of reduced $d$-excisive functors from $\Sp^c$ to $\Sp$.
\end{Def}

\begin{Exa}\label{exa:1-excisive-functors}
    A functor $F \colon \Sp^c \to \Sp$ is 1-excisive precisely when it carries pushout squares in $\Sp^c$ to pushout squares in $\Sp$. Reduced $1$-excisive functors are called \emph{linear}. Evaluation at the sphere spectrum gives an equivalence $\Exc{1}(\Sp^c,\Sp) \xrightarrow{\sim} \Sp$ between the category of linear functors and the category of spectra, whose inverse sends a spectrum~$A$ to the functor $X\mapsto A \otimes X$.
\end{Exa}

\begin{Rem}
    By taking Ind-completions, there is an equivalence $\Exc{d}(\Sp^c,\Sp) \simeq \Exc{d}^c(\Sp,\Sp)$ between $d$-excisive functors $\Sp^c \to \Sp$ and $d$-excisive functors $\Sp \to \Sp$ that preserve filtered colimits; see \cite[Proposition 6.1.5.4]{HALurie}. 
\end{Rem}

\begin{Prop}
    There are inclusions
    \[
        \Exc{1}(\Sp^c,\Sp)\subseteq \cdots \subseteq \Exc{d}(\Sp^c,\Sp) \subseteq \Exc{d+1}(\Sp^c,\Sp) \subseteq \cdots \subseteq \Fun(\Sp^c,\Sp).
    \]
\end{Prop}

\begin{proof}
    Any $d$-excisive functor is $m$-excisive for each $m \ge d$ by \cite[Corollary~1.11]{Goodwillie03} or \cite[Corollary 6.1.1.14]{HALurie}.
\end{proof}

\begin{Thm}\label{thm:d-excisive-localization}
    The inclusion $\Exc{d}(\Sp^c,\Sp) \hookrightarrow \Fun(\Sp^c,\Sp)$ admits an exact left adjoint $P_d \colon \Fun(\Sp^c,\Sp) \to \Exc{d}(\Sp^c,\Sp)$. 
\end{Thm}

\begin{proof}
	This is one of the main theorems of Goodwillie \cite[Section 1]{Goodwillie03} and is established for $\infty$-categories in \cite[Theorem 6.1.1.10]{HALurie}.
\end{proof}

\begin{Rem}\label{rem:pd}
    By slight abuse of notation we will also  write $P_d$  for the composition of $P_d$ with the inclusion functor:
    \[
        \Fun(\Sp^c, \Sp)\xrightarrow{P_d} \Exc{d}(\Sp^c, \Sp)\hookrightarrow \Fun(\Sp^c, \Sp).
    \]  
    This agrees with the notation used in~\cite{Goodwillie03}. Note that since $\Sp$ is stable, the category $\Exc{d}(\Sp^c,\Sp)$ is closed under all colimits; see~\cite[Remark 6.1.5.10]{HALurie}. This means that the inclusion functor $\Exc{d}(\Sp^c,\Sp) \hookrightarrow \Fun(\Sp^c,\Sp)$ commutes with colimits. It follows that $P_d$ commutes with colimits both when regarded as a functor $\Fun(\Sp^c,\Sp)\to\Exc{d}(\Sp^c,\Sp)$ and as an endofunctor on $\Fun(\Sp^c,\Sp)$.
\end{Rem}

\begin{Def}
    For all $k \ge 0$, we have $P_d P_{d+k} \simeq P_d$ as functors $ \Fun(\Sp^c,\Sp) \to \Exc{d}(\Sp^c,\Sp)$. Indeed, just observe that they are both left adjoint to the inclusion. We therefore obtain natural transformations $P_d \to P_{d-1}$. The \emph{Taylor tower} (or \emph{Goodwillie tower}) of $F \colon \Sp^c \to \Sp$ is the following sequence of natural transformations of functors $\Sp^c \to \Sp$:
    \[
        F \to \cdots \to P_{d+1}F \to P_d F \to P_{d-1} F \to \cdots  \to P_1F \to P_0F\simeq 0. 
    \]
\end{Def}

\begin{Def}\label{Def:n-homogeneous}
    The \emph{$d$-th layer} in the Taylor tower is $D_dF \coloneqq \fib(P_dF \to P_{d-1}F)$. The functor $D_dF$ is both $d$-excisive ($P_dF \simeq F)$ and \emph{$d$-reduced} ($P_{d-1}F \simeq 0$). We call such functors \emph{$d$-homogeneous}. These functors form a full stable subcategory $\Homog_{d}(\Sp^c,\Sp) \subseteq \Exc{d}(\Sp^c,\Sp)$; see \cite[Corollary 6.1.2.8]{HALurie}.
\end{Def}

\begin{Rem}[Multi-variable calculus]\label{rem:multi-variable}
    There is also a multi-variable version of Goodwillie calculus. In particular, we say that a functor $F \colon (\Sp^c)^{\times n} \to \Sp$ is \mbox{\emph{$\vec{d}$-excisive}} for $\vec{d} = (d_1,\ldots,d_n)$ if, for all $1 \le i \le n$ and every sequence of objects $\{X_j \in \Sp^c \}_{j \ne i}$, the composite
    \[
        \Sp^c \hookrightarrow \Sp^c \times \prod_{j \ne i}\{X_j\} \hookrightarrow(\Sp^c)^{\times n} \xrightarrow{F} \Sp^c
    \]
    is $d_i$-excisive. The inclusion of $\vec{d}$-excisive functors into the category of all functors also has a left adjoint (see \cite[Section 1]{AroneKankaanrintaLie} or \cite[Proposition 6.1.3.6]{HALurie}). We will only need the following generalization of \cref{exa:1-excisive-functors}: There is an equivalence 
    \begin{equation}\label{eq:multi-variable}
        \Exc{(1,\ldots,1)}((\Sp^c)^{\times d},\Sp) \xrightarrow{\sim} \Sp
    \end{equation}
    between multi-linear functors (that is, functors of $d$ variables that are reduced and $1$-excisive in each variable) and spectra, given by evaluating at $(\mathbb{S},\ldots,\mathbb{S})$. This is shown in \cite[Section 5.2]{Goodwillie03} at the level of homotopy categories and in~\cite[Remark 6.1.3.3]{HALurie} at the level of $\infty$-categories.
\end{Rem}

\subsection*{The symmetric monoidal structure on \texorpdfstring{$d$}{d}-excisive functors}

We now explain how to construct a symmetric monoidal structure on the category of $d$-excisive functors. We begin with \emph{all} functors from finite spectra to spectra. 

\begin{Cons}\label{rem:sym-mon-functors}
   Recall that the category of functors from finite spectra to spectra obtains a symmetric monoidal structure via Day convolution (see \cite{Glasman16} or \cite[Remark 4.8.1.13]{HALurie}), which we denote by $-\circledast - $. In fact, in this case, using \cite[Corollary 2.2.6.14 and Remark 2.2.6.15]{HALurie}, we see that given two functors $F$ and $G$, we can compute $F \circledast G$ as the left Kan extension in the following diagram:
    \[\begin{tikzcd}[ampersand replacement=\&]
    	{\Sp^c \times \Sp^c} \& {\Sp \times \Sp} \& \Sp \\
    	{\Sp^c}
    	\arrow["\otimes"', from=1-1, to=2-1]
    	\arrow["{F \times G}", from=1-1, to=1-2]
    	\arrow["\otimes", from=1-2, to=1-3]
    	\arrow["{F \circledast G}"', dashed, from=2-1, to=1-3]
    \end{tikzcd}\]
    Informally, we have
    \begin{equation}\label{eq:day-convolution}
       (F \circledast G)(C) = \colim_{C_0 \otimes C_1 \to C} F(C_0) \otimes G(C_1). 
    \end{equation}
	The unit is the functor $\Sigma^\infty_+ \mathrm{Map}^{\circ}_{\Sp}(\Sphere,-)$ where $\mathrm{Map}^{\circ}_{\Sp}(-,-)$ denotes the (unpointed) mapping space. The Day convolution comes with a natural transformation of functors $\Sp^c \times \Sp^c \to \Sp$
    \[
    F(-) \otimes G(-) \to (F \circledast G)(- \otimes -)
    \]
    inducing an equivalence
    \[
    \Hom_{\Fun(\Sp^c,\Sp)}(F \circledast G,H) \xrightarrow{\sim} \Hom_{\Fun(\Sp^c \times \Sp^c,\Sp)}(F(-) \otimes G(-),H(- \otimes -))
    \]
    for all $H \in \Fun(\Sp^c,\Sp)$. See also the discussion in \cite[Section 2]{Ching21}. 

	The category of reduced functors $\Fun(\Sp^c,\Sp)$ inherits a symmetric monoidal structure from the category of all functors via the localization of \cref{Rem:reduction}. This follows, for example, by applying~\cite[Lemma 5.3.4]{CDHHLMNNS23} to the constant diagram on $0$. In fact, \eqref{eq:day-convolution} implies that the Day convolution of two reduced functors is reduced. It follows that the localized Day convolution on $\Fun(\Sp^c,\Sp)$ actually coincides with the ordinary Day convolution --- all that changes is that the unit of $\Fun(\Sp^c,\Sp)$ is the localization of the unit for ordinary Day convolution. This is the reduced functor $\Sigma^{\infty}\Map_{\Sp}(\bbS,-)$ corepresented by the sphere spectrum, where $\Map_{\Sp}(-,-)$ denotes the pointed mapping space; see \cite[Lemma 2.10]{Ching21}. That is, it is the functor 
	\[
		\Sigma^{\infty}\Omega^{\infty} \colon \Sp^c \to \Sp, \quad  M \mapsto \Sigma^{\infty}\Omega^{\infty}M. 
	\]
\end{Cons}

\begin{Rem}
	The symmetric monoidal structure on $\Fun(\Sp^c,\Sp)$ is closed, since the category $\Fun(\Sp^c,\Sp)$ is presentably symmetric monoidal by \cite[Proposition 3.3]{Nikolaus16pp}, that is, it is a presentable symmetric monoidal $\infty$-category whose tensor product commutes with colimits in both variables. In particular, it has an internal hom object which can be computed via
	\[
	\begin{split}
		\ihomsub{\Fun(\Sp^c,\Sp)}{F,G}(x)  & \simeq \int_{d \in \Sp^c} \ihomsub{\Sp}{F(d),G(x\otimes d)}  \\
		& \simeq \Hom_{\Fun(\Sp^c,\Sp)}(F,G(x \otimes -)). 
	\end{split}
	\]
	See \cite[Proposition 3.11]{Nikolaus16pp} and \cite[Proposition 5.1]{GepnerHaugsengNikolaus17}. 
\end{Rem}

\begin{Not}\label{Not:corepresentable_functor}
	For any $x \in \Sp^c$ let 
	\[
		h_x \coloneqq \Sigma^{\infty}\Map_{\Sp}(x,-)  \in \Fun(\Sp^c,\Sp)
	\]
	be the corresponding corepresentable functor (where again $\Map_{\Sp}(-,-)$ denotes the pointed mapping space). Note that
	\begin{equation}\label{eq:internel-hom-representable}
		\ihomsub{\Fun(\Sp^c,\Sp)}{h_x,F} \simeq F(x \otimes -).
	\end{equation}
    by \cite[Lemma 2.10]{Ching21}.
\end{Not} 

\begin{Rem}
	There is a formula for Day convolution with a corepresentable functor that we now describe. For any reduced functor $F\colon \Sp^c\to\Sp$, there is a natural assembly map
	\[
		h_x(y_1)\otimes F(y_2)\to F\!\left(\ihomsub{\Sp}{x, y_1\otimes y_2}\right)
	\]
	which, by the universal property of Day convolution, induces a natural transformation
	\begin{equation}\label{eq: assembly}
		(h_x\circledast F)(-)\to F(\ihomsub{\Sp}{x, -})
	\end{equation}
	of reduced functors $\Sp^c \to \Sp$. 
\end{Rem}

\begin{Lem}\label{lem:assembly}
	The map~\eqref{eq: assembly} is an equivalence. 
\end{Lem}

\begin{proof}
	It is enough to show that for any reduced functor $G\colon \Sp^c \to \Sp$, the map~\eqref{eq: assembly} induces an equivalence 
	\[
		\Hom_{\Fun(\Sp^c,\Sp)}(F(\ihomsub{\Sp}{x, -}),G)\to\Hom_{\Fun(\Sp^c,\Sp)}(h_x\circledast  F, G).
	\]
	This map factors as a composition of equivalences
	\begin{align*}
		\Hom_{\Fun(\Sp^c,\Sp)}(F(\ihomsub{\Sp}{x, -}),G)&\xrightarrow{\simeq} \Hom_{\Fun(\Sp^c,\Sp)}(F,G(x\otimes -))\\ &\xrightarrow{\simeq} \Hom_{\Fun(\Sp^c,\Sp)}\left(F,\ihomsub{\Fun(\Sp^c,\Sp)}{h_x, G}\right)\\&\xrightarrow{\simeq}\Hom_{\Fun(\Sp^c,\Sp)}(F\circledast  h_x, G)
	\end{align*}
	and so is itself an equivalence. The first equivalence follows from the following observation: Given an adjunction $L \colon \cat C \leftrightarrows \cat D \colon R$ and functors $F \colon \cat C \to \cat Z$ and $G \colon \cat D \to \cat Z$, there is a natural equivalence
	\[
		\Hom_{\Fun(\cat D,\cat Z)}(FR,G) \simeq \Hom_{\Fun(\cat C,\cat Z)}(F,GL).
	\]
	The second equivalence uses \eqref{eq:internel-hom-representable} while the third equivalence is the closed monoidal adjunction.
\end{proof}

\begin{Rem}\label{rem:unit}
	It follows from \cref{lem:assembly} that $h_{\bbS}$ is indeed the unit of $\Fun(\Sp^c,\Sp)$ under the Day convolution monoidal structure of \Cref{rem:sym-mon-functors}.
\end{Rem}

The following corollary is a variant of~\cite[Lemma 2.18]{Ching21}.

\begin{Cor}\label{Cor:conv-representables}
	The canonical natural transformation
	\[
		h_{x_1}(y_1)\otimes h_{x_2}(y_2)\to h_{x_1\otimes x_2}(y_1\otimes y_2)
	\]
	induces an equivalence
	\[
		h_{x_1}\circledast h_{x_2}\xrightarrow{\simeq}h_{x_1\otimes x_2}.
	\]
\end{Cor}

\begin{Rem}\label{rem:stable-yoneda}
    This corollary can also be proved by observing that $x \mapsto h_x$ under the stable Yoneda embedding
    \begin{equation}\label{eq:stable-yoneda}
		\Sp^c \to \Fun(\Sp^c,\Sp)\op
    \end{equation}
    which is the composite of the usual space-valued Yoneda embedding, followed by composition with $\Sigma^{\infty}_+$ and then followed by the localization functor of \cref{Rem:reduction}. The usual Yoneda embedding is symmetric monoidal by \cite[Section 3]{Glasman16}, as is post-composition with the symmetric monoidal functor $\Sigma^{\infty}_+$. The localization functor is symmetric monoidal by construction. This is another way of appreciating that $h_{\Sphere}$ is the monoidal unit of $\Fun(\Sp^c,\Sp)$.
\end{Rem}
\begin{Rem}\label{rem:rigid-compact-fun}
 Let $\mathcal{G} \subseteq \Fun(\Sp^c,\Sp)$ denote the full subcategory of corepresentable functors. Ching \cite[Lemma 4.14]{Ching21} establishes that $\Fun(\Sp^c,\Sp)$ is compactly generated, that its compact objects are precisely the retracts of finite colimits of diagrams in~$\mathcal G$, and that these compact objects are closed under the objectwise smash product of functors. In particular, by \Cref{rem:stable-yoneda}, these are all dualizable objects, and so $\Fun(\Sp^c,\Sp)$ is compactly generated by dualizable objects. 
\end{Rem}
\subsection*{Day convolution of excisive functors}

We now return to $d$-excisive functors. Our goal is to show that $P_d \colon \Fun(\Sp^c,\Sp) \to \Exc{d}(\Sp^c,\Sp)$ is a \emph{smashing} localization which is compatible with the Day convolution monoidal structure on $\Fun(\Sp^c,\Sp)$. 

\begin{Def}
    A functor $F \colon \cat C \to \cat D$ between $\infty$-categories is a \emph{localization} if $F$ has a fully faithful right adjoint $G$.
\end{Def}

\begin{Rem}
    It follows that there is an equivalence between $\cat D$ and the full subcategory of $\cat C$ given by the essential image of $G$. We will sometimes abuse terminology and refer to the endofunctor $L \coloneqq G\circ F$ as the localization. A map $f\colon X \to Y$ in $\cat C$ is said to be an \emph{$L$-local equivalence} if $Lf\colon LX\to LY$ is an equivalence. For example, the unit $X \to LX$ of the adjunction is an $L$-local equivalence for each $X \in \cat C$.
\end{Rem}

\begin{Exa}
	The $d$-excisive approximation $P_d\colon \Fun(\Sp^c,\Sp)\to\Exc{d}(\Sp^c,\Sp)$ is a localization by \cref{thm:d-excisive-localization}; cf.~\cref{rem:pd}.
\end{Exa}

\begin{Def}
    Suppose that $\cat C$ is a symmetric monoidal $\infty$-category. We say that a localization $L \colon \cat C \to \cat C$ is \emph{compatible} with the symmetric monoidal structure if, whenever $f \colon X \to Y$ is an $L$-local equivalence and $Z \in \cat C$ is any object, the map $f \otimes \text{id} \colon X \otimes Z \to Y \otimes Z$ is also an $L$-local equivalence.
\end{Def}

The following follows from \cite[Proposition 2.2.1.9]{HALurie}:

\begin{Prop}\label{prop:compatible-localizations}
    Let $(\cat C,\otimes,\unit_{\cat C})$ be a symmetric monoidal $\infty$-category. Suppose that $L \colon \cat C \to \cat C$ is a localization which is compatible with the symmetric monoidal structure. Then $L\cat C$ inherits the structure of a symmetric monoidal $\infty$-category with unit $L(\unit_{\cat C})$ and monoidal product $L(- \otimes -)$. In particular, the localization $\cat C \to  L\cat C$ is a symmetric monoidal functor.
\end{Prop}

\begin{Def}
     Let $(\cat C,\otimes,\unit_{\cat C})$ be a presentably symmetric monoidal stable $\infty$-category. A localization $L \colon \cat C \to \cat C$ is \emph{smashing} if it preserves colimits.
\end{Def}

\begin{Rem}
    If a localization $L\colon \cat C \to \cat C$ is compatible with the symmetric monoidal structure, then the map $X \to L(\unit_{\cat C})\otimes X$ induced by the canonical map $\unit_{\cat C} \to L(\unit_{\cat C})$ is a natural $L$-local equivalence. Hence we obtain a natural map $\alpha_X \colon L(\unit_{\cat C}) \otimes X \to L(X)$ as the composite
    \[
		L(\unit_{\cat C}) \otimes X \to L(L(\unit_{\cat C}) \otimes X) \simeq L(X). 
    \]
\end{Rem}

\begin{Prop}\label{prop:smashing}
    Let $(\cat C,\otimes,\unit_{\cat C})$ be a presentably symmetric monoidal stable \mbox{$\infty$-category} which is compactly generated by dualizable objects. Suppose ${L \colon \cat C \to \cat C}$ is a smashing localization which is compatible with the symmetric monoidal structure. Then:
    \begin{enumerate}
        \item	The natural map $\alpha_X \colon L(\unit_{\cat C}) \otimes X \to L(X)$ is an equivalence for all $X \in \cat C$. 
        \item	$L\cat C$ inherits a symmetric monoidal structure with unit~$L(\unit_{\cat C})$ and monoidal product $- \otimes -$. 
        \item	There is a symmetric monoidal equivalence of stable $\infty$-categories
				\[
					\Mod_{\cat C}(L(\unit_{\cat C})) \simeq L\cat C
				\]
				that exhibits $L\cat C$ as base change along the map $\unit \to L(\unit_{\cat C})$. 
    \end{enumerate}
\end{Prop}

\begin{proof}
For $(a)$, we note that the natural map $\alpha_X$ is an equivalence for all dualizable $X$ \cite[Lemma 3.3.1]{HoveyPalmieriStrickland97}.\footnote{Note that the definition of localization given in \cite{HoveyPalmieriStrickland97} automatically assumes compatibility with the symmetric monoidal structure.} Moreover, the collection of $X$ for which $\alpha_X$ is an equivalence is localizing. Together, our assumptions on $\cat C$ then imply that $\alpha_X$ is an equivalence for all $X \in \cat C$. Part~(b) then follows from part (a) and \Cref{prop:compatible-localizations}. Part~(c) is a well-known consequence: it follows, for example, by applying \cite[Proposition 5.29]{MathewNaumannNoel17} to the adjunction $L \colon \cat C \leftrightarrows L\cat C \colon \iota$ where $\iota \colon L\cat C \hookrightarrow \cat C$ is the inclusion. 
\end{proof}

\begin{Rem}
   We now turn our attention back to $d$-excisive functors. In fact, we work a little more generally, following \cite{CDHHLMNNS23} and \cite{HorelRamzi21}. To that end, let $\mathcal{J} = \{ \overline p_{\alpha} \colon K_{\alpha}^{\triangleright} \to \Sp^c\}$ be a small collection of diagrams in $\Sp^c$. Let $\Fun_{\mathcal{J}}(\Sp^c,\Sp) \subseteq \Fun(\Sp^c,\Sp)$ denote the full subcategory spanned by those functors which send every diagram in $\mathcal J$ to a limit diagram. As explained in \cite[p.~159]{CDHHLMNNS23} or \cite[Theorem 4.2]{HorelRamzi21}, the subcategory $\Fun_{\mathcal J}(\Sp^c,\Sp)$ is a localization of $\Fun(\Sp^c,\Sp)$. In fact, it is the collection of $S$-local objects for some set of maps in $\Fun(\Sp^c,\Sp)$ which implies that  $\Fun_{\mathcal J}(\Sp^c,\Sp)$ is presentable and that the inclusion has a left adjoint~$L_{\mathcal{J}}$ \cite[Proposition 5.5.4.15]{Lurie09pp}.
\end{Rem}

\begin{Prop}\label{prop:sub-functor-category}
	Suppose that the small set of diagrams $\mathcal{J}$ is closed under post composition with $x \otimes (-) \colon \Sp^c \to \Sp^c$ for all $x \in \Sp^c$. Then the localization 
	\[
		L_{\mathcal J} \colon \Fun(\Sp^c,\Sp) \to \Fun_{\mathcal J}(\Sp^c,\Sp)
	\]
	is a smashing localization which is compatible with Day convolution. In particular, $\Fun_{\mathcal J}(\Sp^c,\Sp)$ is a presentably symmetric monoidal $\infty$-category with symmetric monoidal structure given by Day convolution $F \circledast  G$ and with tensor unit $L_{\mathcal J}h_{\bbS}$. 
\end{Prop}  

\begin{proof}
    
	This is all contained in \cite[Lemma 5.3.4]{CDHHLMNNS23} and \cite[Theorem 4.5]{HorelRamzi21} except 
 for the fact that the localization is smashing. Indeed, the inclusion $\Fun_{\mathcal J}(\Sp^c,\Sp)\hookrightarrow \Fun(\Sp^c,\Sp)$ preserves filtered colimits (and hence all colimits), since filtered colimits in $\Sp$ are left exact. We can thus apply \cref{prop:smashing} keeping in mind \cref{rem:rigid-compact-fun}.
\end{proof}

\begin{Thm}\label{thm:properties-n-excisive}
	The $d$-excisive approximation $P_d:\Fun(\Sp^c,\Sp) \to \Exc{d}(\Sp^c,\Sp)$ is a smashing localization on $\Fun(\Sp^c,\Sp)$ which is compatible with the Day convolution. It follows that the category $\Exc{d}(\Sp^c,\Sp)$ is a presentably symmetric monoidal stable $\infty$-category with the tensor product and internal hom both computed in $\Fun(\Sp^c,\Sp)$. The monoidal unit is the functor $P_dh_{\bbS}\colon \Sp^c \to \Sp$. Moreover, $P_dh_{\bbS}=P_d\Sigma^\infty\Omega^\infty$.
\end{Thm}

\begin{proof}
	This follows from \Cref{prop:sub-functor-category}, since $\Exc{d}(\Sp^c,\Sp) \subseteq \Fun(\Sp^c,\Sp)$ is of the form $\Fun_{\mathcal{J}}(\Sp^c,\Sp)$ for an appropriate choice of $\mathcal J$, namely the collection corresponding to all strongly cocartesian cubes. The condition in \Cref{prop:sub-functor-category} is satisfied because the functor $x \otimes (-)$ preserves strongly cocartesian cubes (which follows, for example, from the characterization in \cite[Proposition 6.1.1.15]{HALurie}).
\end{proof}

\begin{Rem}
	In \Cref{cor:pm-finite-localization} we show that $\Exc{d}(\Sp^c,\Sp)$ is a \emph{finite} localization of $\Exc{d+k}(\Sp^c,\Sp)$ for all $k \ge 0$ and provide an explicit set of compact generators for the kernel of the localization functor. However, we suspect that the smashing localization $\Fun(\Sp^c,\Sp)\to\Exc{d}(\Sp^c,\Sp)$ is not a finite localization, i.e., its kernel is not generated by compact objects. In \Cref{example:not finite} we show that the ``obvious'' set of compact generators does not work.
\end{Rem}
\begin{Rem}
    The following results follow from \Cref{prop:smashing,thm:properties-n-excisive} using \Cref{rem:rigid-compact-fun}.
\end{Rem}
\begin{Cor}\label{prop:assembly}
	For any $F \in \Fun(\Sp^c,\Sp)$, there is a natural equivalence
    %\begin{equation}\label{eq:assembly}
    \begin{equation*}
	F\circledast  P_d \Sigma^\infty\Omega^\infty \xrightarrow{\simeq}  P_d F.
    \end{equation*}
\end{Cor}

\begin{Cor}
   Suppose $F$ and $G$ are $m$-excisive and $d$-excisive, respectively. Then $F\circledast  G$ is $\min(m, d)$-excisive. 
\end{Cor}

\begin{proof}
	It follows from \Cref{prop:assembly} that 
	\[
		P_m\Sigma^\infty \Omega^\infty\circledast P_d\Sigma^\infty\Omega^\infty \simeq P_{\min(m,d)}\Sigma^\infty \Omega^\infty
	\]
	and hence
	\begin{align*}
		F\circledast G\simeq P_m F\circledast  P_dG&\simeq F\circledast P_m\Sigma^\infty\Omega^\infty \circledast G\circledast P_d\Sigma^\infty\Omega^\infty \\ &\simeq F\circledast G \circledast  P_{\min(m,d)}\Sigma^\infty \Omega^\infty \\
		&\simeq P_{\min(m,d)} (F\circledast  G).\qedhere
	\end{align*}
\end{proof}

\begin{Rem}\label{rem: F_A}
	Recall that if $\cat C$ is a symmetric monoidal stable $\infty$-category then there is an essentially unique symmetric monoidal colimit-preserving functor $\Sp \to \cat C$ given by $A \mapsto A \otimes \unit_{\cat C}$; see \cite[Corollary 4.8.2.19]{HALurie}. For example, if $\cat C = \Fun(\Sp^c,\Sp)$, then this is the functor $A \mapsto F_A$, where $F_A\in \Fun(\Sp^c,\Sp)$ sends $X \in \Sp^c$ to $A \otimes \Sigma^{\infty}\Omega^{\infty}X$. Applying $P_d$ (or appealing to the universal property directly) we can make the following definition:
\end{Rem}

\begin{Def}\label{def:inflation}
	The \emph{inflation functor} $i_d\colon \Sp \to \Exc{d}(\Sp^c,\Sp)$ is the (essentially unique) symmetric monoidal colimit-preserving functor. It is given by ${A \mapsto P_d(F_A)}$.
\end{Def}

\begin{Exa}\label{exa:inflation}
	Suppose $d= 1$. In this case, it is well-known that $P_1(F_A)$ is equivalent to the functor $X \mapsto A \otimes X$. For example, see~\cite[Example 6.1.1.28]{HALurie} for a proof in the language of $\infty$-categories (or see \Cref{rem:adjoint-to-cross-effects}). In particular, the equivalence $\Exc{1}(\Sp^c,\Sp) \simeq \Sp$ of \Cref{exa:1-excisive-functors} is an equivalence of symmetric monoidal stable $\infty$-categories. In fact, the same statement is true for the equivalence of \eqref{eq:multi-variable} which is given by the canonical symmetric monoidal functor $\Sp \to \Exc{(1,\ldots,1)}((\Sp^c)^{\times n},\Sp)$ that sends $A \in \Sp$ to the functor $(X_1,\ldots,X_n) \mapsto A \otimes X_1 \otimes \ldots \otimes X_n$. 
\end{Exa}

\begin{Rem}[Change of target category]\label{rem:targetcat}
	If $\cat D$ is a presentable symmetric monoidal stable $\infty$-category, then we could also consider the category of reduced \mbox{$d$-excisive} functors $\Exc{d}(\Sp^c,\cat D)$. Much of what we have said thus far can be repeated for this category. Our present goal is to explain how this construction can be viewed in terms of the Lurie tensor product of presentable $\infty$-categories.

    Indeed, let $\Pr^L$ denote the $\infty$-category of presentable $\infty$-categories with the colimit-preserving functors as morphisms. This category has a symmetric monoidal structure constructed in \cite[Section 4.8.1]{HALurie} with the $\infty$-category of spaces $\cat S$ serving as the unit. The full subcategory $\Pr^L_{st}$ consisting of the presentable \emph{stable} \mbox{$\infty$-categories} inherits the monoidal structure from $\Pr^L$ with the category of spectra~$\Sp$ now serving as the unit.  Moreover, if $\cat C$ and $\cat D$ are symmetric monoidal, then the Lurie tensor product $\cat C \otimes \cat D$ also inherits a symmetric monoidal structure. 

    We claim there is an equivalence of symmetric monoidal stable $\infty$-categories 
    \[
		\Exc{d}(\Sp^c,\cat D) \simeq \Exc{d}(\Sp^c,\Sp) \otimes \cat D
    \]
    where the left-hand side has the localized Day convolution monoidal structure, and the right-hand side has the symmetric monoidal structure coming from the Lurie tensor product. Indeed, because $\cat D$ is stable, there is an equivalence $\cat D \simeq \cat D \otimes \Sp \otimes \cat S$. Then we have an equivalence of presentable stable $\infty$-categories
    \[
		\Exc{d}(\Sp^c,\cat D) \simeq \Exc{d}(\Sp^c,\cat D \otimes \Sp \otimes \cat S) \simeq \Exc{d}(\Sp^c,\cat S) \otimes \cat D \otimes \Sp \simeq \Exc{d}(\Sp^c,\Sp) \otimes \cat D 
    \]
    by applying \cite[Theorem 4.2(2)]{HorelRamzi21} twice. Moreover, it follows from \cite[Lemma~4.8]{HorelRamzi21} that this is an equivalence of symmetric monoidal $\infty$-categories. 
\end{Rem}

\begin{Exa}\label{ex:integralcoefficients}
    Taking $\cat D = \Mod_{\HZ}$ in the previous example, we see that
    \[
		\Exc{d}(\Sp^c,\Mod_{\HZ}) \simeq \Exc{d}(\Sp^c,\Sp) \otimes \Mod_{\HZ}.
    \]
    This is the $d$-excisive version of Kaledin's derived Mackey functors; cf.~\cite{PatchkoriaSandersWimmer22}. 
\end{Exa}

\section{Cross-effects and idempotents}\label{sec:cross-idempotents}

The cross-effects play an important role in the study of polynomial and excisive functors. In this section we will review several notions of cross-effect and show that for functors with values in a stable $\infty$-category, the definition of the cross-effect via idempotents, in the style of Eilenberg--Mac Lane~\cite{EilenbergMacLane54}, is equivalent to the definition of the cross-effect via total homotopy (co)fibers introduced by Goodwillie~\cite{Goodwillie03}. The idempotent model is convenient for analyzing certain operations with cross-effects that we will need. In particular, it is useful for calculating natural transformations and Day convolutions of cross-effects.

\begin{Not}
	Throughout this section $\cat{D}$ denotes a stable $\infty$-category and $F$ denotes a (not necessarily reduced) functor from $\Sp^c$ to $\cat{D}$. 
\end{Not}

\begin{Rem}
	If $\cat{D}$ is a stable $\infty$-category, then so is $\cat{D}\op$. This means that the material in this section applies both to covariant and contravariant functors from $\Sp^c$ to $\cat{D}$. In subsequent sections we will apply the cross-effect construction to a contravariant functor (a version of the Yoneda embedding) from $\Sp^c$ to $\Fun(\Sp^c, \Sp)$. 
\end{Rem}

\begin{Def}[Goodwillie]
	The \emph{$d$-th cross-effect functor} of $F$ is the functor $\crosseffect_d F\colon (\Sp^c)^{\times d} \to \cat{D}$ defined by
	\[
		\crosseffect_dF(x_1,\ldots,x_d) \coloneqq \tofib_{S \subseteq \num{d}} \left \{ F\left(\bigoplus_{i \not \in S}x_i \right)\right\},
	\]
	the iterated (or total) fiber of a $d$-cube formed by applying $F$ to the direct sums of subsets of $x_1 \ldots, x_d$. The morphisms in the cube are given by the relevant collapse maps $x_i \to 0$. See \cite[Construction 6.1.3.20]{HALurie}.

	Dually, we define the \emph{$d$-th co-cross-effect} $\crosseffect^d F\colon (\Sp^c)^{\times d}\to\cat{D}$ by
	\[
		\crosseffect^dF(x_1,\ldots,x_d) \coloneqq \tcof_{S \subseteq \num{d}}\left \{ F\left(\bigoplus_{i \not \in S}x_i \right)\right\},
	\]
	the total cofiber of the cubical diagram with the same objects as before, but where the morphisms are induced by the inclusions $0 \to x_i$.
\end{Def}

\begin{Rem}\label{rem:sigma-n-action}
	By permutation of variables, the (co-)cross-effect has a (naive) action of~$\Sigma_d$. At the $\infty$-categorical level, this corresponds to the fact that the $d$-th cross-effect is a symmetric $d$-ary functor; see \cite[Section 6.1.4]{HALurie}. Using the coreduction functor defined in \cite[Construction 6.2.3.6]{HALurie} one can also check that the $d$-th co-cross-effect is a symmetric $d$-ary functor.
\end{Rem}

\begin{Exa}
	The first cross-effect of $F$ is the fiber
	\[
		\crosseffect_1F(x) = \fib(F(x) \to F(0))
	\]
	and the first co-cross-effect is the cofiber
	\[
		\crosseffect^1F(x) = \cofib(F(0) \to F(x)).
	\]
	The second cross-effect is the total fiber
	\[
		\crosseffect_2F(x_1,x_2) = \thofib\left(\begin{tikzcd}[ampersand replacement=\&]
		{F(x_1 \oplus x_2)} \& {F(x_2)} \\
		{F(x_1)} \& {F(0)}
		\arrow[from=1-1, to=2-1]
		\arrow[from=1-1, to=1-2]
		\arrow[from=2-1, to=2-2]
		\arrow[from=1-2, to=2-2] \end{tikzcd}\right)
	\]
	and the second co-cross-effect is the total cofiber
	\[
		\crosseffect^2F(x_1,x_2) = \tcof\left(\begin{tikzcd}[ampersand replacement=\&]
		{F(x_1 \oplus x_2)} \& {F(x_2)} \\
		{F(x_1)} \& {F(0)}
		\arrow[from=2-1, to=1-1]
		\arrow[from=1-2, to=1-1]
		\arrow[from=2-2, to=2-1]
		\arrow[from=2-2, to=1-2] \end{tikzcd}\right).
	\]
\end{Exa}

\begin{Lem}\label{lem:crosstococross}
	For any functor $F\colon \Sp^c\to\cat D$, there are natural transformations of functors of $d$ variables
	\begin{equation}\label{eq:crosstococross}
		\crosseffect_dF(x_1, \ldots, x_d) \to F(x_1\oplus\cdots\oplus x_d) \to \crosseffect^d F(x_1, \ldots, x_d)
	\end{equation}
	whose composition is an equivalence.
\end{Lem}

\begin{proof}
	This is shown in~\cite[Lemma B.1]{Heuts21}.
\end{proof}

\begin{Rem}\label{rem:crd-idem}
	It follows from \Cref{lem:crosstococross} that there is an idempotent endomorphism in the homotopy category of functors of $d$ variables:
	\begin{equation}\label{eq:crosseffectidemptotent}
		F(x_1\oplus\cdots\oplus x_d) \to \crosseffect^d F(x_1, \ldots, x_d) \xleftarrow{\simeq} \crosseffect_d F(x_1, \ldots, x_d) \to F(x_1\oplus \cdots \oplus x_d).
	\end{equation}
\end{Rem}

\begin{Rem}\label{rem:idempotents-lift}
	The notion of an idempotent endomorphism in an $\infty$-category is somewhat involved, since it has to incorporate higher coherences; see \cite[Definition 4.4.5.4]{HTTLurie}. However, for a stable $\infty$-category, every idempotent in the homotopy category lifts to an idempotent in the $\infty$-category~\cite[Proof of Lemma 1.2.4.6 and Warning 1.2.4.8]{HALurie}; in other words, the existence of higher coherences comes for free. In the sequel, when we say ``idempotent'', we will mean ``idempotent endomorphism'' in the relevant $\infty$-category, but when we say that idempotents ``commute'' or are ``orthogonal'', we will mean that they do so in the homotopy category.
\end{Rem}

\begin{Rem}\label{rem:idempotents-split}
	Suppose $e\colon X\to X$ is an idempotent in a stable $\infty$-category which admits sequential colimits. We define
	\[
		eX\coloneqq \colim (X\xrightarrow{e}X\xrightarrow{e}\cdots).
	\]
	The map $(1-e)$ is an idempotent as well, and there is a direct sum decomposition
	\[
		X \simeq eX \oplus (1-e)X
	\]
	which is natural with respect to maps that commute with $e$; see \cite[Proposition~1.6.8]{Neeman01}, for example. More generally, if $e_1, \ldots, e_d\colon X \to X$ are idempotents that are pairwise orthogonal in the homotopy category and such that $e_1+\cdots+e_d\simeq 1_X$ then there is an equivalence
	\[
		X \simeq e_1X\oplus\cdots\oplus e_dX.
	\]
\end{Rem}

\begin{Rem}
	It follows from \cref{rem:idempotents-lift} that the idempotent map~\eqref{eq:crosseffectidemptotent} can be thought of as an idempotent in the $\infty$-category of functors $(\Sp^c)^{\times d}\to \cat D$. We denote this idempotent by $\crosseffect(d)$. It follows that the $d$-th cross-effect $\crosseffect_dF(x_1, \ldots, x_d)$ is a direct summand of $F(x_1\oplus\cdots\oplus x_d)$, split off by the idempotent $\crosseffect(d)$. In other words, $\crosseffect_dF(x_1, \ldots, x_d)\simeq \crosseffect(d)F(x_1, \ldots, x_d)$ using the notation of \cref{rem:idempotents-split}.
\end{Rem}

\begin{Rem}
	Our next goal is to give a ``formula'' for the idempotent $\crosseffect(d)$ which expresses it in terms of more elementary idempotents. This formula was first given in~\cite{EilenbergMacLane54} for functors between abelian categories. We will adapt their constructions to the setting of stable $\infty$-categories. 
\end{Rem}

\begin{Def}
	Let $x_1, \ldots, x_d\in\Sp^c$ and $U\subseteq \num{d}$. Define
	\[
		\psi_U \colon x_1\oplus \cdots \oplus x_d\to x_1\oplus \cdots \oplus x_d
	\]
	to be the sum of the identity morphisms $x_i\xrightarrow{1} x_i$ for all $i\in U$ and the zero morphisms $x_i\xrightarrow{0} x_i$ for $i\notin U$. When $U=\{i\}$ is a singleton, we denote $\psi_U$ by $\psi_i$.
\end{Def}

\begin{Rem}
	We regard $\psi_U$ as a natural endomorphism of the functor 
	\[
		\bigoplus\nolimits_{\hspace{-0.1em}d}\colon (\Sp^c)^{\times d}\to \Sp^c
	\]
	which sends $(x_1, \ldots, x_d)$ to $x_1\oplus\cdots\oplus x_d$. For any functor $F\colon \Sp^c\to \cat D$, we then obtain a natural endomorphism $F\ast \psi_U$ of $F\circ \bigoplus_d$. %
\end{Rem}

\begin{Lem}\label{lem:pairwise}
	The morphisms $\{\psi_U\mid U\subseteq\num{d}\}$ are pairwise commuting idempotents. For every $U, V\subseteq \num{d}$, $\psi_U\circ\psi_V=\psi_{U\cap V}$. For every $U\subseteq \num{d}$, there is an equivalence
	\[
		\psi_U \simeq \sum_{i\in U} \psi_i.
	\]
\end{Lem}

\begin{proof}
	These are routine verifications which we leave to the reader.
\end{proof}

\begin{Rem}
	Given a functor $F\colon \Sp^c \to \cat{D}$, it follows from \Cref{lem:pairwise} that 
	\[
		\SET{F*\psi_U}{U\subseteq \num{d}}
	\]
	are pairwise commuting idempotents of the functor $(x_1,\ldots,x_d) \mapsto {F(x_1\oplus\cdots \oplus x_d)}$. Furthermore, $(F*\psi_U)(F*\psi_V)=F*\psi_{U\cap V}$ for all $U, V\subseteq \num{d}$.
\end{Rem}

\begin{Def}\label{def:f_U}
	For each subset $U\subseteq \num{d}$, define the natural transformation
	\[
		\crosseffect(U)\colon F\circ \bigoplus\nolimits_{\hspace{-0.1em}d} \to F\circ \bigoplus\nolimits_{\hspace{-0.1em}d}
	\]
	as follows:
	\begin{align*}
		\crosseffect(U)&\coloneqq F*\psi_U\circ\left(\underset{i\in U}{\bigcirc}(1-F*\psi_{U\setminus\{i\}})\right) \\ 
		&\simeq \underset{i\in U}{\bigcirc}(F*\psi_U-F*\psi_{U\setminus\{i\}}) \\
		&\simeq \sum_{V\subseteq U}(-1)^{|U|-|V|}F*\psi_V.
	\end{align*}
	Here ${\circ}$ and $\bigcirc$ denote composition. All maps of the form $F*\psi_U$ and $1-F*\psi_{U\setminus\{i\}}$ commute with each other, so the order of composition is not important. That the different formulas for $\crosseffect(U)$ are equivalent is an exercise.
\end{Def}

\begin{Exa}\label{exa:crossd}
	When $U=\num{d}$, we denote $\crosseffect(U)$ by $\crosseffect[d]$. Thus we have
	\[
		\crosseffect\num{d}=\underset{i\in \num{d}}{\bigcirc}(1-F*\psi_{[d]\setminus\{i\}})\simeq \sum_{V\subseteq \num{d}}(-1)^{d-|V|}F*\psi_V.
	\]
\end{Exa}

\begin{Rem}
    The $\crosseffect(U)$ are analogous to the maps $D_\sigma$ in~\cite[(9.7), page 77]{EilenbergMacLane54}.
\end{Rem}

\begin{Lem}[\cite{EilenbergMacLane54}, Proof of Theorem 9.1]\label{lemma:basic-idempotents}
	The maps $\{\crosseffect(U)\mid U\subseteq \num{d}\}$ are pairwise orthogonal idempotents whose sum is the identity.
\end{Lem}

\begin{proof}
	The fact that the $\psi_U$ are pairwise commuting homotopy idempotents implies that all maps of the form $F*\psi_U$ commute with each other and also with maps of the  form $1-F*\psi_V$. Therefore, the maps $\crosseffect(U)$ commute with each other. Each map $\crosseffect(U)$ is a composition of commuting idempotents, so it is itself an idempotent.

	To prove that the idempotents $\crosseffect(U)$ are pairwise orthogonal, let $U, U'$ be distinct subsets of $\num{d}$. Without loss of generality we may assume that there exists an element $i\in U\setminus U'$. Then $U'\subseteq \num{d}\setminus \{i\}$ and thus $\crosseffect(U')=F*\psi_{\num{d}\setminus\{i\}} \crosseffect(U')$. On the other hand, since $i\in U$, $F*\psi_U-F*\psi_{U\setminus\{i\}}$ is a factor of $\crosseffect(U)$. So we may write $\crosseffect(U)=\bar \crosseffect(U)\circ (F*\psi_U-F*\psi_{U\setminus\{i\}})$ for some map $\bar \crosseffect(U)$. It follows that there are equivalences
	\begin{align*}
		\crosseffect(U) \circ \crosseffect(U')&\simeq \bar \crosseffect(U)\circ (F*\psi_U-F*\psi_{U\setminus\{i\}})F*\psi_{\num{d}\setminus\{i\}} \circ \crosseffect(U')\\
		&\simeq \bar \crosseffect(U)\circ (F*\psi_{U\setminus\{i\}}-F*\psi_{U\setminus\{i\}}) \circ \crosseffect(U') \\
		&\simeq 0.
	\end{align*}
	Finally, we want to prove that $\sum_{U\subseteq \num{d}}\crosseffect(U)\simeq 1$. If we adopt the formula $\crosseffect(U)=\sum_{V\subseteq U}(-1)^{|U|-|V|}F*\psi_V$, then it is clear that 
	\[
		\sum_{U\subseteq \num{d}}\crosseffect(U) \simeq \sum_{V\subseteq \num{d}}\sum_{V\subseteq U\subseteq \num{d}} (-1)^{|U|-|V|}F*\psi_V.
	\]
	For all $V\subsetneq \num{d}$, $\sum_{V\subseteq U \subseteq \num{d}}(-1)^{|U|-|V|}=0$. So $\sum_U \crosseffect(U)\simeq F*\psi_{\num{d}}$, which is the identity. 
\end{proof}

\begin{Not}\label{notation:cruf}
	We will write $\crosseffect(U)F(x_1\oplus\cdots\oplus x_d)$ for the direct summand of $F(x_1\oplus\cdots\oplus x_d)$ split off by the idempotent $\crosseffect(U)\colon F(x_1\oplus \cdots\oplus x_d) \to F(x_1 \oplus\cdots \oplus x_d)$.
\end{Not}

\begin{Rem}
	If $U=\{i_1, \ldots, i_s\}\subseteq \num{d}$, then $\crosseffect(U) F(x_1\oplus \cdots \oplus x_d)$ is denoted by $F(x_{i_1}|\cdots|x_{i_s})$ in~\cite{EilenbergMacLane54}. We will see in \cref{lem:equivalent} below that indeed $\crosseffect(U) F(x_1\oplus \cdots \oplus x_d)$ really only depends on $x_{i_1}, \ldots, x_{i_s}$.
\end{Rem}

\begin{Cor}\label{cor:splitting}
	Let $F\colon\Sp^c \to \cat D$ be a functor and let $x_1, \ldots, x_d\in \Sp^c$. There is an equivalence
	\[
		F(x_1\oplus\cdots\oplus x_d)\simeq \prod_{U\subseteq \num{d}}  \crosseffect(U)F(x_1\oplus\cdots\oplus x_d)
	\]
	which is natural in $x_1,\ldots,x_d$.
\end{Cor}

\begin{proof}
	This is an immediate consequence of \Cref{lemma:basic-idempotents}. It also is~\cite[Theorem~9.1]{EilenbergMacLane54} lifted to functors between stable $\infty$-categories. 
\end{proof}

The next two lemmas record some elementary properties of the idempotents~$\crosseffect(U)$.

\begin{Lem}\label{lem:trivial}
	Let $x_1, \ldots, x_d\in \Sp^c$ and suppose that $x_i\simeq 0$ for some $1 \le i \le d$. Then for all subsets $\{i_1, \ldots, i_s\}\subseteq\num{d}$ that contain $i$, $\crosseffect(U)F(x_1\oplus\cdots\oplus x_d)\simeq 0$.
\end{Lem}

\begin{proof}
    Suppose $x_i\simeq 0$ and $i\in U$. Then $\psi_U= \psi_{U\setminus\{i\}}$ and therefore
    \[
		F*\psi_U-F*\psi_{U\setminus\{i\}}\colon F(x_1\oplus\cdots\oplus x_d)\to F(x_1\oplus\cdots\oplus x_d)
	\]
	is (homotopic to) the trivial map. But $F*\psi_U-F*\psi_{U\setminus\{i\}}$ is a factor of $\crosseffect(U)$. It follows that $\crosseffect(U)$ acts trivially on $F(x_1\oplus\cdots\oplus x_d)$ in this case, and therefore $\crosseffect(U)F(x_1\oplus\cdots\oplus x_d)\simeq 0$.
\end{proof}

\begin{Lem}\label{lem:equivalent}
	Let $y_1, \ldots y_d\in \Sp^c$ and suppose we are given maps $x_i\to y_i$ for $1\le i \le d$. Suppose there is a set $U\subseteq \num{d}$ such that the map $x_i\to y_i$ is an equivalence for all $i\in U$. Then the following induced map is an equivalence:
	\[
		\crosseffect(U)F(x_1\oplus\cdots\oplus x_d)\xrightarrow{\simeq} \crosseffect(U)F(y_1\oplus \cdots \oplus y_d).
	\]
\end{Lem}

\begin{proof}
	Note that our assumptions imply that the map 
	\[
		x_1\oplus\cdots\oplus x_d \to y_1\oplus\cdots\oplus y_d
	\]
	induces an equivalence
	\[
		\psi_V(x_1\oplus\cdots\oplus x_d) \xrightarrow{\simeq} \psi_V(y_1\oplus\cdots\oplus y_d)
	\]
	for all $V\subseteq U$. It follows that $\crosseffect(U)=\sum_{V\subseteq U}(-1)^{|U|-|V|} F(\psi_V)$ also induces an equivalence
	\[
		\crosseffect(U)F(x_1\oplus\cdots\oplus x_d) \xrightarrow{\simeq} \crosseffect(U)F(y_1\oplus\cdots\oplus y_d). \qedhere
	\]
\end{proof}

\begin{Rem}
	Now we are ready to express the cross-effect of a functor $F$ in terms of the idempotents $\crosseffect(U)$. %
\end{Rem}

\begin{Prop}\label{prop:cubesplitting}
	Let $x_1,\ldots,x_d\in\Sp^c$ and let $F$ be a functor from~$\Sp^c$ to $\cat D$. Consider the two $d$-dimensional cubical diagrams that both send a subset $S\subseteq \num{d}$ to $F\left(\bigoplus_{i\notin S} x_i \right)$ and whose morphisms are induced by the relevant collapse maps $x_i\to 0$ for the first diagram and the relevant inclusion maps $0\to x_i$ for the second diagram. These diagrams are equivalent to the diagrams that send $S$ to the product
	\[
		\prod_{U\subseteq \num{d}\setminus S}\crosseffect(U)F(x_1\oplus\cdots \oplus x_d)
	\]
	and whose maps are the obvious projections (for the first diagram) and the obvious inclusions (for the second diagram).
\end{Prop}

\begin{proof}
	We can identify $\bigoplus_{i\notin S} x_i$ with $\bigoplus_{i=1}^d y_i$ where $y_i=0$ for $i\in S$ and $y_i=x_i$ for $i\notin S$. It follows that we may identify either one of the two versions of the cubical diagram $S\mapsto F(\bigoplus_{i\notin S} x_i)$ with a cubical diagram where $F$ is evaluated at a wedge sum of $d$ spectra, and where the maps are induced by a wedge sum of $d$ maps of spectra.

	By \cref{cor:splitting}, $F(y_1\oplus\cdots\oplus y_d)$ splits as a product of $\crosseffect(U)F(y_1\oplus\cdots\oplus y_d)$, where $U$ ranges over subsets of $\num{d}$, and the maps in either version of the cubical diagram are compatible with the splitting. By \Cref{lem:trivial}, for each $U$ that is not contained in $\num{d}\setminus S$, $\crosseffect(U)F(y_1\oplus\cdots\oplus y_d)\simeq 0$, and so can be dropped from the product. By \Cref{lem:equivalent}, whenever $S\subseteq S'\subseteq \num{d}$, for any $U\subseteq \num{d}\setminus S'$ the factors $\crosseffect(U)F(\bigoplus_{i\notin S} x_i)$ and $\crosseffect(U) F(\bigoplus_{i\notin S'} x_i)$ are equivalent, and can be identified with $\crosseffect(U)F(x_1\oplus \cdots\oplus x_d)$.
\end{proof}

We now can state the precise relationship between $\crosseffect_d$, $\crosseffect^d$, $\crosseffect\num{d}$ and $\crosseffect(d)$.

\begin{Lem}\label{lem:cross-effect}
	The idempotent $\crosseffect\num{d}\colon F(x_1\oplus\cdots\oplus x_d)\to F(x_1\oplus\cdots\oplus x_d)$ of \cref{exa:crossd} is equivalent to the idempotent $\crosseffect{(d)}\colon F(x_1\oplus\cdots\oplus x_d)\to F(x_1\oplus\cdots\oplus x_d)$ of \cref{rem:crd-idem}. For any functor $F$ there are natural equivalences
	\[
		\crosseffect_d F(x_1, \ldots, x_d)\simeq \crosseffect^d F(x_1, \ldots, x_d)\simeq \crosseffect\num{d} F(x_1\oplus \cdots\oplus x_d).
	\]
\end{Lem}

\begin{proof}
    We already saw in \Cref{lem:crosstococross} that $\crosseffect_d F(x_1, \ldots, x_d)\simeq \crosseffect^d F(x_1, \ldots, x_d)$, but this also follows from \Cref{prop:cubesplitting}. The proposition says that $\crosseffect_d F(x_1, \ldots, x_d)$ (respectively, $\crosseffect^d F(x_1, \ldots, x_d)$) is equivalent to the total fiber (respectively, cofiber) of the cubical diagram that sends a subset $S\subseteq \num{d}$ to $\prod_{U\subseteq \num{d}\setminus S} \crosseffect(U)F(\bigoplus_{i\notin S}x_i)$ where the maps are given by projections (respectively, inclusions). A straightforward calculation shows that the factors corresponding to $\crosseffect(U)F$ with $U$ a proper subset of $\num{d}$ ``cancel out'', and the total fiber/cofiber is equivalent to $\crosseffect\num{d}F(x_1\oplus\cdots\oplus x_d)$. It also follows that the map $F(x_1\oplus\cdots\oplus x_d) \to \crosseffect^dF(x_1, \ldots, x_d)$ is equivalent to the projection $F(x_1\oplus\cdots\oplus x_d)\to \crosseffect\num{d}F(x_1\oplus\cdots\oplus x_d)$, and the map $\crosseffect_dF(x_1, \ldots, x_d)\to F(x_1, \ldots, x_d)$ is equivalent to the section of the projection map. It follows that $\crosseffect[d]$ is equivalent to the idempotent $\crosseffect(d)$ of line~\eqref{eq:crosseffectidemptotent} in \Cref{rem:crd-idem}.
\end{proof}

We have one final lemma, which gives a kind of generalized converse of the formula for $\crosseffect(U)$ in \Cref{def:f_U}.

\begin{Lem}\label{lem:product-to-sum}
	For any $U\subseteq \num{d}$, there are equivalences
	\begin{equation}\label{eq:F(e_U)}
		F*\psi_U\simeq \sum_{V\subseteq U} \crosseffect(V)
	\end{equation}
	and
	\begin{equation}\label{eq:1-f(e_U)}
		1-F*\psi_U\simeq \sum_{V\not\subseteq U} \crosseffect(V).
	\end{equation}
	Moreover, for any collection of subsets $U_1, \ldots, U_d\subseteq \num{d}$, there is an equivalence
	\begin{equation}\label{eq:product-to-sum}
		\bigcirc_{i=1}^d(1-F*\psi_{U_i})\simeq \sum_{V\not\subseteq U_1, \ldots, U_d} \crosseffect(V)
	\end{equation}
	where the sum on the right-hand side is over all subsets $V\subseteq \num{d}$ that are not contained in any of the sets $U_1, \ldots, U_d$.
\end{Lem}

\begin{proof}
	Observe that we have equivalences
	\begin{align*}
		1&\simeq\bigcirc_{i\in U}\left(F*\psi_{\num{d}\setminus\{i\}}+\left(1-F*\psi_{\num{d}\setminus\{i\}}\right)\right)\\ 
		&\simeq\sum_{V\subseteq U} F*\psi_{\num{d}\setminus V}\circ (\bigcirc_{i\in U\setminus V}(1-F*\psi_{\num{d}\setminus\{i\}})).
	\end{align*}
	Therefore we have 
	\begin{align*}
		F*\psi_U&\simeq F*\psi_U\sum_{V\subseteq U} F*\psi_{\num{d}\setminus V}\circ (\bigcirc_{i\in U\setminus V}(1-F*\psi_{\num{d}\setminus\{i\}})) \\ 
		&\simeq\sum_{V\subseteq U}  \bigcirc_{i\in U\setminus V}(F*\psi_{U\setminus V}-F*\psi_{(U\setminus V)\setminus \{i\}})) \\ 
		&\simeq \sum_{V\subseteq U} \crosseffect(U\setminus V)= \sum_{V\subseteq U} \crosseffect(V)
	\end{align*}
	which proves~\eqref{eq:F(e_U)}. The equality~\eqref{eq:1-f(e_U)} follows from~\eqref{eq:F(e_U)} and the fact that $\sum_{U\subseteq \num{d}} \crosseffect(U)\simeq 1.$ Finally~\eqref{eq:product-to-sum} follows from~\eqref{eq:1-f(e_U)} and the fact that the $\crosseffect(U)$ are pairwise orthogonal idempotents.
\end{proof}

\section{The cross-effects of representable functors}\label{sec:cross-rep}

We now investigate the cross-effects of representable functors and show that they provide a convenient set of generators for $\Exc{d}(\Sp^c,\Sp)$. These generators will play an important role in this work. The material in the previous section will help us analyze the Day convolutions and natural transformations between these generators.

\begin{Rem}
	Recall from \Cref{Not:corepresentable_functor} that $h_x\in \Fun(\Sp^c,\Sp)$ is the reduced functor corepresented by $x\in \Sp^c$, i.e., $h_x(y)=\Sigma^\infty \Map_{\Sp}(x, y).$ In this section we will view the assignment $x\mapsto h_x$ as a contravariant functor from $\Sp^c$ to $\Fun(\Sp^c, \Sp)$, or equivalently as a functor from $\Sp^c$ to $\Fun(\Sp^c, \Sp)\op$. Let us formalize this in a definition:
\end{Rem}

\begin{Def}
	Let $F\colon \Sp^c\to \Fun(\Sp^c, \Sp)\op$ be the functor defined by the formula
	\[
		F(x)\coloneqq h_x.
	\]
	One can view $F$ as a stable and reduced variant of the Yoneda embedding.
\end{Def}

\begin{Def}\label{def:hxi}
	Fix a finite spectrum $x \in \Sp^c$. For each integer $i \ge 1$, we define $h_x(i)\in \Fun(\Sp^c,\Sp)$ by
	\[
		h_x(i) \coloneqq \crosseffect_i F(x,\ldots,x).
	\]
	More generally, given a finite set $U$, we define
	\[
		h_x(U)\coloneqq\crosseffect(U) F\left(\bigoplus_U x\right)
	\]
	where the right-hand side uses \cref{notation:cruf}.
\end{Def}

\begin{Rem}
	For each $i\ge 1$, we thus have a bi-functor $(x, y)\mapsto h_x(i)(y)$ which is contravariant in $x$ and covariant in $y$.
\end{Rem}

\begin{Rem}\label{rem:h_x(i)}
	By \Cref{lem:cross-effect}, there are three equivalent models for $h_x(i)$: as the cross-effect, the co-cross-effect, and as the image of the idempotent $\crosseffect(i)$ acting on 
	\[
		h_{\underbrace{x\oplus\cdots\oplus x}_i}.
	\] 
	\noindent The co-cross-effect leads to a particularly simple description of $h_x(i)$. Recall that the co-cross-effect of $h_x$ in the variable $x$ is the total cofiber of the cubical diagram 
	\[
		S \mapsto h_{\left(\underset{{j\in \num{i}\setminus S}}{\bigvee} x_j\right)}(-)=\Sigma^{\infty}\prod_{j\in\num{i}\setminus S }\Map_{\Sp}\left(x_j,-\right)
	\] 
	where the maps are inclusions of factors. The total cofiber of this diagram is $h_{x_1}\otimes \cdots\otimes h_{x_i}$. This is just saying that $n$-fold smash product of spaces is equivalent to the total cofiber of a cube of cartesian products of spaces, indexed by all subsets of $\num{n}$. After substituting $x$ for $x_1, \ldots, x_i$, we conclude that
	\begin{equation}\label{eq:smash decomposition}
		h_x(i)(y) \simeq h_x(y)\otimes\cdots\otimes h_x(y). 
	\end{equation}
	In other words, $h_x(i)$ is given by the pointwise tensor product of $i$ copies of $h_x$.
\end{Rem}  

\begin{Rem}\label{rem:same-cross-effect}
	There is a natural equivalence $h_x(y\oplus y)\xrightarrow{\sim} h_{x\oplus x}(y)$. It follows that we could equivalently define $h_x(i)$ to be the $i$-th cross-effect in the covariant variable. However, it suits our goals better to take the cross-effects in the contravariant variable, as in \cref{def:hxi}.
\end{Rem}

\begin{Rem}
	From now on we will mostly take $x=\bbS$ and consider the functors~$h_\bbS(i)$ given by
	\[
		h_{\bbS}(i)(y)\simeq h_{\bbS}(y)^{\otimes i} \simeq (\Sigma^\infty \Omega^\infty y)^{\otimes i}\simeq \Sigma^\infty (\Omega^\infty y)^{\wedge i}.
	\] 
	Our next task is to calculate the Day convolutions of the functors $h_{\bbS}(i)$ for various~$i$. We will use the notation
	\[
		\num{i}_+\wedge x \coloneqq\underbrace{x\oplus\ldots \oplus x}_{i}.
	\]
	We also need the following definition, which will show up again later in the paper.
\end{Rem}

\begin{Def}\label{def:good-subsets}
	Let $i$ and $j$ be positive integers, and write $\num{i} = \{1,\ldots,i\}$ for the finite set with $i$ elements. A subset of $\num{i} \times \num{j}$ is said to be \emph{good}  if its projections onto $\num{i}$ and $\num{j}$ are both surjective.
\end{Def}

\begin{Rem}
	Note that in order to be good, a subset must have cardinality between $\max(i,j)$ and $ij$.
\end{Rem}

\begin{Prop}\label{prop:day-convolution}
	There is an equivalence
	\[
		h_{\bbS}(i) \circledast  h_{\bbS}(j) \simeq \bigoplus_{\substack{U \subseteq \num{i} \times \num{j} \\ \
		U \text{ good}}} h_\bbS(|U|)
	\]
	in $\Fun(\Sp^c,\Sp)$. 
\end{Prop}

\begin{proof}
	To prove the proposition, we will use the idempotent model of the cross-effect. By \Cref{lem:cross-effect}, $h_{\bbS}(i)$ is equivalent to $\crosseffect(i)h_{\num{i}_+\wedge\bbS}$ where 
	\begin{equation}\label{eq:idempotent h_S}
		\crosseffect(i)=\underset{1\le s\le i}{\bigcirc}(1-h_{\psi_{\num{i}\setminus\{s\}}}).
	\end{equation}
	Here $\psi_{\num{i}\setminus\{s\}}\colon \num{i}_+\wedge\bbS\to \num{i}_+\wedge \bbS$ is the idempotent map obtained by collapsing copy number $s$ of $\bbS$ to a point. Thus 
	\[
		h_{\psi_{\num{i}\setminus\{s\}}}\colon h_{\num{i}_+\wedge\bbS}\to h_{\num{i}_+\wedge\bbS}
	\]
	is the idempotent map induced by $\psi_{\num{i}\setminus\{s\}}$.

	Day convolution commutes with colimits in each variable, so there are equivalences
	\[
		h_{\bbS}(i) \circledast  h_{\bbS}(j)\simeq \crosseffect(i)h_{\num{i}_+ \wedge \bbS} \circledast \crosseffect(j) h_{\num{j}_+\wedge \bbS } \simeq\crosseffect(i)\crosseffect(j)\left(h_{\num{i}_+\wedge \bbS}\circledast h_{\num{j}_+\wedge \bbS}\right)
	\]
	where $\crosseffect(i)$ and $\crosseffect(j)$ act on $\num{i}_+\wedge \bbS$ and $\num{j}_+\wedge \bbS$ respectively. Next we need to figure out the effect of $\crosseffect(i)$ and $\crosseffect(j)$ on $h_{\num{i}_+\wedge \bbS}\circledast h_{\num{j}_+\wedge \bbS}$. Recall that $h_x=\Sigma^\infty \Map_{\Sp}(x, -)$. By \Cref{Cor:conv-representables} there is a natural equivalence
	\[
		h_{\num{i}_+\wedge \bbS}\circledast h_{\num{j}_+\wedge \bbS} \xrightarrow{\simeq} h_{\num{i}\times\num{j}_+\wedge \bbS}.
	\]
	It follows that given, say, an $s\in \num{i}$, the action of $h_{\psi_{\num{i}\setminus\{s\}}}$ on $h_{\num{i}_+\wedge \bbS}\circledast h_{\num{j}_+\wedge \bbS}$ can be identified with the action of $h_{\psi_{(\num{i}\setminus\{s\})\times \num{j}}}$ on $h_{\num{i}\times\num{j}_+\wedge \bbS}$. Here 
	\[
		\psi_{(\num{i}\setminus\{s\})\times \num{j}}\colon \num{i}\times \num{j}_+ \wedge\bbS \to \num{i}\times \num{j}_+ \wedge\bbS
	\]
	is the idempotent that collapses $\{s\}\times \num{j}$ to the basepoint.

	It follows that $h_{\bbS}(i) \circledast  h_{\bbS}(j)$ is equivalent to
	\[
		\underset{s\in \num{i}}{\bigcirc} (1-h_{\psi_{(\num{i}\setminus\{s\})\times \num{j}}})\circ \underset{t\in \num{j}}{\bigcirc}(1-h_{\psi_{\num{i}\times (\num{j}\setminus\{t\})}})h_{\num{i}\times \num{j}_+ \bbS}.
	\]
	By formula~\eqref{eq:product-to-sum} of \Cref{lem:product-to-sum}, this is equivalent to
	\[
		\sum_{U} \crosseffect(U)h_{\num{i}\times\num{j}_+\wedge \bbS}
	\]
	where $U$ ranges over subsets $U\subseteq \num{i}\times\num{j}$ that are not contained in a subset of the form $(\num{i}\setminus \{s\}) \times \num{j}$ or $\num{i}\times (\num{j}\setminus\{t\})$. But these are precisely the good subsets of $\num{i}\times \num{j}$, and for each good subset $U$, $\crosseffect(U)h_{\num{i}\times\num{j}_+\wedge \bbS}$ is precisely $h_{\bbS}(|U|)$.
\end{proof}

\begin{Exa}
	We have that 
	\[
		h_{\bbS}(2) \circledast  h_{\bbS}(2) \simeq h_{\bbS}(2)^{\oplus 2} \oplus  h_{\bbS}(3)^{\oplus 4} \oplus h_{\bbS}(4). 
	\]
\end{Exa}

As a consequence of \eqref{eq:smash decomposition}, we obtain the following.

\begin{Lem}\label{lem:hSi-is-algebra}
    For each $i \ge 1$ the functor $h_{\bbS}(i)$ is a commutative algebra in $\Fun(\Sp^c,\Sp)$. 
\end{Lem}

\begin{proof}
	We recall that commutative algebras with respect to Day convolution are exactly the lax symmetric monoidal functors, as proven in the $\infty$-categorical setting in \cite[Example 2.2.6.9]{HALurie} or \cite[Proposition 2.12]{Glasman16}. If $x$ is a cocommutative coalgebra spectrum, then the functor $h_x$ is lax symmetric monoidal via the composition
	\[\begin{tikzcd}[column sep=tiny]
		h_x(y_1)\otimes h_x(y_2)&=&
		\Sigma^{\infty}\Map_{\Sp}(x,y_1)\otimes \Sigma^{\infty}\Map_{\Sp}(x,y_2)\ar[d]\\ 
		&&\Sigma^\infty \Map_{\Sp}(x\otimes x, y_1\otimes y_2)\ar[d]\\
		&& \Sigma^\infty \Map_{\Sp}(x, y_1\otimes y_2)\ar[r,phantom,"="] &h_x(y_1\otimes y_2).
	\end{tikzcd}\]
	In particular, it follows that the functor $h_{[i]_+\wedge\bbS}$ is lax symmetric monoidal, and thus a commutative algebra with respect to Day convolution. Furthermore, if we have an inclusion of finite sets $U\hookrightarrow V$, the induced map $U_+\wedge \bbS\to V_+\wedge \bbS$ is a map of cocommutative coalgebra spectra. It follows that the induced natural transformation $h_{V_+\wedge \bbS}\to h_{U_+\wedge \bbS}$ is symmetric monoidal. This implies that the total fiber of the contravariant cubical diagram 
	\[
		U\mapsto h_{U_+\wedge \bbS}
	\]
	as $U$ ranges over the subsets of $[i]$, being a limit of a diagram of lax symmetric monoidal functors and symmetric monoidal natural transformations, is lax symmetric monoidal. This total fiber is equivalent to $h_{\bbS}(i)$ and thus we conclude that $h_{\bbS}(i)$ is a commutative algebra with respect to Day convolution.
\end{proof}

\begin{Cor}\label{lem:commutative-algebra-structure}
    Each $P_dh_{\bbS}(i)$ is a commutative algebra in $\Exc{d}(\Sp^c,\Sp)$. 
\end{Cor}

\begin{proof}
	This follows from \cref{lem:hSi-is-algebra} since $P_d\colon \Fun(\Sp^c,\Sp)\to\Exc{d}(\Sp^c,\Sp)$ is symmetric monoidal by \cref{thm:properties-n-excisive}.
\end{proof}

The following lemma describes the multiplication on $h_{\bbS}(i)$ in terms of the splitting of \Cref{prop:day-convolution}.

\begin{Lem}
	The multiplication map $h_{\bbS}(i)\circledast h_{\bbS}(i)\to h_{\bbS}(i)$ corresponds, under the splitting of \Cref{prop:day-convolution}, to projection onto the summand corresponding to the diagonal $[i]\subseteq [i]\times [i]$.  
\end{Lem}

\begin{proof}
	The map  $h_{\bbS}(i)\circledast h_{\bbS}(i)\to h_{\bbS}(i)$ is a retract of the restriction map
	\begin{equation}\label{eq:restriction}
		h_{[i]\times [i]_+\wedge \bbS}\to h_{[i]_+\wedge \bbS}
	\end{equation}
	induced by the diagonal inclusion $[i]\hookrightarrow [i]\times [i]$. By \Cref{cor:splitting}, there are equivalences
	\[
		h_{[i]\times [i]_+\wedge \bbS} \simeq \prod_{U\subset [i]\times [i]} h_{\bbS}(U)
	\]
	and 
	\[
		h_{[i]_+\wedge \bbS} \simeq \prod_{U\subset [i]} h_{\bbS}(U).
	\]
	It follows from the proof of \Cref{cor:splitting} that under these splittings the restriction map~\eqref{eq:restriction} sends a summand $h_{\bbS}(U)$ of $h_{[i]\times [i]_+\wedge \bbS}$ to itself if $U\subseteq [i]\times [i]$ is a subset of the diagonal, and sends $h_{\bbS}(U)$ to zero otherwise. By \Cref{prop:day-convolution}, $h_{\bbS}(i)\circledast h_{\bbS}(i)$ is the wedge of those summands of $h_{[i]\times [i]_+\wedge \bbS}$ which correspond to good subsets of $[i]\times [i]$. The only good subset of $[i]\times [i]$ that is contained in the diagonal is the diagonal itself, and it gets mapped to $h_{\bbS}(i)$.
\end{proof}

\begin{Rem}
	An algebra $A$ in a symmetric monoidal category $\cat C$ is {\it separable} if the multiplication map $A\otimes A\to A$ has an $A$-$A$-bilinear section. Separable algebras have been studied by Balmer \cite{Balmer11} in the context of additive and triangulated categories motivated by the close connection between commutative separable algebras and the notion of \'{e}tale morphism in tensor triangular geometry; see \cite{Balmer16b,BalmerDellAmbrogioSanders15,Sanders22}. Recent work of Ramzi \cite[Theorem~B]{ramzi2023separability} shows that a commutative separable algebra in the homotopy category of  an additive symmetric monoidal $\infty$-category $\cat C$ admits an essentially unique lift to a commutative algebra object in $\cat C$ itself. 
\end{Rem}

\begin{Lem}\label{lem:hSi-is-separable}
	The commutative algebra structure on $h_{\bbS}(i)$ is separable.    
\end{Lem}

\begin{proof}
	We begin by noting that the algebra $h_{[i]_+\wedge \bbS}$ is separable. Indeed, the algebra structure is induced by the coalgebra structure on the spectrum $[i]_+\wedge \bbS$ which is determined by the diagonal map $[i]_+\wedge \bbS\to [i]\times [i]_+\wedge \bbS$. This map has a retraction $[i]\times [i]_+\wedge \bbS\to [i]_+ \wedge \bbS$. It is induced by the map of pointed sets $[i]\times [i]_+\to [i]_+$ which sends $(x, y)$ to $x$ when $x=y\in [i]$ and sends $(x, y)$ to the basepoint if $x\ne y$. It is straightforward to check that this retraction is a map of bi-comodules over $[i]_+\wedge \bbS$. The retraction induces a map $h_{[i]_+\wedge \bbS}\to h_{[i]\times [i]_+\wedge \bbS}$ which is the desired section of the multiplication map.

	Next, consider the multiplication map $h_{\bbS}(i)\circledast h_{\bbS}(i)\to h_{\bbS}(i)$. By the proof of \Cref{lem:hSi-is-algebra}, this multiplication can be identified with a map between homotopy fibers of the following form
	\[
		\tofib h_{U\times V_+ \wedge\bbS} \to \tofib h_{W_+ \wedge\bbS}
	\]
	where $U, V, W$ range over subsets of $[i]$, and the map is induced by the diagonal functor $W\mapsto W\times W$. This map has a section
	\[
		\tofib h_{W_+ \wedge\bbS}\to \tofib h_{U\times V_+ \wedge\bbS}.
	\]
	The section is induced by the functor from the poset of subsets of $[i]\times [i]$ to the poset of subsets of $[i]$ that sends $(U, V)$ to $U\cap V$, and the collapse maps $U\times V_+ \to (U\cap V)_+$. It is straightforward to check that the section is well-defined and that it is indeed a map of bimodules as required.
\end{proof}

\section{Compact generation of the category of \texorpdfstring{$d$}{d}-excisive functors}\label{sec:compact-generation}

In this section we prove that the functors $P_d h_{\bbS}(i)$, $1 \le i \le d$, form a set of compact generators for $\Exc{d}(\Sp^c,\Sp)$. To this end, we first record some features of cross-effects of excisive functors.

\begin{Rem}\label{Rem:vanishing-cross-effects}
	Recall that the $m$-th cross-effect $\crosseffect_m(F)$ of a functor $F$ is a functor of $m$ variables (see \cref{sec:cross-idempotents}). If $F$ is $d$-excisive then for each $1 \le m \le d+1$, the cross-effect $\crosseffect_{m}(F)$ is $(d-m+1)$-excisive in each variable; see \cite[Proposition~3.3]{Goodwillie03} or \cite[Proposition~6.1.3.22]{HALurie}. In particular, the $d$-th cross-effect of a $(d-1)$-excisive functor is trivial. Therefore, by applying $\crosseffect_d$ to the fiber sequence $D_dF \to P_dF \to P_{d-1}F$, we obtain
	\[
		\crosseffect_d(D_dF) \simeq \crosseffect_d(P_dF)
	\]
	for any functor $F$.
\end{Rem}

\begin{Rem}\label{rem:rep-rem}
	Recall that $h_x$ is the corepresentable functor $h_x(y)=\Sigma^\infty \Map_{\Sp}(x, y)$ and that the $i$-th cross-effect of the contravariant functor $x\mapsto h_x$ is equivalent to $h_{x_1}\otimes \cdots \otimes h_{x_i}$ (see \cref{sec:cross-rep}, in particular \cref{rem:h_x(i)}). The following lemma is well-known; for example, see \cite[Lemma 2.11]{Ching21}:
\end{Rem}

\begin{Lem}\label{lem:ching-yoneda}
    Let $F \in \Fun(\Sp^c,\Sp)$ and $x_1,\ldots,x_i \in \Sp^c$. There is an equivalence
    \[
		\Hom_{\Fun(\Sp^c,\Sp)}(h_{x_1}(-) \otimes \cdots \otimes h_{x_i}(-),F(-)) \simeq \crosseffect_iF(x_1,\ldots,x_i).
    \]
\end{Lem}

In $d$-excisive functors, this implies:

\begin{Lem}\label{lem:cross-effect-yoneda-lemma}
	Let $F \in \Exc{d}(\Sp^c,\Sp)$, $1\le i \le d$ and $x_1,\ldots,x_i \in \Sp^c$. There is an equivalence
	\[
		\Hom_{\Exc{d}(\Sp^c,\Sp)}(P_d(h_{x_1}(-) \otimes \cdots \otimes h_{x_i}(-)),F(-)) \simeq \crosseffect_iF(x_1,\ldots,x_i). 
	\]
\end{Lem}

\begin{proof}
	Since $F$ is $d$-excisive, there is an equivalence
	\[
	\begin{split}
		\Hom_{\Exc{d}(\Sp^c,\Sp)}(P_d(h_{x_1}(-) &\otimes \cdots \otimes h_{x_i}(-)),F(-))  \\&\simeq \Hom_{\Fun(\Sp^c,\Sp)}(h_{x_1}(-) \otimes \cdots \otimes h_{x_i}(-),F(-))
	\end{split}
	\]
	and the lemma then follows from \cref{lem:ching-yoneda}.
\end{proof}

\begin{Rem}\label{rem:dict1}
    In terms of the analogy with equivariant stable homotopy theory, we may regard the functors $\crosseffect_i(-)(\bbS,\ldots,\bbS)\colon \Exc{d}(\Sp^c,\Sp)\to\Sp$ as functor calculus analogs of the categorical fixed point functors $(-)^H\colon \Sp_G \to \Sp$ and we think of the functors $P_d h_{\bbS}(i)\in\Exc{d}(\Sp^c,\Sp)$ as analogous to the $G$-spectra $\Sigma_G^\infty G/H_+ \in \Sp_G$. From this perspective, \Cref{lem:ching-yoneda} is a functor calculus version of the fact that the $G$-spectrum $\Sigma_G^\infty G/H_+$ represents $H$-fixed points in $G$-equivariant stable homotopy theory. See \Cref{appendix} for a dictionary between functor calculus and equivariant stable homotopy theory.
\end{Rem}

\begin{Exa}[The first cross-effect]\label{exa:first-cross-effect}
	Let $\crosseffect_1(-)(\bbS) \colon \Exc{d}(\Sp^c,\Sp) \to \Sp$ be the functor that takes the first cross-effect and evaluates it at $\bbS$. Since our $\Exc{d}(\Sp^c,\Sp)$ consists of reduced functors, this is the same as just evaluating a functor at $\bbS$. We observe that this functor is also right adjoint to the inflation functor $i_d \colon \Sp \to \Exc{d}(\Sp^c,\Sp)$ of \Cref{def:inflation}:
\end{Exa}

\begin{Lem}
	The functor $\crosseffect_1(-)(\bbS)$ is right adjoint to $i_d \colon \Sp \to \Exc{d}(\Sp^c,\Sp)$. In particular, the functor $\crosseffect_1(-)(\bbS)$ is lax symmetric monoidal. 
\end{Lem}

\begin{proof}
	Let $\alpha \colon \Exc{d}(\Sp^c,\Sp) \to \Sp$ be the right adjoint to $i_d$, which exists by the adjoint functor theorem. Note that we have
	\[
		\Hom_{\Exc{d}(\Sp^c,\Sp)}(i_d\bbS,F) \simeq \Hom_{\Sp}(\bbS,\alpha(F)) \simeq \alpha(F)
	\]
	by adjunction and the Yoneda lemma. Since $i_d\bbS \simeq P_dh_{\bbS}$ (as $i_d$ is symmetric monoidal), this shows that $\alpha$ is corepresented by $P_dh_{\bbS}$. By \Cref{lem:ching-yoneda} this identifies~$\alpha$ with $\crosseffect_1(-)(\bbS)$.  Finally, as the right adjoint of a symmetric monoidal functor, $\crosseffect_1(-)(\bbS)$ is lax symmetric monoidal \cite[Corollary 7.3.2.7]{HALurie}.
\end{proof}

\begin{Rem}
	This matches the perspective of \cref{rem:dict1} in which the categorical fixed point functor $(-)^G$ is right adjoint to the canonical geometric functor $\Sp\to\Sp_G$ which equips a spectrum with the ``trivial'' $G$-action.
\end{Rem}

\begin{Def}
	For any $F\colon \Sp^c\to\Sp$ and $i \ge 1$, let $c_iF\colon \Sp^c \to \Sp$ denote the composite
	\[
		\Sp^c \xrightarrow{\Delta} (\Sp^c)^{\times i} \xrightarrow{\crosseffect_i F} \Sp
	\]
	where $\Delta$ is the diagonal. That is, $c_i F$ is the functor $X \mapsto \crosseffect_i F(X,\ldots,X)$. In particular, note that there is an equivalence of functors $c_ih_x=h_x(i)$ by \cref{rem:rep-rem}.
\end{Def}

\begin{Lem}\label{lem:n-homogeneous-functors}
	If $F \colon \Sp^c \to \Sp$ is $d$-excisive, then $c_d F\colon \Sp^c \to \Sp$ is $d$-homogeneous.
\end{Lem}

\begin{proof}
	This follows from Propositions 3.1 and 3.3 of \cite{Goodwillie03}. 
\end{proof}

\begin{Lem}\label{lem:ci-pn-commute}
	For each $1 \le i \le d$, we have $P_dc_iF \simeq c_iP_dF$.
\end{Lem}

\begin{proof}
	In the case $i = d$ this is shown in the proof of \cite[Theorem 6.1]{Goodwillie03} (in Goodwillie's notation $\lambda \crosseffect_d$ is what we write as $c_d$); the same proof works for arbitrary $i$. Indeed, $c_iF(x)$ is the total fiber of the $i$-dimensional cubical diagram that sends a subset $U\subseteq \num{i}$ to $F(\bigoplus_{\num{i}\setminus U} x)$. The functor $P_d$ commutes with finite limits and in particular commutes with total fibers of cubical diagrams~\cite[Proposition~1.7]{Goodwillie03}.
\end{proof}

As a consequence of \Cref{lem:ci-pn-commute} we obtain the following. 

\begin{Lem}\label{lem:cn-functor}
	For each $1 \le i \le d$, there are natural equivalences 
	\[
		c_iP_dh_{\bbS} \simeq P_dc_ih_{\bbS} \simeq P_dh_{\bbS}(i)
	\]
	of functors $\Sp^c \to \Sp$.
\end{Lem}

\begin{Prop}\label{prop:pnhsn-homogeneous}
	The functor $P_dh_{\bbS}(d)$ is $d$-homogeneous.
\end{Prop}

\begin{proof}
	Since $P_dh_{\bbS}$ is $d$-excisive, \cref{lem:n-homogeneous-functors} implies that $c_dP_dh_{\bbS}$ is \mbox{$d$-homogeneous.} The result then follows from \Cref{lem:cn-functor}. 
\end{proof}

\begin{Cor}\label{cor:value-of-pnhsn}
	We have $P_dh_{\bbS}(d)(X) \simeq X^{\otimes d}$ where $\Sigma_d$ acts on $X^{\otimes d}$ by permuting the factors. 
\end{Cor}

\begin{proof}
	We have equivalences 
	\[
		P_dh_{\bbS}(d)(X) \simeq c_dP_dh_{\bbS}(X) \simeq \crosseffect_d(P_dh_{\bbS})(X,\ldots,X)\simeq \crosseffect_d(D_dh_{\bbS})(X,\ldots,X).
	\]
	The functor $h_{\bbS}$ is the functor $\Sigma^{\infty}\Omega^{\infty}$ whose Goodwillie tower is well-known, see for example \cite[Corollary~1.3]{AhearnKuhn2002}. In particular, we have $D_dh_{\bbS}(X) \simeq (X^{\otimes d})_{h\Sigma_d}$. Computing $d$-th cross-effects (for example, by a minor modification of the method used in the proof of \cite[Proposition 5.2]{AroneChing15}), we obtain the desired result. 
\end{proof}

\begin{Cor}\label{cor:vanishing}
	If $u > d$, then $P_dh_{\bbS}(u) = 0$.
\end{Cor}

\begin{proof}
	Note that $P_dh_{\bbS}(u) \simeq P_dP_uh_{\bbS}(u)$, but $P_uh_{\bbS}(u)$ is $u$-homogeneous (\cref{prop:pnhsn-homogeneous}). 
\end{proof}

For the next statement, recall the notion of a ``good'' subset from \cref{def:good-subsets}.

\begin{Prop}\label{prop:local-day-convolution}
	For any $i,j\ge 1$, there is an equivalence
	\begin{equation}\label{eq:mackey-formula-goodwillie}
		P_dh_{\bbS}(i) \circledast P_dh_{\bbS}(j) \simeq \bigoplus_{\substack{\mathcal{U} \subseteq \num{i} \times \num{j} \\ \
		\mathcal{U} \text{ good} \\|\mathcal{U}| \le d}} P_dh_{\bbS}(|\mathcal{U}|)
	\end{equation}
	in $\Exc{d}(\Sp^c,\Sp)$. 
\end{Prop}

\begin{proof}
	This follows from \Cref{prop:day-convolution} and \Cref{cor:vanishing}.
\end{proof}

\begin{Rem}
    We view this proposition as a functor calculus version of the Mackey decomposition
    \[
		\Sigma_G^{\infty}G/H_+ \otimes \Sigma_G^{\infty}G/K_+ \simeq \bigoplus_{g \in H \backslash G/K} \Sigma_G^{\infty}G/(H^g \cap K)_+
    \]
    that holds in the category of $G$-spectra. 
\end{Rem}

\begin{Exa}\label{exa:day-convolution-n}
	We have
	\[
		P_dh_{\bbS}(i) \circledast P_dh_{\bbS}(d) \simeq \bigoplus_{\left|\surj(d,i)\right|}P_dh_{\bbS}(d),
	\]
	where $\left|\surj(d,i)\right|$ denotes the number of surjections from a set of cardinality $d$ to a set of cardinality $i$. This follows from the observation that the only summands in \eqref{eq:mackey-formula-goodwillie} correspond to those good subsets of $[i] \times [d]$ of cardinality exactly $d$, and there are exactly ${|\surj(d,i)|}$ of them. 
\end{Exa}

\begin{Rem}
	It is useful to recall, for example using inclusion-exclusion (or alternatively using the table on page 73 along with Equation 1.94(a) of \cite{Stanley12}), that
	\begin{equation}\label{eq:number_or_surjections}
		\left|\surj(i,j)\right| \cong j^i+\sum_{s=1}^{j-1}(-1)^s{\binom{j}{s}}(j-s)^i.
	\end{equation}
\end{Rem}

\begin{Rem}
	The $G$-spectra $\Sigma_G^\infty G/H_+$ are dualizable and in fact (for finite $G$) self-dual in the category of $G$-spectra. Our next goal is to show that the same holds for their functor calculus analogs $P_dh_{\bbS}(i)$; see \cref{rem:dict1}. In fact, this holds before applying $P_d$. 
\end{Rem}

\begin{Prop}\label{prop:self-dual}
    For each $i \ge 1$ the functor $h_{\bbS}(i)$ is a self-dual dualizable object of $\Fun(\Sp^c,\Sp)$.
\end{Prop}

\begin{proof}
	We first establish that $h_{\bbS}(i)$ is dualizable in $\Fun(\Sp^c,\Sp)$ for each $i \ge 1$. Indeed, by \Cref{rem:stable-yoneda} if $x \in \Sp^c$ is dualizable, then $h_{x} \in \Fun(\Sp^c,\Sp)$ is dualizable, and the dual of $h_{x}$ is $h_{\mathbb {D}x}$, where $\mathbb D x$ is the Spanier--Whitehead dual of $x$. In particular, if $x\in \Sp^c$ is self-dual, then $h_x$ is self-dual.  This implies that $h_{\oplus_i \bbS}$ is self-dual, and in particular $h_{\bbS}(i)$ is dualizable, as it is a summand of a dualizable object.

	Moreover, we will show that $h_{\bbS}(i)$ is itself self-dual. Indeed, suppose $F$ is a dualizable object of a symmetric monoidal stable $\infty$-category and $e\colon F\to F$ is an idempotent map. Recall that $eF$ is the summand of $F$ split off by $e$. Then the object $eF$ is dualizable, and furthermore the dual of $eF$ is equivalent to $\mathbb De(\mathbb DF)$, i.e., the summand of $\mathbb DF$ split off by the dual idempotent $\mathbb De$. In particular, if $F$ and $e$ are both self-dual, then~$eF$ is self-dual. Now recall from \eqref{eq:idempotent h_S} that $h_{\bbS}(i)$ is split off the self-dual object~$h_{\oplus_i \bbS}$ by the idempotent 
    \[
        \crosseffect\num{i}=\underset{1\le s\le i}{\bigcirc}(1-h_{\psi_{\num{i}\setminus\{s\}}}).
    \] 
    Here $h_{\psi_{\num{i}\setminus\{s\}}}\colon h_{\oplus_i \bbS}\to h_{\oplus_i \bbS}$ is the idempotent induced by the map $\oplus_i \bbS\to \oplus_i \bbS$ that collapses summand $s$ to a point. This map is self-dual and therefore $h_{\psi_{\num{i}\setminus\{s\}}}$ is self-dual. It follows that $\crosseffect\num{i}$ is self-dual and thus $h_{\bbS}(i)$ is self-dual. 
\end{proof}

\begin{Rem}
	In fact, since the commutative algebra $h_{\bbS}(i)$ is separable (\cref{lem:hSi-is-separable}), dualizability of $h_{\bbS}(i)$ actually forces $h_{\bbS}(i)$ to be self-dual; see \cite[Section~2]{Sanders22}.
\end{Rem}

\begin{Cor}\label{prop:generators-dualizable-self-dual}
    Each $P_dh_{\bbS}(i)$ is a self-dual dualizable object of $\Exc{d}(\Sp^c,\Sp)$.
\end{Cor}

\begin{proof}
	This follows from \cref{prop:self-dual} and \cref{thm:properties-n-excisive} since symmetric monoidal functors preserve dualizable objects and their duals.
\end{proof}

\begin{Rem}
	In equivariant homotopy theory, restricting to the trivial subgroup $\res^G_1\colon \Sp_G \to \Sp$ is strong symmetric monoidal, whereas $(-)^H\colon \Sp_G \to \Sp$ is only lax symmetric monoidal for $e \neq H \le G$. The corresponding result in functor calculus is the following: 
\end{Rem}

\begin{Prop}\label{prop:monoidal-cross-effect}
	The functor 
	\[
		\crosseffect_d(-)(\bbS,\ldots,\bbS) \colon \Exc{d}(\Sp^c,\Sp) \to \Sp
	\]
	admits a symmetric monoidal refinement. 
\end{Prop}

\begin{proof}
	The following is essentially the proof given for the $d=2$ case on page 108 of \cite{CDHHLMNNS23}. Let $\Delta \colon \Sp^c \to (\Sp^c)^{\times d}$ denote the diagonal functor, and consider the induced functor $\Delta^{\ast}\colon \Fun((\Sp^c)^{\times d},\Sp) \to \Fun(\Sp^c,\Sp)$.  This functor admits a left adjoint, via the left Kan extension along $\Delta$. Since the diagonal is symmetric monoidal, so is the left Kan extension. Moreover, since the diagonal has a two-sided adjoint given by the direct sum functor $q \colon (\Sp^c)^{\times d} \to \Sp$, this left Kan extension is just given by restriction along $q^*$.   

	We note that by \cite[Corollary 6.1.3.5]{HALurie} the functor $\Delta^*$ restricts to a functor $\Exc{(1,\ldots,1)}((\Sp^c)^{\times d},\Sp) \to \Exc{d}(\Sp^c,\Sp)$. Moreover, as in \Cref{exa:1-excisive-functors,exa:inflation} we can identify
	\begin{equation}\label{eq:multivariable-d}
		\Exc{(1,\ldots,1)}((\Sp^c)^{\times d},\Sp) \simeq \Sp
	\end{equation}
	as symmetric monoidal $\infty$-categories, given by evaluation at $(\mathbb{S},\ldots,\mathbb{S})$. Using the universal property of the localization $P_{d}$ (see \cite[Proposition 5.5.4.20]{HTTLurie}), there then exists a symmetric monoidal functor $\overline{q} \colon \Exc{d}(\Sp^c,\Sp) \to \Exc{(1,\ldots,1)}((\Sp^c)^{\times d},\Sp)$ which fits in the commutative diagram:
	\[\begin{tikzcd}[ampersand replacement=\&]
		{\Fun(\Sp^c,\Sp)} \& {\Exc{d}(\Sp^c,\Sp)} \\
		{\Exc{(1,\ldots,1)}((\Sp^c)^{\times d},\Sp)} \\
		\Sp
		\arrow["{P_d}", from=1-1, to=1-2]
		\arrow["{p_{1,\ldots,1}\circ q^*}"', from=1-1, to=2-1]
		\arrow["{\overline{q}}", from=1-2, to=2-1]
		\arrow["{\ev_{(\bbS,\ldots,\bbS)}}", "\sim"', from=2-1, to=3-1]
		\arrow[curve={height=-18pt}, dashed, from=1-2, to=3-1]
	\end{tikzcd}\]
	The composite ${p_{1,\ldots,1}\circ q^*}$ is the $d$-th cross-effect by \cite[Remark 6.1.3.23]{HALurie}. Unwinding the definitions, we see that the dashed arrow $\Exc{d}(\Sp^c,\Sp) \to \Sp$ is a symmetric monoidal functor whose underlying functor is exactly $\crosseffect_d(-)(\bbS,\ldots,\bbS)$. 
\end{proof}

\begin{Rem}
	For $i <d$ the functor 
	\[
		\crosseffect_i(-)(\bbS,\ldots,\bbS) \colon \Exc{d}(\Sp^c,\Sp) \to \Sp
	\]
	is \emph{lax} symmetric monoidal (for example, because it is corepresented by the commutative algebra object $P_dh_{\bbS}(i)$) but not symmetric monoidal. For example, one has $\crosseffect_1P_dh_{\bbS}(d) \simeq \mathbb{S}$ but $\crosseffect_1(P_dh_{\bbS}(d) \circledast P_dh_{\bbS}(d)) \simeq \mathbb{S}^{\oplus d!}$ as follows from \Cref{exa:day-convolution-n}. 
\end{Rem}

\begin{Rem}
	In equivariant homotopy theory, the categorical fixed point functors $\{(-)^H\}_{H \le G}$ are jointly conservative on  $G$-spectra. Our next goal is to show that the same is true for the collection $\{ \crosseffect_i(-)(\bbS,\ldots,\bbS) \}_{1 \le i \le d}$ on the category of $d$-excisive functors. We begin by working with $d$-homogeneous functors (\Cref{Def:n-homogeneous}). 
\end{Rem}

\begin{Lem}\label{lem:cross-effects-homogeneous}
	Let $F$ be a $d$-homogeneous functor. Then 
	\[
		F = 0 \iff \crosseffect_dF(\bbS,\ldots,\bbS) = 0.
	\]
\end{Lem}

\begin{proof}
    By \cite[Proposition 3.4]{Goodwillie03} (see also \cite[Corollary 6.1.4.11]{HALurie}) we have that $F = 0 \iff \crosseffect_dF(X_1,\ldots,X_d)= 0$ for all $X_1,\ldots,X_d \in \Sp^c$. Thus, it suffices to show that
    \[
		\crosseffect_dF(\bbS,\ldots,\bbS) = 0 \implies \crosseffect_dF(X_1,\ldots,X_n) = 0.
    \]
    We first show that 
    \[
		\crosseffect_dF(\bbS,\ldots,\bbS) = 0 \implies \crosseffect_dF(X_1,\bbS,\ldots,\bbS) = 0.
    \]
    Indeed, the collection $\cat C \coloneqq \SET{X \in \Sp^c }{ \crosseffect_dF(X,\bbS,\ldots,\bbS) = 0}$ is a thick subcategory (this uses that $F$ is reduced to show that it is closed under extensions) which contains~$\bbS$ by assumption. Therefore, $\cat C = \Sp^c$. Repeating the argument in each variable then gives the desired result. 
\end{proof}

\begin{Prop}\label{prop:cross-effect-conservative}
	The functor 
	\[
		\prod_{1 \le i \le d} \crosseffect_i(-)(\bbS,\ldots,\bbS) \colon \Exc{d}(\Sp^c,\Sp) \to \prod_{1 \le i \le d} \Sp
	\]
	is conservative. In other words, a functor $F \in \Exc{d}(\Sp^c,\Sp)$ is trivial if and only if $\crosseffect_i(F)(\bbS,\ldots,\bbS) = 0$ for $1 \le i \le d$. 
\end{Prop}

\begin{proof}
	We prove this by induction on $d$. For $d = 1$, under the identification $\Exc{1}(\Sp^c,\Sp) \simeq \Sp$ of \Cref{exa:1-excisive-functors}, $\crosseffect_1(-)(\bbS)$ is the identity functor, and so the claim becomes a tautology. Next assume inductively that the statement is true for $(d-1)$-excisive functors, and let $F$ be an $d$-excisive functor whose cross-effects vanish. Recall that we have a fiber sequence
	\[
		D_dF \to F \to P_{d-1}F
	\]
	where $D_dF$ is $d$-homogeneous and $P_{d-1}F$ is $(d-1)$-excisive. By \Cref{Rem:vanishing-cross-effects} and our hypothesis we have $\crosseffect_d(D_dF) \simeq \crosseffect_d(F) = 0$, so that $D_dF = 0$ by \Cref{lem:cross-effects-homogeneous}. In particular, $F \simeq P_{d-1}F$ and $\crosseffect_iP_{d-1}F(\bbS,\ldots,\bbS) = 0$ for $1 \le i \le d-1$. By induction, $F \simeq P_{d-1}F = 0$, and we are done. 
\end{proof}

\begin{Thm}\label{prop:compact_generation}
	The category $\Exc{d}(\Sp^c,\Sp)$ is compactly generated. Moreover, the functors $P_d(h_{\bbS}(i))$ for $1 \le i \le d$ form a set of compact generators of $\Exc{d}(\Sp^c,\Sp)$. 
\end{Thm}

\begin{proof}
 We recall that Ching has shown that $\Fun(\Sp^c,\Sp)$ is compactly generated (\Cref{rem:rigid-compact-fun}). By \Cref{rem:pd}, $P_d$ is a smashing localization, so \cite[Corollary 5.5.7.3]{HTTLurie} implies that the category $\Exc{d}(\Sp^c,\Sp)$ is also compactly generated. In fact, the same corollary shows that the functor $P_d$ preserves compact objects and that a functor $G \in \Exc{n}(\Sp^c,\Sp)$ is compact if and only if there exists $F \in \Fun(\Sp^c,\Sp)^c$ such that $G$ is a retract of~$P_dF$. 

	We now show that $\SET{P_d(h_{\Sphere}(i))}{1\le i\le d}$ is a set of compact generators. Since the objectwise smash product of compact functors in $\Fun(\Sp^c,\Sp)$ is still compact, each $h_{\bbS}(i)$ is compact in $\Fun(\Sp^c,\Sp)$ and hence is compact in $\Exc{d}(\Sp^c,\Sp)$ after applying $P_d$, as explained above. To show that this set of compact objects is a set of generators, it suffices by \cite[Lemma 2.2.1]{SchwedeShipley03} to show that if
	\[
		\Hom_{\Exc{d}(\Sp^c,\Sp)}(P_d(h_{\mathbb{S}}(i)),F) = 0 \quad \text{ for all } 1 \le i \le d
	\]
	then $F = 0$. Applying \Cref{lem:cross-effect-yoneda-lemma}, we see that this is equivalent to the statement that if $\crosseffect_iF(\bbS,\ldots,\bbS) = 0$ for $1 \le i \le d$, then $F = 0$, which is precisely the content of \Cref{prop:cross-effect-conservative}. 
\end{proof}

\begin{Rem}\label{rem:BCR}
    Ideas related to the proof of \cref{prop:compact_generation} have also appeared in work of Biedermann--Chorny--R\"ondigs \cite[Section~8]{BiedermannChornyRoendigs2007}.
\end{Rem}

\begin{Cor}\label{cor:rigid-compact}
	The category $\Exc{d}(\Sp^c,\Sp)$ is rigidly-compactly generated; that is, it is compactly generated and the compact and dualizable objects coincide. 
\end{Cor}

\begin{proof}
	We have already seen in \Cref{prop:compact_generation} that $\Exc{d}(\Sp^c,\Sp)$ is compactly generated by $\SET{P_d(h_{\bbS}(i))}{1\le i \le d}$. Moreover, each of these compact generators is dualizable by \Cref{prop:generators-dualizable-self-dual}. In the language of \cite{HoveyPalmieriStrickland97}, we have shown that $\Exc{d}(\Sp^c,\Sp)$ is a unital algebraic stable homotopy category. It then follows from \cite[Theorem 2.1.3(d)]{HoveyPalmieriStrickland97} that the compact and dualizable objects coincide.
\end{proof}

\begin{Rem}\label{rem:adjoint-to-cross-effects}
	We can also deduce from the proof of \Cref{prop:monoidal-cross-effect} that 
	\[
		\crosseffect_d(-)(\bbS,\ldots,\bbS) \colon \Exc{d}(\Sp^c,\Sp) \to \Sp
	\]
	has equivalent left and right adjoints, given by the composite $\Sp \xrightarrow{i} \Exc{(1,\ldots,1)} \xrightarrow{\Delta^*} \Exc{d}(\Sp^c,\Sp)$, i.e., sending a spectrum $A$ to the $d$-excisive functor $X \mapsto A \otimes X^{\otimes d}$. (This follows since the diagonal $\Delta \colon \Sp^c \to (\Sp^c)^{\times d}$ has equivalent left and right adjoints.) Let us denote this adjoint by $\kappa_d$. We then obtain the following, which is the functor calculus analog of \cite[Theorem 1.1]{BalmerDellAmbrogioSanders15}:
\end{Rem}

\begin{Thm}\label{thm:barr-beck}
	There is an equivalence of symmetric monoidal $\infty$-categories
	\[
		\Mod_{\Exc{d}(\Sp^c,\Sp)}\bigl(P_dh_{\bbS}(d)\bigr) \simeq \Sp
	\]
	under which $\crosseffect_d(-)(\bbS,\ldots,\bbS)$ is identified with extension of scalars along the commutative algebra $P_dh_{\bbS}(d)$. That is, we have a commutative diagram
	\[\begin{tikzcd}[column sep=tiny]
		& \Exc{d}(\Sp^c,\Sp) \ar[dl,shift right,"{\crosseffect_d(-)(\bbS,\ldots,\bbS)}"'] \ar[dl,leftarrow,shift left,"\kappa_d"]
		\ar[dr,shift right,end anchor=north west,"F"']\ar[dr,shift left,leftarrow,end anchor=north west,"U"]& \\
		\Sp \ar[rr,"\simeq"]& &  \Mod_{\Exc{d}(\Sp^c,\Sp)}\bigl(P_dh_{\bbS}(d)\bigr)
	\end{tikzcd}\]
	where $F$ and $U$ denote extension and restriction of scalars along $P_dh_{\bbS} \to P_dh_{\bbS}(d)$.
\end{Thm}

\begin{proof}
	Note that by \Cref{cor:value-of-pnhsn} the functor $\kappa_d$ sends the unit $\mathbb{S}$ to~$P_dh_{\bbS}(d)$. Moreover, $\crosseffect_d(-)(\bbS,\ldots,\bbS)$ is symmetric monoidal by \cref{prop:monoidal-cross-effect}. The theorem will then  be a consequence of \cite[Proposition 5.29]{MathewNaumannNoel17} if we can show the following:
	\begin{enumerate}
		\item The adjunction $\crosseffect_d(-)(\bbS,\ldots,\bbS)\dashv \kappa_d$ satisfies the projection formula; 
		\item $\kappa_d$ preserves colimits; and
		\item $\kappa_d$ is conservative.
	\end{enumerate}
	The first follows formally from \cite[Proposition 2.16]{BalmerDellAmbrogioSanders16} using \cref{cor:rigid-compact}; the second follows from the fact that $\kappa_d$ is both a left and right adjoint; and the third follows from the observation that $\crosseffect_d(-)(\bbS,\ldots,\bbS)$ is essentially surjective (as the $d$-th cross-effect of the functor $X \mapsto P_dh_{\bbS}(A \otimes X)$ evaluated at $(\bbS,\ldots,\bbS)$ is $A$). Here we are using the observation that $\kappa_d$ is conservative if and only if the essential image of $\crosseffect_d(-)(\bbS,\ldots,\bbS)$ contains a family of generators. 
\end{proof}

\begin{Rem}\label{rem:separableextension}
    The commutative algebra $P_dh_{\bbS}(d)$ is separable, as a consequence of \Cref{lem:hSi-is-separable}. In other words, up to equivalence, the functor 
    \[
		\crosseffect_d(-)(\bbS,\ldots,\bbS) \colon \Exc{d}(\Sp^c,\Sp) \to \Sp
    \]
	is given by extension of scalars along a compact commutative separable algebra. Such extensions are called finite \'etale; see \cite{Balmer16b,Sanders22}.
\end{Rem}

\section{Tensor idempotents and Goodwillie derivatives}

We now recall some facts about tensor idempotents and their relation with completeness, torsion and the Tate construction. We will utilize the geometric perspective afforded by the Balmer spectrum. In the context of $d$-excisive functors, these constructions will also lead naturally to the Goodwillie derivatives.

\subsection*{Tensor idempotents and the Tate construction}

\begin{Def}[The Balmer spectrum]
	Let $\cat T$ be a rigidly-compactly generated tensor triangulated category such as (the homotopy category of) $\Exc{d}(\Sp^c,\Sp)$. We write $\cat T^c$ for the tensor triangulated subcategory of compact ($=$dualizable) objects. The \emph{Balmer spectrum} of $\cat T^c$ is the topological space $\Spc(\cat T^c)$ whose points are the prime tt-ideals of $\cat T^c$ and whose topology has $\{\supp(a)\}_{a \in \cat T^c}$ as a basis of closed sets, where
	\[
		\supp(a) \coloneqq \SET{\cat P \in \Spc(\cat T^c)}{a \not \in \cat P}.
	\]
	This topology is spectral in the sense of \cite{DickmannSchwartzTressl19} and the closure of a prime tt-ideal $\cat P \in \Spc(\cat T^c)$ is given by $\overline{\{\cat P\}} = \SET{\cat Q}{\cat Q \subseteq \cat P}$. See \cite{Balmer05a} for more details.
\end{Def}

\begin{Rem}\label{Rem:basic-balmer-properties}
	The spectrum is contravariantly functorial. In particular, if $F\colon \cat T \to \cat U$ is a tt-functor of rigidly-compactly generated tt-categories, then there is an induced continuous map $\Spc(F)\colon \Spc(\cat U^c)\to\Spc(\cat T^c)$ given by $\cat P \mapsto \SET{x \in \cat T^c}{F(x)\in \cat P}$. Note that $\Spc(F)$ preserves inclusions $\cat Q \subseteq \cat P$ of prime tt-ideals by construction.
\end{Rem}

\begin{Def}
	A subset $Y \subseteq \Spc(\cat T^c)$ is a \emph{Thomason subset} if it can be written as a union $\cup_{\alpha} Y_{\alpha}$ in which each $Y_{\alpha}$  is a closed set whose complement is quasi-compact.
\end{Def}

\begin{Rem}\label{rem:balmer-classification}
	By \cite[Theorem 4.10]{Balmer05a} there is a bijection
	\[\begin{tikzcd}
		\left\{ \begin{matrix}\text{thick ideals} \\\text{of $\cat T^c$}\end{matrix} \right\}
		\ar[leftrightarrow,r,"\sim"] & 
		\left\{
		\begin{matrix}
		\text{Thomason subsets}\\
		\text{of $\Spc(\cat T^c)$}
		\end{matrix}
		\right\}
	\end{tikzcd}\]
	which sends a thick ideal $\cat J$ to $\supp(\cat J) \coloneqq \bigcup_{a \in \cat J}\supp(a)$ and with inverse given by $Y \mapsto \cat T_Y^c \coloneqq \SET{a \in \cat T^c}{\supp(a) \subseteq Y}$. 
\end{Rem}

\begin{Exa}\label{exa:balmer-spectrum-spectra}
	The Balmer spectrum of $\Sp^c$ is described in \cite[Section 9]{Balmer10b} by reinterpreting the work of Hopkins--Smith~\cite{HopkinsSmith98}. For a prime number $p$ and chromatic height $1 \le h \le \infty$, we have a prime ideal $\cat C_{p,h} \in \Spc(\Sp^c)$ which consists of those finite spectra annihilated by the $p$-local Morava $K$-theory $K(p,h-1)$:
	\[
		\cat C_{p,h} \coloneqq \SET{x \in \Sp^c}{K(p,h-1)_*(x)=0}.
	\]
	In particular, $\cat C_{p,\infty}$ is the kernel of $\Fp$-homology $\HFp$, while $\cat C_{p,1}$ is the kernel of rational homology $\HQ$ (that is, the finite torsion spectra), independently of $p$. Consequently, we also write $\cat C_{0,1}\coloneqq \cat C_{p,1}$ (for all $p$). With these definitions in hand, $\Spc(\Sp^c)$ is the topological space depicted in \Cref{fig:balmer-sp-of-sp} below. In this figure, inclusion goes downwards, while closure goes upwards; in particular, the points $\cat C_{p,\infty}$ are closed points, while $\cat C_{0,1}$ is a generic point. A more detailed explanation of the topology may be found in \cite[Corollary 9.5]{Balmer10b}. For later use, we also introduce the notation $\cat C_{p,0} \coloneqq \Sp_{(p)}^c$ for the category of finite $p$-local spectra.
	\begin{figure}[h!]
	\[\begin{tikzcd}[ampersand replacement=\&,row sep=small]
		{\cat{C}_{2,\infty}} \& {\cat{C}_{3,\infty}} \& \ldots \& {\cat{C}_{p,\infty}} \& \ldots \\
		\vdots \& \vdots \&\& \vdots \\
		{\cat{C}_{2,h}} \& {\cat{C}_{3,h}} \& \ldots \& {\cat{C}_{p,h}} \& \cdots \\
		{\vdots } \& \vdots \&\& \vdots \\
		{\cat{C}_{2,2}} \& {\cat{C}_{3,2}} \& \cdots \& {\cat{C}_{p,2}} \& \cdots \\
		\&\& {\cat{C}_{0,1}}
		\arrow[no head, from=5-1, to=6-3]
		\arrow[no head, from=5-2, to=6-3]
		\arrow[no head, from=6-3, to=5-4]
		\arrow[no head, from=5-1, to=4-1]
		\arrow[no head, from=5-2, to=4-2]
		\arrow[no head, from=5-4, to=4-4]
		\arrow[no head, from=3-1, to=2-1]
		\arrow[no head, from=3-2, to=2-2]
		\arrow[no head, from=3-4, to=2-4]
		\arrow[no head, from=4-1, to=3-1]
		\arrow[no head, from=4-2, to=3-2]
		\arrow[no head, from=4-4, to=3-4]
		\arrow[no head, from=2-1, to=1-1]
		\arrow[no head, from=2-2, to=1-2]
		\arrow[no head, from=2-4, to=1-4]
	\end{tikzcd}\]
	\caption{The Balmer spectrum $\Spc(\Sp^c)$ of the category of finite spectra}\label{fig:balmer-sp-of-sp}
	\end{figure}
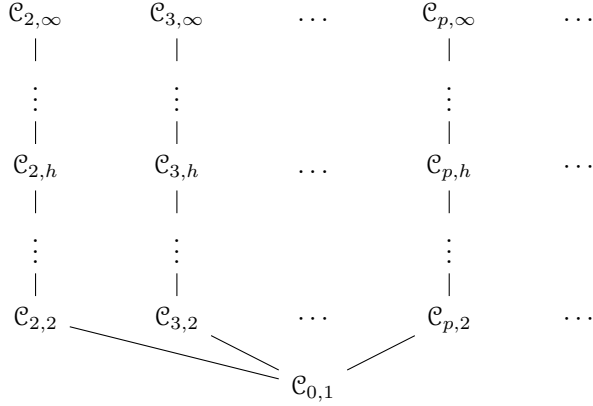
\end{Exa}

\begin{Rem}\label{rem:balmer-favi-properties}
	We now recall some constructions from \cite{BalmerFavi11} and \cite[Section 5]{BalmerSanders17}. Let $Y \subseteq \Spc(\cat T^c)$ be a Thomason subset and let $V \coloneqq \Spc(\cat T^c) \setminus Y$ be its complement. Then there is an associated \emph{idempotent triangle} 
	\begin{equation}\label{eq:basic-idempotent-triangle}
		e_Y \to \unit \to f_Y \to \Sigma e_Y
	\end{equation}
	in $\cat T$, i.e., an exact triangle with $e_Y \otimes f_Y  = 0$. Moreover, we have
	\[
		\cat T_Y \coloneqq e_Y \otimes \cat T = \ker(f_Y \otimes -) = \Loco{e_Y} = \Loc \langle \cat T_Y^c\rangle
	\]
	and
	\[
		\cat T(V) \coloneqq f_Y\otimes \cat T \simeq \cat T/\Loco{e_Y} = \cat T/\Loc \langle \cat T^c_Y \rangle 
	\]
	and
	\[
		\cat T^Y \coloneqq \ihom{e_Y,\cat T} = \Coloc\langle \cat T_Y^c \otimes \cat T\rangle.
	\]
	The three categories $\cat T_Y, \cat T(V)$ and $\cat T^Y$ are related by taking right orthogonals:
	\[
		\cat T(V) = (\cat T_Y)^{\perp}  = \SET{t \in \cat T}{\ihom{s,t} = 0 \text{ for all } s \in \cat T_Y }
	\]
	and 
	\[
		\cat T^Y = (\cat T(V))^{\perp}  = \SET{t \in \cat T}{\ihom{s,t} = 0 \text{ for all } s \in \cat T(V)}.
	\]
	Moreover, the categories $\cat T_Y$ and $\cat T^Y$ are equivalent via the functors $\ihom{e_Y,-}$ and~$e_Y \otimes -$. Note in particular that $\cat T \to \cat T(V)$ is a finite localization whose compactly generated subcategory of acyclic objects is $\cat T_Y=\Loc\langle \cat T_Y^c\rangle \subseteq \cat T$. On Balmer spectra, this finite localization induces the embedding
	\[
		\Spc(\cat T(V)^c) \cong V \hookrightarrow \Spc(\cat T^c)
	\]
	which explains the notation $\cat T(V)$; see \cite[Remark 1.23]{bhs1}.
\end{Rem}

\begin{Rem}
	We can represent the situation in the previous remark as follows: 
	\begin{equation}\label{eq:abstractdiagram}
	\begin{tikzcd}[ampersand replacement=\&,row sep=0.4in ]
		\& \cat T(V) = f_Y \otimes \cat T \\
		\& {\cat T} \\
		{\cat T_Y = e_Y \otimes \cat T } \&\& {\ihom{e_Y,\cat T}=\cat T^Y}
		\arrow[shift left=1, hook, from=3-1, to=2-2]
		\arrow["e_Y \otimes -", shift left=1, from=2-2, to=3-1]
		\arrow[shift left=1, hook, from=1-2, to=2-2]
		\arrow["{f_Y \otimes -}", shift left=1, from=2-2, to=1-2]
		\arrow[shift right=1, hook, from=3-3, to=2-2]
		\arrow["\ihom{e_Y,-}"', shift right=1, from=2-2, to=3-3]
		\arrow[bend left = 40, dashed, from=3-1, to=1-2]
		\arrow[shift right=1, bend left = 40, dashed, from=1-2, to=3-3]
		\arrow["\sim"', shift right=1, from=3-1, to=3-3]
		\arrow["\sim"', shift right=1, from=3-3, to=3-1]
	\end{tikzcd}
	\end{equation}
	where the dashed arrows indicate orthogonality. In other words, we have a recollement; see \cite[Section 5]{BalmerSanders17}. 
\end{Rem}	

\begin{Exa}\label{exa:completeness-wrt-an-object}
	For any compact object $A \in \cat T^c$, we can consider the Thomason closed subset $Y\coloneqq \supp(A)$. Then the categories $\cat T_{\supp(A)},\cat T(\supp(A)^{\complement})$ and $\cat T^{\supp(A)}$ are precisely the categories of $A$-torsion, $A$-local, and $A$-complete objects respectively, in the sense of \cite[Theorem 3.35]{HoveyPalmieriStrickland97}, \cite[Theorem 2.21]{bhv1}, or \cite[Theorem~3.9]{MathewNaumannNoel17}. In particular, we have that $\cat T_{\supp(A)} = \Loco{A} \subseteq \cat T$. However, these categories only depend on the thick ideal in $\cat T^c$ generated by $A$ or, equivalently by \cref{rem:balmer-classification}, the Thomason subset $\supp(A)$. Our approach emphasizes the geometric perspective by regarding these constructions as associated to a Thomason subset of the Balmer spectrum rather than to a collection of objects of the category. Finally, note that $\cat T \to \cat T(\supp(A)^{\complement})$ is the finite localization away from the object~$A$. Geometrically it restricts to the open complement of $\supp(A)$.
\end{Exa}

The following definition is due to Greenlees \cite{Greenlees01} in the axiomatic setting. 

\begin{Def}\label{def:tate-functor}
	Let $Y \subseteq \Spc(\cat T^c)$ be a Thomason subset. We define the \emph{Tate functor} with respect to $Y$ to be
	\[
		t_Y \coloneqq f_Y \otimes \ihom{e_Y,-}\colon \cat T \to \cat T.
	\]
\end{Def}

\begin{Rem}\label{rem:warwick}
	Greenlees' Warwick duality \cite[Corollary 2.5]{Greenlees01} shows that the Tate construction satisfies $t_Y \simeq \ihom{f_Y,\Sigma e_Y \otimes -}$. Moreover, there is a pullback square
	\begin{equation}\label{eq:tate-square}
	\begin{tikzcd}[ampersand replacement=\&]
		X \& {\ihom{e_Y,X}} \\
		{f_Y \otimes X} \& {t_Y (X)}
		\arrow[from=1-1, to=2-1]
		\arrow[from=1-1, to=1-2]
		\arrow[from=2-1, to=2-2]
		\arrow[from=1-2, to=2-2]
	\end{tikzcd}\end{equation} 
	for any $X \in \cat T$.
\end{Rem}

The following result is due to Balmer and Sanders~\cite[Proposition~5.11]{BalmerSanders17}.

\begin{Prop}[Balmer--Sanders]\label{prop:idempotents-under-tt-functors}
	Let $F \colon \cat T \to \cat U$ be a coproduct-preserving tensor triangulated functor and let $\phi \colon \Spc(\cat U^c) \to \Spc(\cat T^c)$ be the induced map. Let $Y \subseteq \Spc(\cat T^c)$ be a Thomason subset and set $Y' \coloneqq \phi^{-1}(Y) \subseteq \Spc(\cat U^c)$. Then $\cat U_{Y'}^c  = \thickid\langle F(\cat T^c_Y)\rangle$ and there is a unique isomorphism of idempotent triangles
	\[
		F(e_Y \to \unit \to f_Y \to \Sigma e_Y) \simeq (e_{Y'} \to \unit \to f_{Y'} \to \Sigma e_{Y'})
	\]
	in $\cat U$. If moreover $F$ is closed monoidal, then $F \circ t_Y \simeq t_{Y'} \circ F$. 
\end{Prop}

\begin{Exa}[Chromatic truncation]\label{exa:chromatic-truncation}
	Suppose that $\mathcal C$ is a symmetric monoidal stable $\infty$-category, so that it admits a unique symmetric monoidal colimit preserving functor $i \colon \Sp \to \mathcal C$ as in \Cref{rem: F_A}. Let $1 \le h \le \infty$ and consider the Thomason subset
	\[
		Y_{p,h} \coloneqq \SET{\cat C_{p,m}}{m > h} \cup \SET{\cat C_{q,m}}{q \ne p, m>1} \subseteq \Spc(\Sp^c)
	\]
	with associated idempotent triangle $e_{p,h} \to \unit \to f_{p,h} \to \Sigma e_{p,h}$ in $\Sp$, as in \cite[Example 5.12]{BalmerSanders17}. The right idempotent $f_{p,h}$ is the finite localization $L_{h-1}^f\mathbb{S}$ of the sphere spectrum. By \Cref{prop:idempotents-under-tt-functors}, the Thomason subset
	\[
		Y'_{p,h} \coloneqq \Spc(i)^{-1}(Y_{p,h})\subseteq \Spc(\mathcal C^c)
	\]
	has associated idempotent triangle
	\[
		i(e_{p,h}) \to \unit \to i(f_{p,h}) \to \Sigma i(e_{p,h})
	\]
	in $\mathcal C$.
	The \emph{chromatic truncation} of $\mathcal C$ below height
    $h$ (at the prime $p$) is the finite localization
	\[
		\mathcal C_{p,\le h} \coloneqq i(f_{p,h}) \otimes \mathcal C.
	\]
	It has the property that $\Spc((\mathcal C_{p,\le h})^c) \cong \Spc(\mathcal C^c) \setminus Y'_{p,h}$. Note that $\mathcal C_{p,\le \infty}$ is simply the $p$-localization of $\mathcal C$.
\end{Exa}

\begin{War}
	Because of the way the primes $\cat C_{p,h}$ of $\Spc(\Sp^c)$ are indexed (as the kernel of $K(p,h-1)$), there can be an off-by-one discrepancy in different usages of the term ``height'' which can lead to confusion. For example, note that the prime~$\cat C_{p,h}$ has height $h-1$ in the sense of Krull dimension. Geometrically, truncation below height $h$ is restriction to the open piece
	\[\begin{tikzcd}[ampersand replacement=\&,row sep=small,column sep=tiny]
		{\cat{C}_{p,h}} \\
		{\cat{C}_{p,h-1}} \\
		{\vdots } \\
		{\cat{C}_{p,2}} \\
		\& {\cat{C}_{0,1}}.
		\arrow[no head, from=1-1, to=2-1]
		\arrow[no head, from=2-1, to=3-1]
		\arrow[no head, from=3-1, to=4-1]
		\arrow[no head, from=4-1, to=5-2]
	\end{tikzcd}\]
	This explains the notation $\Sp \to \Sp_{p,\le h}$. However, note that this is the finite localization associated to Morava $E$-theory $E(h-1)$ which is a ring spectrum of height $h-1$ in the sense of chromatic homotopy theory; see \cref{def:height} below. Nevertheless, it is convenient to use language like ``consider a chromatic height $1 \le h \le \infty$'' even if the construction associated to this choice of $h$ results in something that really has ``height''~$h-1$.
\end{War}

\begin{Rem}\label{rem:Utate}
	In the situation of \cref{prop:idempotents-under-tt-functors}, if $G$ denotes the right adjoint of $F\colon \cat T \to \cat U$ then one readily checks that
	\[
		Gt_{\phi^{-1}(Y)}(Ft) \simeq t_Y(t \otimes G\unit)
	\]
	for all $t \in \cat T$ by using the standard isomorphisms of \cite{BalmerDellAmbrogioSanders16}.
\end{Rem}

\subsection*{Completeness and torsion in \texorpdfstring{$d$}{d}-excisive functors}

We now specialize the discussion to the category of $d$-excisive functors. Although we have not yet determined its spectrum, we can proceed as in \cref{exa:completeness-wrt-an-object} with respect to the compact $d$-excisive functor $A = P_dh_{\bbS}(d)$, that is, with respect to the Thomason closed subset $Y = \supp(P_dh_{\bbS}(d))$. Our goal is to understand the recollement~\eqref{eq:abstractdiagram} and the associated Tate construction in this example.

\begin{Not}\label{not:t_d}
	We will write $t_d\colon \Exc{d}(\Sp^c,\Sp)\to\Exc{d}(\Sp^c,\Sp)$ for the Tate functor associated to the Thomason closed subset $Y\coloneqq \supp(P_d h_{\bbS}(d))$. That is, $t_d \coloneqq t_{\supp(P_d h_{\bbS}(d))}$ in the notation of \cref{def:tate-functor}.
\end{Not}

\begin{Thm}\label{thm:finite-localization-p-n}
	Let $Y \coloneqq \supp(P_dh_{\bbS}(d)) \subseteq \Spc(\Exc{d}(\Sp^c,\Sp)^c)$. Then 
	\[
		e_{Y} \otimes \Exc{d}(\Sp^c,\Sp) = \Homog_d(\Sp^c,\Sp)
	\]
	and the associated idempotent triangle in $\Exc{d}(\Sp^c,\Sp)$ is 
	\[
		D_dh_{\bbS} \to P_dh_{\bbS} \to P_{d-1}h_{\bbS} \to \Sigma D_dh_{\bbS}.
	\]
	In particular, $P_{d-1}\colon \Exc{d}(\Sp^c,\Sp) \to \Exc{d-1}(\Sp^c,\Sp)$ is the finite localization away from~$P_dh_{\bbS}(d)$. 
\end{Thm}

\begin{proof}
	We need to show that $\ker(P_{d-1}) = \Locideal(P_dh_{\bbS}(d))$. To see this, observe first that $P_{d-1}P_dh_{\bbS}(d) \simeq P_{d-1}h_{\bbS}(d) = 0$ because $P_dh_{\bbS}(d)$ is $d$-homogeneous (\Cref{prop:pnhsn-homogeneous}). It follows that $P_dh_{\bbS}(d) \in \ker(P_{d-1})$, so that $\Locideal(P_dh_{\bbS}(d)) \subseteq \ker(P_{d-1})$. On the other hand, let $H \in \ker(P_{d-1})$ and note that $H$ is $d$-homogeneous. Suppose then that
	\[
		\Hom_{\Exc{d}(\Sp^c,\Sp)}(P_dh_{\mathbb{S}}(d),H) = 0.
	\]
	By \Cref{lem:cross-effect-yoneda-lemma} we have that $\crosseffect_d(H)(\bbS,\ldots,\bbS) = 0$ and hence $H = 0$ by \Cref{lem:cross-effects-homogeneous}. We deduce that $\ker(P_{d-1})$ coincides with the localizing ideal generated by the compact object~$P_dh_{\bbS}(d)$. Hence the localization $P_{d-1}\colon\Exc{d}(\Sp^c,\Sp) \to \Exc{d}(\Sp^c,\Sp)$ coincides with the finite localization away from $P_dh_{\bbS}(d)$.
\end{proof}

\begin{Rem}
	This result shows that the category of $d$-homogeneous functors in Goodwillie calculus is analogous to the category of free $G$-spectra in equivariant homotopy theory. 
\end{Rem}

\begin{Cor}\label{cor:pm-finite-localization}
	The functor $P_{d-i}\colon \Exc{d}(\Sp^c,\Sp) \to \Exc{d-i}(\Sp^c,\Sp)$ is a finite localization whose ideal of acyclics is generated by $\SET{P_d(h_{\mathbb{S}}(j))}{d-i+1 \leq j \leq d}$.
\end{Cor}

\begin{proof}
	We prove this by induction on $i$. One can either start the induction at $i = 0$, where the statement is vacuous, or at $i = 1$, where it is the content of \Cref{thm:finite-localization-p-n}. In either case suppose the statement of the corollary holds for the functor $P_{d-i}$. The localization $P_{d-i-1} \colon \Exc{d-i}(\Sp^c,\Sp) \to \Exc{d-i-1}(\Sp^c,\Sp)$ is the finite localization that kills $P_{d-i}(h_{\bbS}(d-i))$. Noting that this is the image of the generator $P_d(h_{\bbS}(d-i))$ under the localization $P_{d-i}$, the inductive hypothesis and the evident `third isomorphism theorem' for localizations gives the result.  
\end{proof}

\begin{Exa}\label{example:not finite}
	The obvious extension of \Cref{cor:pm-finite-localization} obtained by letting $d$ go to~$\infty$ is false. By this we mean that the kernel of the functor
	\[
		P_k\colon \Fun(\Sp^c, \Sp)\to \Exc{k}(\Sp^c, \Sp)
	\]
	is not generated by $\SET{h_{\mathbb{S}}(j)}{k < j}$. To see this, let us consider the case $k=0$. In this case $P_0$ is the trivial functor and the kernel of $P_0$ is all of $\Fun(\Sp^c, \Sp)$. This category is not generated by $\SET{h_{\mathbb{S}}(j)}{0 < j}$. Indeed, consider the functor $F(X)=\Sigma^\infty\Omega^{\infty+1} (H\mathbb Q \wedge X)$. This functor has the property that $F(X)\simeq 0$ whenever~$X$ is a wedge of finitely many copies of the sphere spectrum. It follows that all the cross-effects of $F$ are trivial, i.e., $\crosseffect_jF(\bbS, \ldots, \bbS)\simeq 0$ for all $j>0$. By \Cref{lem:ching-yoneda}, this means that for all $j>0$ we have
	\[
		\Hom_{\Fun(\Sp^c,\Sp)}(h_{\mathbb{S}}(j), F)\simeq 0
	\]
	but $F$ is clearly not a trivial functor.

	Going further, we suspect that $\Exc{k}(\Sp^c, \Sp)$ is {\it not} a finite localization of $\Fun(\Sp^c, \Sp)$, i.e., that the kernel of the localization is not generated by compact objects.
\end{Exa}

With the torsion categories now identified, let us turn to the complete category. 

\begin{Prop}\label{prop:n-homogeneous-functors}
	In the situation of \Cref{thm:finite-localization-p-n}, we have an equivalence of symmetric monoidal $\infty$-categories
	\[
		\ihom{e_Y,\Exc{d}(\Sp^c,\Sp)} \simeq \Fun(B\Sigma_d,\Sp).
	\]
	In particular, there is an equivalence of symmetric monoidal stable $\infty$-categories $\Homog_{d}(\Sp^c,\Sp) \simeq \Fun(B\Sigma_d,\Sp)$.
\end{Prop}

\begin{proof}
	Our proof is modeled on \cite[Proposition 6.17]{MathewNaumannNoel17}, which is the corresponding statement in equivariant homotopy theory. For brevity, let us write $\cC $ for $\Exc{d}(\Sp^c,\Sp)$. Because $P_dh_{\bbS}(d)$ is dualizable, we can apply \cite[Theorem~2.30]{MathewNaumannNoel17} to deduce that 
	\[
		\ihom{e_Y,\cat C} \simeq \Tot\left(\xymatrix{\Mod_{\cC}(P_dh_{\bbS}(d)) \ar@<-.5ex>[r] \ar@<.5ex>[r] & \Mod_{\cC}(P_dh_{\bbS}(d) \circledast P_dh_{\bbS}(d)) \ar@<-1ex>[r] \ar@<1ex>[r] \ar[r] & \cdots}  \right).
	\]
	We have seen in \Cref{thm:barr-beck} that $\Mod_{\cC}(P_dh_{\bbS}(d)) \simeq \Sp$. Using \Cref{exa:day-convolution-n} we deduce moreover that 
	\[
		\Mod_{\cC}(P_dh_{\bbS}(d)^{\circledast k}) \simeq \prod_{(\Sigma_d)^{\times(k-1)}}\Sp.
	\]
	Unwinding the definitions, we find that  $\ihom{e_Y,\cat C}$ is identified with a totalization 
	\[
		\ihom{e_Y,\cat C} \simeq \Tot\left(\xymatrix{\Sp\ar@<-.5ex>[r] \ar@<.5ex>[r] & \displaystyle \prod_{\Sigma_d} \Sp \ar@<-1ex>[r] \ar@<1ex>[r] \ar[r] & \cdots} \right)
	\]
	which recovers the functor category $\Fun(B\Sigma_d,\Sp)$ for the standard simplicial decomposition of $B\Sigma_d$. The final claim then follows from the symmetric monoidal equivalence of $\infty$-categories
	\[
		e_Y \otimes \cat C \simeq \ihom{e_Y,\cat C}
	\]
	as noted in \Cref{rem:balmer-favi-properties}.
\end{proof}

\begin{Rem}
	The local duality diagram can then be written in the following form:
	\begin{equation}\label{eq:local-duality}
	\begin{tikzcd}[ampersand replacement=\&,row sep=0.4in ]
		\& {\Exc{d-1}(\Sp^c,\Sp)} \\
		\& {\Exc{d}(\Sp^c,\Sp)} \\
		{\Homog_d(\Sp^c,\Sp)} \&\& {\Fun(B\Sigma_d,\Sp)}
		\arrow[shift left=1, hook, from=3-1, to=2-2]
		\arrow["D_d", shift left=1, from=2-2, to=3-1]
		\arrow[shift left=1, hook, from=1-2, to=2-2]
		\arrow["{P_{d-1}}", shift left=1, from=2-2, to=1-2]
		\arrow[shift right=1, hook, from=3-3, to=2-2]
		\arrow["", shift right=1, from=2-2, to=3-3]
		\arrow[bend left = 40, dashed, from=3-1, to=1-2]
		\arrow[shift right=1, bend left = 40, dashed, from=1-2, to=3-3]
		\arrow["\sim"', shift right=1, from=3-1, to=3-3]
		\arrow["\sim"', shift right=1, from=3-3, to=3-1]
	\end{tikzcd}
	\end{equation}
\end{Rem}

The equivalence $\Homog_d(\Sp^c,\Sp) \xrightarrow{\sim} \Fun(B\Sigma_d,\Sp)$ is Goodwillie's classification of $d$-homogeneous functors \cite[Theorem 3.5 and \S5]{Goodwillie03}:

\begin{Lem}[Goodwillie]\label{lem:goodwillie-classification}
	The equivalence $\Homog_d(\Sp^c,\Sp) \xrightarrow{\sim} \Fun(B\Sigma_d,\Sp)$ sends a $d$-homogeneous functor $F$ to $\crosseffect_d(F)(\bbS,\ldots,\bbS)$. Its inverse sends $A \in \Fun(B\Sigma_d,\Sp)$ to the functor $G_A$ with $G_A(X) = (A \otimes X^{\otimes d})_{h\Sigma_d}$. 
\end{Lem}

\begin{Def}
	Let $F \colon \Sp^c \to \Sp$ be a reduced functor. The $d$-th \emph{Goodwillie derivative} of $F$ is defined by
	\[
		\partial_dF \coloneqq \crosseffect_d(D_dF)(\bbS,\ldots,\bbS).  
	\]
\end{Def}

\begin{Rem}\label{rem:derivative-depends-on-pnf}
	Note that \Cref{Rem:vanishing-cross-effects} implies that $\crosseffect_d(D_dF) \simeq \crosseffect_d(P_dF)$, so we could have alternatively defined $\partial_d$ in terms of its $d$-excisive approximation. In particular, if $F$ is $d$-excisive, then $\partial_d(D_dF) \simeq \partial_d(F)$. More generally, for any $d$-excisive functor~$F$, we have a natural equivalence
	\begin{equation}\label{eq:shifting-derivatives}
		\partial_d(D_dF(X \otimes -)) \simeq \partial_d(F) \otimes X^{\otimes d}.
	\end{equation}
	Indeed there are equivalences
	\[
	\begin{split}
		\partial_d(D_dF(X \otimes -)) &\simeq \crosseffect_dD_dF(X \otimes -)(\bbS,\ldots,\bbS) \\ &\simeq \crosseffect_dD_dF(X,\ldots,X) \\
		& \simeq \partial_d(F) \otimes X^{\otimes d}.
	\end{split}
	\]
\end{Rem}

\begin{Rem}
	As a special case of the previous definition, we have $d$-derivatives defined on the category of $d$-excisive functors: 
\end{Rem}

\begin{Def}\label{def:goodwillie-derivatives}
	For a $d$-excisive functor $F$, we define the $k$-th \emph{Goodwillie derivative}~$\partial_kF$ for each $1 \le k \le d$ by
	\[
		\partial_kF \coloneqq \crosseffect_k(P_kF)(\bbS,\ldots,\bbS) \simeq \crosseffect_k(D_kF)(\bbS,\ldots,\bbS) .
	\]
\end{Def}

\begin{Lem}\label{lem:geometric-derivatives}
	For each $1 \le k \le d$, the functor $\partial_k\colon \Exc{d}(\Sp^c,\Sp)\to \Sp$ is symmetric monoidal and preserves colimits.
\end{Lem}

\begin{proof}
	By definition, $\partial_k$ is the composite of $P_k \colon \Exc{d}(\Sp^c,\Sp) \to \Exc{k}(\Sp^c,\Sp)$ and $\crosseffect_k(-)(\bbS,\ldots,\bbS) \colon \Exc{k}(\Sp^c,\Sp) \to \Sp$, which are symmetric monoidal by \cref{cor:pm-finite-localization} and \cref{prop:monoidal-cross-effect}, respectively. That $\partial_k$ preserves colimits is already clear from Goodwillie's original constructions; alternatively, it also follows from the fact that $P_k$ and $\crosseffect_k(-)(\bbS,\ldots,\bbS)$ both have right adjoints. For the latter, recall \cref{rem:adjoint-to-cross-effects}.
\end{proof}

\begin{Rem}\label{rem:top-derivative-etale}
	The top derivative $\partial_d \colon \Exc{d}(\Sp^c,\Sp) \to \Sp$ given by $\crosseffect_d(-)(\bbS,\ldots,\bbS)$ is finite \'etale (\Cref{rem:separableextension}). 
\end{Rem}

\begin{Lem}\label{lem:conservativity}
	The collection of functors $\{ \partial_k \}_{1 \le k \le d}$ are jointly conservative on $\Exc{d}(\Sp^c,\Sp)$. 
\end{Lem}

\begin{proof}
	The equivalence of categories $\Homog_d(\Sp^c,\Sp) \simeq \Fun(B\Sigma_d,\Sp)$ of \Cref{prop:n-homogeneous-functors} implies that $D_dF = 0 \iff \partial_dF = 0$. Therefore, if $\partial_dF = 0$, then $F$ is $(d-1)$-excisive, and the result follows by induction. 
\end{proof}

\begin{Rem}
	It is a consequence of \cref{lem:conservativity} that every compact separable algebra in $\Exc{d}(\Sp^c,\Sp)$ has finite degree in the sense of \cite{Balmer14}. This follows from the corresponding statement for $\Sp$ established in \cite[Corollary 4.8]{Balmer14}. Indeed, $d_k \coloneqq \deg(\partial_k(A)) < \infty$ for each $1\le k \le d$ implies that for $d \coloneqq \max_{1\le k \le d} d_k$ we have $\partial_k(A^{[d+1]})=\partial_k(A)^{[d+1]}=0$ for all $1 \le k\le d$ by \cite[Theorem 3.7(a)]{Balmer14}. Hence $A^{[d+1]}=0$ by \cref{lem:conservativity} so that $\deg(A) \le d$. 
\end{Rem}

\begin{Exa}\label{exa:top-derivative-etale-degree}
	The finite \'{e}tale map $\partial_d\colon \Exc{d}(\Sp^c,\Sp)\to \Sp$ of \cref{rem:top-derivative-etale} has finite degree.
\end{Exa}

\begin{Rem}
	We now give an explicit description of the completion functor $\Exc{d}(\Sp^c,\Sp) \to \ihom{e_Y,\Exc{d}(\Sp^c,\Sp)}$ associated to $Y=\supp(P_dh_{\bbS}(d))$. In other words, bearing in mind \cref{thm:finite-localization-p-n}, we compute the internal hom $\ihom{D_dh_{\bbS},F}$ for a $d$-excisive functor $F$.
\end{Rem}

\begin{Prop}\label{prop:ihom-dn}
    Suppose that $F$ is $d$-excisive. Then 
    \[
		\ihomsub{\Exc{d}(\Sp^c,\Sp)}{D_dh_{\bbS},F}(X) \simeq (\partial_dF \otimes X^{\otimes d})^{h\Sigma_d}.
    \]
\end{Prop}

\begin{proof}
	Suppose that $F$ and $G$ are $d$-excisive functors. We recall from \Cref{thm:properties-n-excisive} that the internal hom in $\Exc{d}(\Sp^c,\Sp)$ is computed in $\Fun(\Sp^c,\Sp)$ where it is given by 
	\[
	\begin{split}
		\ihomsub{\Fun(\Sp^c,\Sp)}{G,F}(X) & \simeq \Hom_{\Fun(\Sp^c,\Sp)}(G,F(X \otimes -)) \\
		& \simeq \Hom_{\Exc{d}(\Sp^c,\Sp)}(G,F(X \otimes -))
	\end{split}
	\]
	where the last step follows because $\Exc{d}(\Sp^c,\Sp) \subseteq \Fun(\Sp^c,\Sp)$ is a full subcategory.  Therefore,
	\[
	\begin{split}
		\ihomsub{\Exc{d}(\Sp^c,\Sp)}{D_dh_{\bbS},F}(X) &\simeq \Hom_{\Exc{d}(\Sp^c,\Sp)}(D_dh_\bbS,F(X \otimes -)) \\& \simeq
		\Hom_{\Homog_d(\Sp^c,\Sp)}(D_dh_\bbS,D_dF(X \otimes -))
	\end{split}
	\]
	where we have used that $D_d \colon \Exc{d}(\Sp^c,\Sp) \to \Homog_d(\Sp^c,\Sp)$ is right adjoint to the inclusion.

	By \Cref{prop:n-homogeneous-functors}, there is an equivalence 
	\[
		\Homog_d(\Sp^c,\Sp) \leftrightarrows \Fun(B\Sigma_d, \Sp)
	\]
	given by assigning to a $d$-homogeneous functor $F$ the Goodwillie derivative $\partial_dF$ with $\Sigma_d$-action. Now we observe that $\partial_d(D_dh_{\bbS}) \simeq \bbS$ with trivial $\Sigma_d$-action, and $\partial_d(D_dF(X \otimes -)) \simeq (\partial_d(F) \otimes X^{\otimes d})$ by \eqref{eq:shifting-derivatives}.

	We deduce from the above discussion that 
	\[
	\begin{split}
		\ihomsub{\Exc{d}(\Sp^c,\Sp)}{D_dh_{\bbS},F}(X) &\simeq \Hom_{\Sp}(\partial_dD_dh_{\bbS}, \partial_dD_dF(X \otimes -))^{h\Sigma_d} \\
		&\simeq \Hom_{\Sp}(\bbS,\partial_dF \otimes X^{\otimes d})^{h\Sigma_d} \\
		& \simeq (\partial_dF \otimes X^{\otimes d})^{h\Sigma_d},
	\end{split}
	\]
	as required. 
\end{proof}

\begin{Rem}
	We can also identify the Tate square of~\eqref{eq:tate-square}. The existence of such a pullback square was originally shown by McCarthy \cite[Proposition 4]{McCarthyRy2001Dual}; see also \cite[Proposition 1.9]{Kuhn04}.
\end{Rem}

\begin{Prop}[Kuhn--McCarthy]\label{prop:tatesquare}
	For any reduced functor $F \colon \Sp^c \to \Sp$ there is a pullback square of the form
	\[\begin{tikzcd}[ampersand replacement=\&]
		{P_dF(X)} \& {(\partial_dF \otimes X^{\otimes d})^{h\Sigma_d}} \\
		{P_{d-1}F(X)} \& {(\partial_dF \otimes X^{\otimes d})^{t\Sigma_d}}.
		\arrow[from=1-1, to=2-1]
		\arrow[from=1-1, to=1-2]
		\arrow[from=2-1, to=2-2]
		\arrow[from=1-2, to=2-2]
	\end{tikzcd}\]
	The fiber of the vertical maps is $D_dF(X) \simeq (\partial_dF \otimes X^{\otimes d})_{h\Sigma_d}$. In particular, the Tate construction associated to the Thomason closed subset $Y = \supp(P_dh_{\bbS}(d)) \subseteq \Spc(\Exc{d}(\Sp^c,\Sp)^c)$ is given by
	\[
		t_Y(F)(X) \simeq {(\partial_dF \otimes X^{\otimes d})^{t\Sigma_d}}.
	\]
\end{Prop}

\begin{proof}
	By \Cref{prop:n-homogeneous-functors,lem:goodwillie-classification} we have $D_dF(X) \simeq (\partial_dF \otimes X^{\otimes d})_{h\Sigma_d}$. Using \Cref{rem:warwick}, we can identify the bottom right-hand corner of the Tate square as the $(d-1)$-excisive approximation to the functor $X \mapsto (\partial_dF \otimes X^{\otimes d})^{h\Sigma_d}$. A direct computation shows that this functor is $X \mapsto (\partial_dF \otimes X^{\otimes d})^{t\Sigma_d}$ (see, for example, \cite[Lemma 5.2]{Kuhn07} or \cite[Proof of Proposition 4]{McCarthyRy2001Dual}) and the result follows. 
\end{proof}

\begin{Rem}
	We have already introduced the notation
	\[
		t_d \colon \Exc{d}(\Sp^c,\Sp) \to \Exc{d}(\Sp^c,\Sp)
	\]
	for the Tate construction associated to $\supp(P_dh_{\bbS}(d))\subseteq \Spc(\Exc{d}(\Sp^c,\Sp)^c)$; see \cref{not:t_d}. By \cref{prop:tatesquare}, the $d$-excisive functor $t_d(F)$ has the following explicit description:
    \[
		t_d(F)(X) \simeq (\partial_dF \otimes X^{\otimes d})^{t\Sigma_d}.
    \]
\end{Rem}

\begin{Rem}
    The proof of \Cref{prop:tatesquare} shows that $t_d(F)$ is actually $(d-1)$-excisive, but we are considering it as a $d$-excisive functor via the natural inclusion. 
\end{Rem}

\begin{Rem}
	Combining the Tate construction $t_d$ with the inflation functor $i_d$ (\cref{def:inflation}) and the Goodwillie derivatives~$\partial_k$ (\cref{def:goodwillie-derivatives}) yields:
\end{Rem}

\begin{Def}\label{def:tate-derivatives}
    The \emph{Tate-derivatives} on the category of $d$-excisive functors  are the functors $\Sp \to \Sp$ defined for each  $1 \le k \le d$ as the composite
    \[
		\Sp \xrightarrow{i_d} \Exc{d}(\Sp^c,\Sp) \xrightarrow{t_d} \Exc{d}(\Sp^c,\Sp) \xrightarrow{\partial_k} \Sp.
    \]
    Explicitly, a spectrum $A$ is sent to the $k$-th derivative of the functor
    \[
		X \mapsto (A \otimes X^{\otimes d})^{t\Sigma_d}.
    \]
\end{Def}

\begin{Rem}
	Understanding the chromatic shifting behaviour of the Tate-derivatives will play a crucial role in \cref{part:III}.
\end{Rem}

\newpage
\part{The primes of \texorpdfstring{$d$}{d}-excisive functors}

\section{The Balmer spectrum of \texorpdfstring{$d$}{n}-excisive functors as a set}

We begin by computing the underlying set of the Balmer spectrum of $\Exc{d}(\Sp^c,\Sp)$.

\begin{Rem}
	For each $1 \le k \le d$, the $k$-th Goodwillie derivative
	\[
		\partial_k\colon \Exc{d}(\Sp^c,\Sp)\to\Sp
	\]
	is a colimit-preserving symmetric monoidal functor (a.k.a.~\emph{geometric functor}) by \cref{lem:geometric-derivatives}. Since $\Exc{d}(\Sp^c,\Sp)$ is rigidly-compactly generated (\cref{cor:rigid-compact}), it restricts to a functor
	\[
		\partial_k \colon \Exc{d}(\Sp^c,\Sp)^c \to \Sp^c
	\]
	between compact objects. By pulling back prime ideals of $\Sp^c$ (whose spectrum was described in \Cref{exa:balmer-spectrum-spectra}), we produce a family of prime ideals of $\Exc{d}(\Sp^c,\Sp)$: 
\end{Rem}

\begin{Def}\label{def:prime-in-exc}
	Let $\cPd(\num{k},p,h) \coloneqq \{ x \in \Exc{d}(\Sp^c,\Sp)^c \mid \partial_k(x) \in \cat {C}_{p,h} \text{ in } \Sp^c \}$.
\end{Def}

The following lemma shows that these are all the prime ideals in $\Exc{d}(\Sp^c,\Sp)$. 

\begin{Lem}\label{lem:cover}
	Let $\partial \coloneqq \prod_{1 \le k \le d}\partial_k$. The induced map
	\begin{equation}\label{eq:Spcdelta}
		\Spc(\partial) \colon \coprod_{1 \le k \le d} \Spc(\Sp^c) \to \Spc(\Exc{d}(\Sp^c,\Sp)^c)
	\end{equation}
	is surjective. 
\end{Lem}

\begin{proof}
	This follows from \Cref{lem:conservativity,lem:geometric-derivatives} and \cite[Theorem 1.3]{barthel2023surjectivity}. 
\end{proof}

\begin{Lem}\label{lem:partial-splits-inflation}
	For each $1 \le k \le d$, the composite
	\[ 
		\Sp \xrightarrow{i_d} \Exc{d}(\Sp^c,\Sp) \xrightarrow{\partial_k} \Sp
	\]
	is equivalent to the identity functor.
\end{Lem}

\begin{proof}
	It suffices to prove the claim when $k = d$, since the composite 
	\[
		\Sp \xrightarrow{i_d} \Exc{d}(\Sp^c,\Sp) \xrightarrow{P_k} \Exc{k}(\Sp^c,\Sp)
	\]
	is equivalent to $i_k$. In other words, we must show that $\partial_d(i_d(A)) \simeq A$. Recall that $i_d(A) \simeq P_d(F_A)$, where $F_A(X) = A \otimes h_{\bbS}(X) = A \otimes \Sigma^{\infty}\Omega^{\infty}X$ (\Cref{def:inflation}). When $A = \bbS$ the claim is clear using, for example, that both functors are symmetric monoidal. One then computes $\partial_d(P_dF_A) \simeq A \otimes \partial_d(P_dF_{\bbS}) \simeq A$ and the result follows.
\end{proof}

\begin{Lem}\label{lem:Spc(i_d)}
	The functor $i_d \colon \Sp \to \Exc{d}(\Sp^c,\Sp)$ induces a map
	\[
		\Spc(i_d)\colon \Spc(\Exc{d}(\Sp^c,\Sp)^c)\to\Spc(\Sp^c)
	\]
	which sends $\cPd(\num{k},p,h)$ to $\cat C_{p,h}$ for all $1 \le k \le d$. 
\end{Lem}

\begin{proof}
	This follows from \cref{lem:partial-splits-inflation} and the definitions.
\end{proof}

\begin{Lem}\label{lem:injectivity-of-partial-k}
	For each $1 \le k \le d$, the functor $\partial_k \colon \Exc{d}(\Sp^c,\Sp) \to \Sp$ induces an embedding
	\[
		\Spc(\partial_k)\colon \Spc(\Sp^c)\hookrightarrow \Spc(\Exc{d}(\Sp^c,\Sp)^c)
	\]
	which sends $\cat C_{p,h}$ to $\cPd(\num{k},p,h)$.
\end{Lem} 

\begin{proof}
	The splitting $\partial_k \circ i_d \simeq \Id_{\Sp}$ established by \cref{lem:partial-splits-inflation} implies by \Cref{Rem:basic-balmer-properties} that $\cat C_{p,n} \subseteq \cat C_{q,m}$ if and only if $\cPd(\num{k},p,n)\subseteq \cPd(\num{k},q,m$). It follows that the map $\Spc(\partial_k)$ is injective and, in fact, an embedding.
\end{proof}

\begin{Prop}\label{prop:pn-induced-spectra}
	Let $1 \le m \le d$. The functor $P_{m} \colon \Exc{d}(\Sp^c,\Sp) \to \Exc{m}(\Sp^c,\Sp)$ induces an open embedding
	\[
		\Spc(P_m) \colon \Spc(\Exc{m}(\Sp^c,\Sp)^c) \hookrightarrow \Spc(\Exc{d}(\Sp^c,\Sp)^c)
	\]
	which sends $\cPm(\num{k},p,h)$ to $\cPd(\num{k},p,h)$ for each $1 \le k \le m$.
\end{Prop}

\begin{proof}
	By \cref{cor:pm-finite-localization} the functor $P_m$ is the finite localization on $\Exc{d}(\Sp^c,\Sp)$ whose kernel is the ideal generated by $\SET{P_dh_{\bbS}(i)}{m+1 \le i \le d}$. Any finite localization induces an embedding on spectra by \cite[Proposition 3.11]{Balmer05a}. Moreover, the complement of its image is the closed set 
	\[
		\bigcup_{m+1 \le i \le d} \supp(P_d h_{\bbS}(i)) \subseteq \Spc(\Exc{d}(\Sp^c,\Sp)^c).
	\]
	This establishes that $\Spc(P_m)$ is an open embedding. The fact that $\cPm(\num{k},p,h)$ maps to $\cPd(\num{k},p,h)$ follows directly from the definitions.
\end{proof}

The following is our key computational result for determining the underlying set of the spectrum of $\Exc{d}(\Sp^c,\Sp)$. 

\begin{Prop}\label{prop:hderivatives}
	For $1 \le i \le d$ and $j\ge 1$ we have
	\[
		\partial_i(P_d(h_{\mathbb{S}}(j))) \cong \mathbb{S}^{\oplus \left|\surj(i,j)\right|}
	\]
	where the right-hand side is to be interpreted as $0$ if $j>i$.
\end{Prop}

\begin{proof}
	The map $h_\bbS(j)\to P_d(h_{\mathbb{S}}(j))$ induces an equivalence of $\partial_i$ for $i\le d$. So it is enough to prove that 
	\[
		\partial_i(h_{\mathbb{S}}(j)) \cong \mathbb{S}^{\oplus \left|\surj(i,j)\right|}.
	\] 
	Recall that $h_{\bbS}(j)$ is the $j$-th cross-effect of the functor $x\mapsto h_x$, evaluated at $\bbS, \bbS, \ldots, \bbS$. It is well-known that for $x \in \Sp^c$ there is an equivalence $\partial_i(h_x) \simeq \mathbb{D}(x^{\otimes i})$, with $\Sigma_i$ action given by permuting the factors of $x$; see, for example, \cite[Lemma~5.10]{AroneChing15}. In particular, we deduce that 
    \begin{equation}\label{eq:derivative-representable}
		\partial_i(h_{\bbS^{\oplus j}}) \simeq \bbS^{\oplus j^i}.
    \end{equation}
	The case of $h_{\bbS}(j)$ follows by writing $h_{\bbS}(j)$ as the total fiber of a $j$-cube.  Let us demonstrate this in the case $i = j = 3$, showing that $\partial_3(h_{\bbS}(3)) \simeq \bbS^{\oplus 6}$. We can write $h_{\bbS}(3)$ as the total fiber of the $3$-cube
	\[\begin{tikzcd}[row sep={40,between origins}, column sep={40,between origins}]
		& h_{\bbS \oplus \bbS}  \ar{rr}\ar{dd} & & h_{\bbS} \ar{dd}\\
		  h_{\bbS \oplus \bbS \oplus \bbS}  \ar[crossing over]{rr} \ar{dd} \ar{ur} & & h_{\bbS \oplus \bbS} \ar{ur} \\
		& h_{\bbS} \ar{rr} & &  0   \\
		  h_{\bbS \oplus \bbS} \ar{rr} \ar{ur} && h_{\bbS} \ar[from=uu,crossing over] \ar{ur}
	\end{tikzcd}\]
	Applying $\partial_3$ and using~\eqref{eq:derivative-representable} gives the 3-cube
	\[\begin{tikzcd}[row sep={40,between origins}, column sep={40,between origins}]
		& \bbS^{\oplus 8} \ar{rr}\ar{dd} & & \bbS \ar{dd}\\
		  \bbS^{\oplus 27} \ar[crossing over]{rr} \ar{dd} \ar{ur} & & \bbS^{\oplus 8} \ar{ur} \\
		& \bbS \ar{rr} & &  0   \\
		  \bbS^{\oplus 8}  \ar{rr} \ar{ur} && \bbS \ar[from=uu,crossing over] \ar{ur}
	\end{tikzcd}\]
	The total fiber of this diagram, which is $\partial_3(h_{\bbS}(3))$, is $\bbS^{\oplus(27-3\cdot8 + 3\cdot 1)} = \bbS^{\oplus 6}$. In the general case, the same type of argument gives that $\partial_i(h_{\bbS}(j))$ is a sum of
	\[
		j^i - \binom{j}{j-1}(j-1)^i + \binom{j}{j-2}(j-2)^i + \cdots + (-1)^{j-1}j = \sum_{k=0}^{j-1} (-1)^k \binom{j}{j-k}(j-k)^i
	\]
	copies of the sphere spectrum. But this is precisely the number of surjections from $i$ to $j$; see~\eqref{eq:number_or_surjections}. The result then easily follows.
\end{proof}

\begin{Rem}
	Armed with \cref{prop:hderivatives}, we have analogs of \cite[Propositions 4.11 and~4.12]{BalmerSanders17}:
\end{Rem}

\begin{Prop}\label{prop:representable-inclusion}
	Let $1 \le i,j \le d$. Then $P_d(h_{\mathbb{S}}(j)) \in \cPd(\num{i},p,h)$ if and only if $i < j$. 
\end{Prop}

\begin{proof}
	By definition, $\cPd(\num{i},p,h) = (\partial_i)^{-1}(\cat C_{p,h})$. Now recall that $\mathbb{S} \not \in \cat P$ and $0 \in \cat P$ for every prime $\cat P = \cat C_{p,h}$ in $\Spc(\Sp^c)$ and invoke \cref{prop:hderivatives}.
\end{proof}

\begin{Cor}\label{cor:inclusion-layers}
	If $\cPd(\num{i},p,h) \subseteq \cPd(\num{j},p',h')$ then $i \ge j$. 
\end{Cor}

\begin{proof}
	Invoking \cref{prop:representable-inclusion} twice, $P_d(h_{\mathbb{S}}(j)) \not\in \cPd(\num{j},p',h')$ implies that $P_d(h_{\mathbb{S}}(j)) \not\in \cPd(\num{i},p,h)$ by the hypothesis, so that $i \ge j$.
\end{proof}

\begin{Thm}\label{thm:spec-as-a-set}
	The map $\Spc(\partial)$ in \eqref{eq:Spcdelta} is a bijection: Every prime ideal in $\Spc(\Exc{d}(\Sp^c,\Sp)^c)$ is of the form $\cPd(\num{i},p,h)$ for some triple $(i,p,h)$ consisting of an integer $1 \le i \le d$, a prime number $p$ or $p = 0$, and a chromatic height $1 \le h \le \infty$. Moreover, we have $\cPd(\num{i},p,h) = \cPd(\num{j},p',h')$ if and only if $i = j$ and $\cat C_{p,h} = \cat C_{p',h'}$ in $\Sp^c$ (i.e., $h = h'$ and if $h = h' >1$, then also $p = p'$). 
\end{Thm}

\begin{proof}
	We established that $\Spc(\partial)$ is surjective in \cref{lem:cover}, so it remains to show that it is injective, i.e., if $\cPd(\num{i},p,h) = \cPd(\num{j},p',h')$ then $i=j$ and $\cat C_{p,h}=\cat C_{p',h'}$. Indeed, \cref{cor:inclusion-layers} implies $i=j$ while \cref{lem:Spc(i_d)} implies $\cat C_{p,h} = \cat C_{p',h'}$.
\end{proof}

\begin{Cor}\label{lem:support-representables}
	Let $1 \le k \le d$. Then
	\[
		\supp(P_d(h_{\mathbb{S}}(k))) = \SET{\cPd(\num{i},p,h)}{i \ge k}.
	\]
\end{Cor}

\begin{proof}
	By definition,
	\[
		\supp(P_d(h_\mathbb{S}(k))) = \SET{\cat P \in \Spc(\Exc{d}(\Sp^c,\Sp)^c)}{P_d(h_\mathbb{S}(k)) \not \in \cat P}.
	\]
	By \Cref{thm:spec-as-a-set} we know every prime is of the form $\cat P(\num{i},p,h)$ and so the result follows from \Cref{prop:representable-inclusion}. 
\end{proof}

\begin{Rem}
	It follows from  \cref{thm:spec-as-a-set} and \cref{lem:injectivity-of-partial-k} that $\Spc(\Exc{d}(\Sp^c,\Sp)^c)$ consists of $d$ disjoint embedded copies of $\Spc(\Sp^c)$. Understanding the topological interactions between these $d$ pieces will be the topic of \cref{part:III}.
\end{Rem}

\begin{Rem}\label{rem:p-local-derivative}
	The $p$-localization $\Exc{d}(\Sp^c,\Sp)\to\Exc{d}(\Sp^c,\Sp)_{(p)}$ is the finite localization associated to $i_d(f_{p,\infty})$ in the notation of \cref{exa:chromatic-truncation}. It induces an identification
	\[
		\Spc(\Exc{d}(\Sp^c,\Sp)_{(p)}^c) \xrightarrow{\sim} \SET{\cPd(\num{k},p,h)}{\text{all } k, h} \subseteq \Spc(\Exc{d}(\Sp^c,\Sp)^c)
	\]
	on Balmer spectra. This readily follows from \cref{lem:Spc(i_d)} and the definitions. Moreover, it follows from \cref{lem:partial-splits-inflation} that if $F$ is a $p$-local $d$-excisive functor then its derivative $\partial_k(F) \in \Sp$ is also $p$-local for each $1\le k\le d$. This is almost tautological if one makes the identification $\Exc{d}(\Sp^c,\Sp)_{(p)} \simeq \Exc{d}(\Sp^c,\Sp_{(p)})$ provided by \cref{rem:targetcat}. From this perspective, $p$-localization is just changing coefficients along $\Sp\to \Sp_{(p)}$.
\end{Rem}

\begin{Rem}
	Via the abstract nilpotence theorem of \cite[Theorem 2.25]{BCHNP2023Quillen}, the joint conservativity of the derivatives on $\Exc{d}(\Sp^c,\Sp)$ established in \cref{lem:conservativity} implies the following nilpotence theorem:
\end{Rem}

\begin{Thm}\label{thm:nilpotence}
	The collection of functors $\{ \partial_k \}_{1 \le k \le d}$ detects tensor-nilpotence of morphisms with compact source in $\Exc{d}(\Sp^c,\Sp)$: A map $\alpha\colon F \to G$ in $\Exc{d}(\Sp^c,\Sp)$ with $F$ compact satisfies $\alpha^{\circledast m} = 0$ for some $m \ge 0$ if and only if for all $1 \leq k \leq d$ there exists $m_k \geq 0$ such that $\partial_k(\alpha)^{\otimes m_k} = 0$.
\end{Thm}

\begin{Rem}
	The computation of \Cref{thm:spec-as-a-set} along with \Cref{thm:nilpotence} show that Balmer's homological spectrum $\Spc^h(\Exc{d}(\Sp^c,\Sp)^c)$ (see \cite{Balmer20_nilpotence}) agrees with the tensor triangular spectrum. (Balmer has conjectured that this holds for any rigidly-compactly generated category; see \cite[Remark 5.15]{Balmer20_nilpotence}.) More specifically, we have the following:
\end{Rem}

\begin{Cor}\label{cor:bijhyp}
	The natural comparison map 
	\[
		\phi\colon\Spc^h(\Exc{d}(\Sp^c,\Sp)^c) \to \Spc(\Exc{d}(\Sp^c,\Sp)^c)
	\]
	of \cite[Remark 3.4]{Balmer20_nilpotence} is a homeomorphism.
\end{Cor}

\begin{proof}
    We will verify the conditions of \cite[Theorem 5.6]{Balmer20_nilpotence}. This will establish that~$\phi$ is a bijection and hence a homeomorphism by \cite[Theorem A]{bhs2}. The argument is analogous to the corresponding argument for the equivariant stable homotopy category given in \cite[Corollary 5.10]{Balmer20_nilpotence}. We must show that for every $\cat P = \cPd(\num{k},p,h) \in \Spc(\Exc{d}(\Sp^c,\Sp)^c)$ there exists a coproduct-preserving homological symmetric monoidal functor to a tensor abelian category $\cat A_{\cat P}$ satisfying the three conditions of \cite[Theorem 5.6]{Balmer20_nilpotence}. From the definition of the prime ideal $\cPd(\num{k},p,h)$ we take this to be the composite
    \[
		\Exc{d}(\Sp^c,\Sp) \xrightarrow{\partial_k} \Sp \xrightarrow{K(p,h)_{\bullet}} \cat A_{p,h} \coloneqq \Fp[v_{h}^{\pm 1}]\GrMMod
    \]
    where $K(p,h)$ is Morava $K$-theory at the prime $p$ and height $h$ and the graded field $K(p,h)_*(\mathbb{S})=\Fp[v_h^{\pm 1}]$ has $v_h$ in degree $2(p^h-1)$. When $h = 0$ we take $K(p,h)$ to be rational homology (for all primes $p$) and the target to be graded rational vector spaces, while for $h = \infty$ we take $K(p,h)$ to be mod $p$ homology and the target to be graded $\Fp$-vector spaces. Conditions (1) and (2) of \cite[Theorem 5.6]{Balmer20_nilpotence} are then straightforward to verify, while (3) follows from the nilpotence theorem of \cref{thm:nilpotence} together with the classical nilpotence theorem \cite{DevinatzHopkinsSmith88,HopkinsSmith98} for the stable homotopy category.
\end{proof}

\section{The Goodwillie--Burnside ring}\label{sec:goodwillie-burnside}

In this section we introduce, for each integer $d \ge 1$, a commutative ring $A(d)$ which we call the Goodwillie--Burnside ring. The name is justified by \Cref{thm:burnside-ring-isomorphism} in the next section which identifies $A(d)$ with the endomorphism ring of the unit in the category $\Exc{d}(\Sp^c,\Sp)$. It is the analog in Goodwillie calculus of the Burnside ring in equivariant homotopy theory and, following Dress \cite{Dress69}, we are able to completely describe its Zariski spectrum.

\begin{Rem}
	Yoshida \cite{Yoshida87} introduces a Burnside ring for any finite category $\cat C$ satisfying certain assumptions. For example, if $\cat C$ is the orbit category of a finite group, then Yoshida's abstract Burnside ring is the usual Burnside ring of a finite group. Our construction of the Goodwillie--Burnside ring can be regarded as a special case of Yoshida's construction, namely as the abstract Burnside ring associated to the following category.
\end{Rem}

\begin{Def}
    Let $\Epi{d}$ denote the category of non-empty sets of cardinality at most $d$ and surjective maps. 
\end{Def}

\begin{Rem}\label{rem:spectral-mackey}
    Glasman \cite{Glasman18pp} has shown that the category of $d$-excisive functors from finite spectra to spectra is equivalent to the category of spectral Mackey functors on a category built from $\Epi{d}$. This provides a conceptual reason for the appearance of $\Epi{d}$ in the study of $\Exc{d}(\Sp^c,\Sp)$. 
\end{Rem}

\begin{Rem}
	Note that when $d = 2$, there is an equivalence of categories 
	\[
		\Epi{2} \simeq \mathcal{O}_{C_2}
	\]
	where the latter denotes the orbit category of the cyclic group $C_2$. This is reflected in the fact that the ring $A(2)$ we construct in this section is isomorphic to $A(C_2)$, the Burnside ring of the group $C_2$. 
\end{Rem}

\begin{Cons}\label{cons-burnside-homomorphism}
	Let $d \ge 1$ be an integer. Let $A(d) \coloneqq \langle x_1, \ldots, x_d \rangle$ denote the free abelian group on generators $x_1, \ldots, x_d$ and let $\bbZ^d \coloneqq \langle y_1, \ldots, y_d \rangle$ denote the free abelian group on generators $y_1, \ldots, y_d$. Thus $A(d)$ and $\bbZ^d$ are isomorphic as abelian groups, but they will be endowed with different ring structures. The ring $\bbZ^d=\prod_{k=1}^d\bbZ$ is the product of $d$ copies of $\bbZ$. Thus the product structure on $\bbZ^d$ is determined by the relations 
    \begin{equation}\label{eq:product_in_Z^n}
		y_iy_j=\left\{\begin{array}{cc} y_i & i=j \\ 0 & i\ne j\end{array}\right. .
	\end{equation}
	To describe the ring structure on $A(d)$ we first define a group homomorphism $\phi\colon A(d) \to \bbZ^d$ by the following formula:
	\begin{equation}\label{eq:burnside}
		\phi(x_i)\coloneqq\sum_{k=1}^d|\surj(k, i)|y_k=\sum_{k=i}^d|\surj(k, i)|y_k. 
	\end{equation}
	We call $\phi$ the \emph{Burnside homomorphism}.
\end{Cons}

\begin{Rem}
	Recall from \Cref{def:good-subsets} that a subset of $\num{i}\times \num{j}$ is called good if it projects surjectively onto $\num{i}$ and $\num{j}$.
\end{Rem}

\begin{Def}\label{def:mu}
	Let $\mu(i,j,l)$ be the number of good subsets of $\num{i}\times \num{j}$ of cardinality~$l$. 
\end{Def}

\begin{Thm}\label{thm:fundamental-ses-a(n)}
	The homomorphism $\phi$ of \eqref{eq:burnside} is injective. Moreover, there is an exact sequence
    \[
		0 \to A(d) \xrightarrow{\phi} \prod_{1 \le i \le d}\bbZ \xrightarrow{\psi} \prod_{1 \le i \le d} \bbZ/i!\bbZ \to 0.
    \]
    The abelian group $A(d)$ possesses a unique ring structure which makes $\phi$ a ring homomorphism. The multiplication in $A(d)$ is determined by the formula 
	\begin{equation}\label{eq:defining_relations}
		x_i x_j = \sum_{1 \le l \le d} \mu(i,j,l)x_l.
	\end{equation}
\end{Thm}

\begin{proof}
    This is a consequence of Yoshida's work on abstract Burnside rings \cite{Yoshida87}, but we will give a direct proof.

	The formula~\eqref{eq:burnside} means that as a homomorphism from $\bbZ^d$ to itself, with our chosen bases, $\phi$ is represented by a lower triangular matrix, whose diagonal entries are $\surj(i, i)=i!$ for $i=1, \ldots, d$. It follows that $\phi$ is injective, and the cokernel of $\phi$ is isomorphic to $\prod_{1 \le i \le d} \bbZ/i!\bbZ$.

	From the injectivity of $\phi$ it follows that there exists at most one product structure on $A(d)$ for which $\phi$ is a ring homomorphism. It remains to find a formula for multiplication in $A(d)$ that is respected by $\phi$. Since $x_1, \ldots, x_d$ form an abelian group basis of $A(d)$, for any ring structure on $A(d)$ there exist unique integers $\nu(i,j,l)$ such that 
	\begin{equation}\label{eq:product_in_A(n)}
		x_i x_j=\sum_{l=1}^d\nu(i,j,l)x_l.
	\end{equation}
	The map $\phi$ is a ring homomorphism if and only if 
	\[
		\phi(x_i)\phi(x_j)=\phi(x_ix_j)
	\]
	for all $1\le i, j\le n$. Substituting the formula for $\phi$, and using the formulas for the product in $\bbZ^d$~\eqref{eq:product_in_Z^n} and the product in $A(d)$~\eqref{eq:product_in_A(n)}, we obtain the relations
	\[
		\sum_{k=1}^d |\surj(k, i)||\surj(k, j)| y_k=\sum_{k=1}^d\left(\sum_{l=1}^d\nu(i,j,l)|\surj(k,l)|\right) y_k.
	\]
	It follows that $\phi$ is a ring homomorphism if and only if the numbers $\nu(i,j,l)$ satisfy the following relations, for all $1\le i, j, k\le d$:
	\[
		|\surj(k, i)||\surj(k, j)|=\sum_{l=1}^d\nu(i,j,l)|\surj(k,l)|.
	\]
	Since the multiplication on $A(d)$ is unique (if it exists), there is at most one set of numbers $\nu(i,j,k)$ that satisfies these relations. So if we show that the numbers $\mu(i,j,k)$ satisfy these relations, then $\mu(i,j,l)=\nu(i,j,l)$ for all $i, j, l$, and we are done.

	We have to prove that for all $i, j, k$ the following equality holds
	\[
		|\surj(k, i)||\surj(k, j)|=\sum_{l=1}^d\mu(i,j,l)|\surj(k,l)|.
	\]
	The number $|\surj(k, i)||\surj(k, j)|$ is the same as $|\surj(\num{k}, \num{i})\times\surj(\num{k}, \num{j})|$. Let $\num{i}^{\num{k}}$ denote the set of all functions from $\num{k}$ to $\num{i}$. Then $\surj(\num{k}, \num{i})\subseteq \num{i}^{\num{k}}$ and
	\[
		\surj(\num{k}, \num{i})\times\surj(\num{k}, \num{j})\subseteq \num{i}^{\num{k}}\times \num{j}^{\num{k}}\cong \left(\num{i}\times \num{j}\right)^{\num{k}}.
	\]
	A function $\alpha\colon \num{k}\to \num{i}\times \num{j}$, considered as an element of $\left(\num{i}\times \num{j}\right)^{\num{k}}$, is in the image of $\surj(\num{k}, \num{i})\times\surj(\num{k}, \num{j})$ if and only if the image of $\alpha$ is a good subset of $\num{i}\times \num{j}$. It follows that the number of elements of $\surj(\num{k}, \num{i})\times\surj(\num{k}, \num{j})$ is the sum, indexed by good subsets of $\num{i}\times \num{j}$, of the numbers of surjections from~$\num{k}$ onto each good subset. Grouping the good subsets by cardinality, and recalling that $\mu(i,j,l)$ is the number of good subsets with $l$ elements, we obtain the desired formula. 
\end{proof}

\begin{Cor}\label{cor:polynomial}
    The ring $A(d)$ is isomorphic to the quotient of the polynomial ring $\bbZ[x_1, \ldots, x_d]$ by the ideal $I$ generated by the relations~\eqref{eq:defining_relations}. Furthermore, $x_1=1$, so $A(d)$ is also the quotient of $\bbZ[x_2, \ldots, x_d]$ by the relations~\eqref{eq:defining_relations}.
\end{Cor}

\begin{proof}
    It follows from \Cref{thm:fundamental-ses-a(n)} that there is a surjective ring homomorphism $\bbZ[x_1, \ldots, x_d]/I\twoheadrightarrow A(d).$ It is easy to see from the definition of $I$ that $x_1, \ldots, x_d$ generate $\bbZ[x_1, \ldots, x_d]/I$ as an abelian group. Since $A(d)$ is a free abelian group of rank $d$, it follows that the surjective homomorphism is in fact an isomorphism.
\end{proof}

\begin{Rem}
	There does not seem to be a closed formula for $\mu(i,j,k)$. Nevertheless, we can make a few observations about these numbers, and also write some summation formulas.

	To begin with, clearly $\mu(i,j,k)\ne 0$ if and only if $\max(i, j)\le k\le ij$. Moreover, if $i\le j$ then $\mu(i,j,j)=|\surj(j, i)|$. At the other end of the range, $\mu(i, j, ij)=1$.

	Note that $\mu(i,j,k)$ can be characterized as the number of labelled bipartite graphs with $k$ edges and no isolated vertices, and with parts of size $i$ and $j$. In this guise, there are some formulas for this number, or closely related ones, scattered in the literature~\cite{Harary58, Lee61, Atmaca18}. For example, one can use the inclusion-exclusion principle to write a summation formula for $\mu(i, j, k)$. We learned the following formula from Achim Krause. It can be found in~\cite{Lee61} as the Corollary to Lemma~3 on page 220.
\end{Rem}

\begin{Lem}\label{lem:summationformula}
	There is an equality
	\[
		\mu(i,j,k)=\sum_{s \ge 0, t \ge 0}(-1)^{s+t}\binom{(i-s)(j-t)}{k}  \binom{i}{s}  \binom{j}{t}.
	\]
\end{Lem}

\begin{proof}
    There are altogether $\binom{ij}{k}$ subsets of $\num{i}\times \num{j}$ of cardinality $k$. Let us say that a subset of $\num{i}\times \num{j}$ is \emph{bad} if it is not good. We are going to apply the inclusion-exclusion principle to analyze the number of bad subsets, and ultimately arrive at the number of good ones. 
    
    A subset $U\subseteq \num{i}\times \num{j}$ is bad if there exists a row or a column of the array $\num{i}\times \num{j}$ that is disjoint from $U$. Suppose we have subsets $S\subseteq \num{i}$ and $T\subseteq \num{j}$. Let $B_{S, T}$ be the number of bad subsets of $\num{i}\times \num{j}$ with $k$ elements with the property that their projections onto $\num{i}$ and $\num{j}$ are disjoint from $S$ and $T$ respectively. In other words, $B_{S, T}$ is the number of subsets of $(\num{i}\setminus S)\times (\num{j}\setminus T)$ of cardinality $k$. It follows that $B_{S, T}$ has $\binom{(i-|S|)(j-|T|)}{k}$ elements. 

	It is clear that the union of $B_{S, T}$ where $S, T$ range over all subsets of $\num{i}$ and~$\num{j}$ with $S\cup T\ne \emptyset$ is the set of bad subsets of $\num{i}\times \num{j}$ of cardinality $k$. It is also clear that $B_{S, T}\cap B_{S', T'}=B_{S\cup S', T\cup T'}$. The maximal elements of the family of subsets $B_{S, T}$ are those where $S\cup T$ consists of a single element. The set $B_{S, T}$ is the intersection of $|S|+|T|$ maximal elements. 

	Applying the inclusion-exclusion principle to the family of subsets $B_{S, T}$ we obtain that the number of elements in the complement of the union of all the sets $B_{S, T}$, which is $\mu(i,j,k)$, is given by the following formula:
	\begin{align*}
		\mu(i,j,k)&=\sum_{S\subseteq \num{i}, T\subseteq \num{j}} (-1)^{|S|+|T|}|B_{S, T}| \\ &=\sum_{S\subseteq \num{i}, T\subseteq \num{j}} (-1)^{|S|+|T|} \binom{(i-|S|)(j-|T|)}{ k}\\&=\sum_{s\ge 0, t\ge 0} (-1)^{s+t} \binom{(i-s)(j-t)}{k}\binom{i}{s}\binom{j}{t}.\qedhere
	\end{align*}
\end{proof}

\begin{Rem}\label{rem:tom}
	For an explicit way to compute the ring structure, we define a $d\times d$ matrix $M$ by
	\[
		M_{ij}  \coloneqq |\surj(i,j)|.
	\]
	Then one can show that 
	\begin{equation}
		\mu(i,j,k) = \sum_{m=1}^d M_{mi}\cdot M_{mj}\cdot M^{-1}_{km},
	\end{equation}
	see for example \cite[Equation (1.6)]{YoshidaOdaTakegahara18}. Moreover, we can give an explicit description of the inverse matrix $M^{-1}$. Indeed, recall that the number of surjections from $\num{i}$ to $\num{j}$ can be expressed in terms of Stirling numbers of the second kind:
	\[
		M_{ij} = |\surj(i,j)| = j!\stirling{i}{j},
	\]
	see for example \cite[Theorem 5.3.1]{Biggs89}. We can therefore write $M = LU$ where $L_{ij} = \stirling{i}{j}$ and $U$ is the diagonal matrix with $U_{ii} = i!$. The inverse of $L$ is the matrix $L_{ij}^{-1} = s(i,j)$ where $s(-,-)$ denotes a Stirling number of the first kind (this follows from \cite[Proposition 1.4.1]{Stanley12}). Therefore, 
	\[
		M^{-1}_{ij} = \frac{1}{i!}s(i,j).
	\]
	This gives the formula
	\begin{equation}
		\mu(i,j,k) = \sum_{m=1}^d i!\stirling{m}{i}j!\stirling{m}{j}\frac{1}{k!}s(k,m).
	\end{equation}
	Note that the terms in the sum can only be non-zero for $\max(i,j)  \le m \le k$. 
\end{Rem}

\begin{Exa}
	We now give an explicit example of how to compute the ring structure in the $d=3$ case.
 
	We can write this ring as a quotient of $\bbZ[x_2,x_3]$ by \cref{cor:polynomial}. The matrix of \Cref{rem:tom} for $d = 3$ is:
	\[
	\begin{tabular}{@{}lll@{}}
	\toprule
	1     & 0     & 0     \\
	1     & 2     & 0     \\
	1     & 6     & 6     \\
	\bottomrule
	\end{tabular}\label{table:toma3}
	\]
	\noindent
	The generators $x_2,x_3$ correspond to the second and third columns of the matrix, and the product of two generators is given by component-wise multiplication of column vectors, which are decomposed as a linear combination of the columns. In particular, we see that $A(3)$ is the quotient of $\bbZ[x_2,x_3]$ by the following relations: 
	\[
	\begin{split}
		(0,0,6) \cdot (0,0,6) = (0,0,36) = 6 \cdot (0,0,6) &\implies x_3^2 = 6x_3 \\
		(0,2,6) \cdot (0,0,6) = (0,0,36) = 6 \cdot (0,0,6) &\implies x_2x_3 = 6x_3 \\
		(0,2,6) \cdot (0,2,6) = (0,4,12) = 2 \cdot (0,2,6) + 4 \cdot (0,0,6) & \implies x_2^2 = 2x_2 + 4x_3.
	\end{split}
	\]
\end{Exa}

\begin{Rem}
	We will now determine all the primes ideals in $A(d)$ and describe the Zariski spectrum $\Spec(A(d))$. These ideas closely follow those for the classical Burnside ring \cite{Dress69,Dress73}. 
\end{Rem}

\begin{Def}\label{def:burnside-homomorphism}
    Let $\phi_i \colon A(d) \to \bbZ$ denote the homomorphism defined by $\phi_i(x_j) \coloneqq |\surj(i,j)|$.
\end{Def}

\begin{Rem}
    In other words, the map $\phi \colon A(d) \to \prod_{1 \le i \le d} \bbZ$  from \Cref{cons-burnside-homomorphism} is the product of the maps $\phi_i$ for $1 \le i \le d$. In particular, each $\phi_i$ is a ring homomorphism by \cref{thm:fundamental-ses-a(n)}.
\end{Rem}

\begin{Def}\label{def:prime_ideals}
	Let $\mathfrak p(\num{i},p)$ (for $p$ a prime or 0 and $1 \le i \le d$) be the preimage of the prime ideal $(p)$ under the map $\phi_i \colon A(d) \to \bbZ$. 
\end{Def}

\begin{Prop}
	Every prime ideal of $A(d)$ is of the form $\mathfrak p(\num{i},p)$ for some~$p$ and $1 \le i \le d$.
\end{Prop}

\begin{proof}
	We observe that the injection $\phi \colon A(d) \hookrightarrow \prod_{1 \le k \le d}\bbZ$ of \Cref{thm:fundamental-ses-a(n)} is an integral extension since $\prod_{1 \le k \le d}\bbZ$ is additively generated by idempotent elements. It then follows from the going-up theorem that $\Spec(\prod_{1 \le k \le d}\bbZ) \to \Spec(A(d))$ is surjective; see e.g.~\cite[Section 1.6]{Kaplansky74}. 
\end{proof}

\begin{Lem}
	The maximal ideals of $A(d)$ are $\mathfrak p(\num{i},p)$ for $p>0$, while the minimal prime ideals are $\mathfrak p(\num{i},0)$. In particular, $A(d)$ has Krull dimension 1. Moreover, we have $\mathfrak p(\num{i},0) \subseteq \mathfrak p(\num{j},p)$ if and only if $\mathfrak p(\num{i},p) = \mathfrak p(\num{j},p)$. 
\end{Lem}

\begin{proof}
	This is the same as the classical proof for the Burnside ring: If $p>0$, then the quotient of $A(d)$ by $\mathfrak p(\num{i},p)$ is $\bbZ/p$, so that $\mathfrak p(\num{i},p)$ is maximal. If $p = 0$, then the quotient is $\bbZ$, and there cannot be any containment among the ideals $\mathfrak p(\num{i},0)$ for varying $i$, as such an inclusion would correspond to a surjective ring homomorphism $\bbZ \to \bbZ$. We deduce that $\mathfrak p(\num{i},0)$ is minimal for each $i$. This then shows that $A(d)$ has Krull dimension~1. 

	Suppose then that $\mathfrak p(\num{i},p) = \mathfrak p(\num{j},p)$. Then clearly $\mathfrak p(\num{i},0) \subseteq \mathfrak p(\num{i},p) = \mathfrak p(\num{j},p)$. Conversely, we will show that if $\mathfrak p(\num{i},p) \ne \mathfrak p(\num{j},p)$, then $\mathfrak p(\num{i},0) \not \subseteq \mathfrak p(\num{j},p)$. From the previous paragraph, both $\mathfrak p(\num{j},p)$ and $\mathfrak p(\num{i},p)$ are maximal ideals; in particular, there exists $a \in A(d)$ with $a \in \mathfrak p(\num{i},p)$ and $a \not \in \mathfrak p(\num{j},p)$. Now consider the element $b = a - \phi_i(a) \cdot 1$. We see that $b \in \mathfrak p(\num{i},0)$ but $b \not \in \mathfrak p(\num{j},p)$, and so $\mathfrak p(\num{i},0) \not \subseteq \mathfrak p(\num{j},p)$, as claimed.
\end{proof}

\begin{Prop}
	Suppose $p, q$ are primes and $1\le i, j\le n$. We have $\mathfrak p(\num{i},p) = \mathfrak p(\num{j},q)$ in $A(d)$ if and only if $p = q$ and $p-1 \mid j-i$. 
\end{Prop}

\begin{proof}
	($\Leftarrow$) Suppose first that $p=q$ and $p-1$ divides $j-i$. We claim that in this case $\phi_i \equiv \phi_j \text{ mod } (p)$, so that $\mathfrak p(\num{i},p) = \mathfrak p(\num{j},p)$. To prove that $\phi_i \equiv \phi_j \text{ mod } (p)$, we need to prove that $\phi_i(x_\ell) \equiv \phi_j(x_\ell) \text{ mod } (p)$, for all $\ell$. This in turn means that we need to prove that $\left| \surj(i, \ell)\right| \equiv \left|\surj(j, \ell)\right| \text{ mod } (p)$ for all $\ell$. Note that the symmetric group $\Sigma_\ell$ acts freely on the set of surjections $\surj(-, \ell)$. It follows that both numbers $\left|\surj(i, \ell)\right|$ and $\left|\surj(j, l)\right|$ are divisible by $\ell !$. So if $\ell$ is at least $p$ then both numbers are divisible by $p$ and we are done. It remains to prove the case when $1 \le \ell \le p$. In this case we use~\eqref{eq:number_or_surjections} to see that:
	\begin{equation}\label{eq:sur}
	\begin{split}
		\left|\surj(i, \ell)\right| & = \ell^i - \binom{\ell}{\ell - 1}(\ell-1)^i +\cdots \pm \binom{\ell}{k} k^i \cdots \\
		\left|\surj(j, \ell)\right| & = \ell^j - \binom{\ell}{\ell - 1}(\ell-1)^j + \cdots \pm \binom{\ell}{k} k^j \cdots 
	\end{split}
	\end{equation}
	But if $p-1$ divides $j-i$ then $k^i \equiv k^j \text{ mod } (p)$ for all $k$ and we are done.

	($\Rightarrow$) For the converse, suppose first that $i$ and $j$ are arbitrary, and $p\ne q$ are distinct primes. Then $\phi_i(p)=\phi_j(p)=p\not \equiv 0 \text{ mod } (q)$. It follows that $p\in \mathfrak p(\num{i},p)$ but $p \notin \mathfrak p(\num{j},q)$, so $\mathfrak p(\num{i},p)\ne \mathfrak p(\num{j},q)$.

	Now suppose that $p=q$ and $p-1 \nmid j-i$. Then there exists a positive integer $a$ such that $a^i \not \equiv a^j \text{ mod } (p)$. For example, a primitive root modulo $p$ satisfies this. Let $\ell$ be the smallest positive integer for which $\ell^i \not \equiv \ell^j \text{ mod } (p)$. 

	We claim that $\ell\le\max(i, j)$ and therefore $\ell\le d$. Indeed, suppose first that $p\le \max(i, j)$. Then $\ell$ is bounded above by the primitive roots modulo $p$, which are all smaller than $\max(i, j)$. Now suppose that $p>\max(i, j)$. Without loss of generality, suppose $i<j$. For integers $0<x<p$, the condition $x^i\equiv x^j \text{ mod } (p)$ is equivalent to $x^{j-i}\equiv 1 \text{ mod } (p)$. This equation has $\gcd(j-i, p-1)$ solutions between $1$ and $p$. It follows that among the numbers $1, \ldots, j-i+1$ there exists at least one $x$ for which $x^{j-i}\not \equiv 1 \text{ mod } (p)$ and thus $x^i\not \equiv x^j \text{ mod } (p)$. It follows that $l\le j-i+1\le j$. This proves the claim. It follows that $x_\ell\in A(d)$.

	Since $\ell$ is the smallest positive integer for which $\ell^i \not \equiv \ell^j \text{ mod } (p)$, it follows, using the formulas in~\eqref{eq:sur}, that 
	\[
		\left|\surj(i, l)\right| \not \equiv \left| \surj(j, l)\right| \text{ mod } (p).
	\]
	Consider the expression $x_\ell - |\surj(i, l)|\in A(n)$. Then
	\[
		\phi_i\left(x_\ell - |\surj(i, l)|\right)=|\surj(i, l)|-|\surj(i, l)| \equiv 0 \text{ mod } (p),
	\]
	while
	\[
		\phi_j\left(x_\ell - |\surj(i, l)|\right)=|\surj(j, l)|-|\surj(i, l)| \not\equiv 0 \text{ mod } (p).
	\]
	It follows that $\mathfrak p(\num{i},p)\ne \mathfrak p(\num{j},p)$.
\end{proof}

In summary:

\begin{Thm}\label{thm:burnside-goodwillie-spectrum}
	Let $1 \le i,j \le d$. Then:
	\begin{enumerate}
		\item The prime ideals of $A(d)$ are precisely the $\mathfrak p(\num{i},p)$ defined in \Cref{def:prime_ideals}. The prime ideals $\mathfrak p(\num{i},p)$ for $p>0$ are maximal prime ideals, and the prime ideals $\mathfrak p(\num{i},0)$ are minimal prime ideals. 
		\item For $p,q >0$ and $i \ne j$ we have  $\mathfrak p(\num{i},p) = \mathfrak p(\num{j},q)$ if and only if $p = q$ and $p-1 \mid j-i$. 
		\item We have an inclusion $\mathfrak p(\num{i},p) \subseteq \mathfrak p(\num{j},q)$ if and only if any of the following are satisfied:
		\begin{enumerate}[label=(\roman*)]
			\item $p=q>0$ and $\mathfrak p(\num{i},p) = \mathfrak p(\num{j},p)$
			\item $p = q = 0$ and $i = j$
			\item $p = 0,q > 0$ and $\mathfrak p(\num{i},q) = \mathfrak p(\num{j},q)$.
		\end{enumerate}
	\end{enumerate}
\end{Thm}

\begin{Rem}
	Since the ring $A(d)$ is noetherian (\cref{cor:polynomial}), the Zariski topology is completely determined by the inclusions among prime ideals. Thus, \cref{thm:burnside-goodwillie-spectrum} provides a complete description of the Zariski spectrum of $A(d)$. It consists of~$d$ copies of $\Spec(\bbZ)$ with certain closed points glued together, namely, the closed point~$p$ in the $i$-th copy is glued together with the closed point $p$ in the $j$-th copy precisely when $p-1 \mid j-i$.
\end{Rem}

\begin{Exa}
	The Zariski spectrum of $A(3)$ is shown in \Cref{fig:zarizki_a3}.
	\begin{figure}[htbp]
	\includegraphics{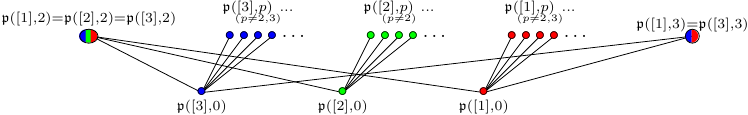}\caption{The Zariski spectrum of the ring $A(3)$.}\label{fig:zarizki_a3}
	\end{figure}
\end{Exa}

\section{The comparison map to the Goodwillie--Burnside ring}\label{sec:comparison-map}

We now investigate the endomorphism ring of the unit object in $\Exc{d}(\Sp^c,\Sp)$ and show that it is isomorphic to the Goodwillie--Burnside ring defined in \cref{sec:goodwillie-burnside}. This identification is significant due to the following comparison map constructed by Balmer \cite{Balmer10b}:

\begin{Def}\label{def:comparison-map}
    For an essentially small tensor triangulated category $\cat K$, the \emph{comparison map}
    \[
		\rho_{\cat K} \colon \Spc(\cat K) \to \Spec(\End_{\cat K}(\unit))
    \]
    is the continuous map defined by $\cat P \mapsto \SET{f \colon \unit \to \unit}{\cofib(f) \not \in \cat P}$.
\end{Def}

\begin{Rem}
    This map is always inclusion-\emph{reversing}:  If $\cat P \subseteq \cat Q$ are two primes of $\cat K$, then $\rho_{\cat K}(\cat Q) \subseteq \rho_{\cat K}(\cat P)$. 
\end{Rem}

\begin{Rem}
	Applied to (the homotopy category of) a symmetric monoidal stable $\infty$-category $\cat C$, note that $\End_{h\cat C}(\unit) = \pi_0(\Hom_{\cat C}(\unit,\unit))$.
\end{Rem}

\begin{Exa}\label{exa:comparison-map-for-spectra}
	The comparison map
    \[
		\Spc(\Sp^c) \to \Spec(\pi_0\Hom_{\Sp}(\bbS,\bbS)) \cong \Spec(\bbZ)
    \]
    sends $\cat C_{p,h}$ to $(p)$ and $\cat C_{0,1}$ to $(0)$; see \cref{exa:balmer-spectrum-spectra}. 
\end{Exa}

\begin{Rem}
	Recall from \cref{thm:properties-n-excisive} that the monoidal unit of $\Exc{d}(\Sp^c,\Sp)$ is the $d$-excisive functor $P_d h_\bbS$. We are therefore interested in the following ring:
\end{Rem}

\begin{Def}
    We let
	\[
		R(d) \coloneqq \pi_0\Hom_{\Exc{d}(\Sp^c,\Sp)}(P_dh_{\bbS},P_dh_{\bbS})
	\]
    denote the endomorphism ring of the unit object in the category of $d$-excisive functors from finite   spectra to spectra. 
\end{Def}

\begin{Rem}
	For typesetting reasons, we'll sometimes drop or simplify the subscript indicating the ambient category and just write, for example, $\Hom(P_d h_{\bbS},P_d h_{\bbS})$  or $\Hom_{\Exc{d}}(P_d h_{\bbS},P_d h_{\bbS})$ instead of $\Hom_{\Exc{d}(\Sp^c,\Sp)}(P_dh_{\bbS},P_dh_{\bbS})$.
\end{Rem}

\begin{Rem}
    By \Cref{lem:cross-effect-yoneda-lemma} and \Cref{exa:first-cross-effect} we have
    \[
		\Hom_{\Exc{d}(\Sp^c,\Sp)}(P_dh_{\bbS},P_dh_{\bbS}) \simeq P_d\Sigma^{\infty}\Omega^{\infty}(\mathbb{S}).
    \]
	A well-known theorem of Goodwillie shows that the Goodwillie tower of the functor $X \mapsto \Sigma^{\infty}\Omega^{\infty}(X)$ splits when evaluated on connected suspension spectra \cite[Example 6.1]{Kuhn07}, with $(D_d\Sigma^{\infty}\Omega^{\infty})(X) = X^{\otimes d}_{h\Sigma_d}$. In particular, we have the following analog of the tom Dieck splitting in the functor calculus world:
    \[
	   P_d\Sigma^{\infty}\Omega^{\infty}(\mathbb{S}) \cong \bigoplus_{1 \le i \le d}(\Sigma^{\infty}S^0)_{h\Sigma_i}. 
    \]
    This implies that $R(d) \cong \prod_{1 \le i \le d}\mathbb{Z}$ as an abelian group. Our next goal is to analyze how this interacts with the derivatives, and then use this to determine $R(d)$ as a commutative ring.
\end{Rem}

\begin{Rem}
	Recall that $\partial_i(P_dh_{\mathbb S}) \simeq \mathbb S$ for $1\le i\le d$ by \Cref{prop:hderivatives}. Hence, the $i$-th derivative functor induces a map, which we also denote by $\partial_i$:
	\[
		\partial_i\colon \Hom_{\Exc{d}(\Sp^c, \Sp)}(P_dh_{\bbS}, P_dh_{\bbS})\to \Hom_{\Sp}(\partial_i h_{\bbS}, \partial_i h_{\bbS})\simeq \bbS.
	\]
	Applying $\pi_0$ we obtain a ring homomorphism $\theta_d^i\colon R(d)\to \bbZ$. Taken together, we obtain a ring homomorphism $\theta_d \coloneqq \prod_{i=1}^d \theta_d^i$:
	\[
		\theta_d\colon R(d)\to \prod_{i=1}^d \bbZ.
	\]
	Similarly, the functor $P_{d-1}\colon \Exc{d}(\Sp^c,\Sp)\to\Exc{d-1}(\Sp^c,\Sp)$ induces a map, which we also denote simply by $P_{d-1}$:
	\[
		P_{d-1}\colon \Hom_{\Exc{d}(\Sp^c,\Sp)}(P_dh_{\bbS}, P_dh_{\bbS})\to \Hom_{\Exc{d-1}(\Sp^c,\Sp)}(P_{d-1}h_{\bbS}, P_{d-1}h_{\bbS}).
	\]
	Applying $\pi_0$ we obtain a ring homomorphism $\pi_0(P_{d-1})\colon R(d)\to R(d-1).$ 
\end{Rem}

\begin{Prop}\label{prop:injective}
	There is an isomorphism of abelian groups $R(d)\cong \bbZ^d$. The ring homomorphism $\theta_d$ is injective and the ring homomorphism $\pi_0(P_{d-1})$ is surjective for all $d$. Furthermore, there is a commutative diagram, where the columns are split short exact sequences of abelian groups:
	\begin{equation}\label{diag:R(n)}
	\begin{tikzcd}[ampersand replacement=\&]
		{\mathbb Z} \& {\mathbb Z} \\
		{R(d)} \& {\displaystyle\prod_{i=1}^d\mathbb Z} \\
		{R(d-1)} \& {\displaystyle\prod_{i=1}^{d-1}\mathbb Z}
		\arrow["{d!}", from=1-1, to=1-2]
		\arrow[from=1-2, to=2-2]
		\arrow[from=2-2, to=3-2]
		\arrow[from=1-1, to=2-1]
		\arrow["{\pi_0(P_{d-1})}"', from=2-1, to=3-1]
		\arrow["{\theta_d}", from=2-1, to=2-2]
		\arrow["{\theta_{d-1}}"', from=3-1, to=3-2]
	\end{tikzcd}
	\end{equation}
	Here the vertical maps on the right are the inclusion of the last factor and projection onto the first $d-1$ factors respectively.
\end{Prop}

\begin{proof}
	The homomorphism $\theta_d$ is defined by taking the first $d$ derivatives of a natural transformation $P_dh_{\bbS}\to P_dh_{\bbS}$. Since $\partial_iP_{d-1}\simeq \partial_i$ for $i\le d-1$, and $\partial_dP_{d-1}\simeq 0$, it follows that the bottom square of \eqref{diag:R(n)} commutes. We will see that the top row of the diagram is the kernel of the vertical homomorphisms in the bottom square, so the entire diagram commutes.

	For each $d\ge 1$, consider the following commutative diagram, where the vertical maps are induced by the natural transformation $p_{d-1}\colon P_d\to P_{d-1}$:
	\[\begin{tikzcd}[ampersand replacement=\&]
		{P_dh_{\bbS}(\bbS)} \& {\Hom(P_dh_{\bbS},P_dh_{\bbS})} \& {\Hom(h_{\bbS},P_dh_{\bbS})} \\
		{P_{d-1}h_{\bbS}(\bbS)} \& {\Hom(P_{d-1}h_{\bbS},P_{d-1}h_{\bbS})} \& {\Hom(h_{\bbS},P_{d-1}h_{\bbS}).}
		\arrow["\simeq", from=1-1, to=1-2]
		\arrow["\simeq", from=1-2, to=1-3]
		\arrow[from=1-1, to=2-1]
		\arrow["\simeq"', from=2-1, to=2-2]
		\arrow[from=1-2, to=2-2]
		\arrow["\simeq"', from=2-2, to=2-3]
		\arrow[from=1-3, to=2-3]
	\end{tikzcd}\]
	The horizontal maps are equivalences by the Yoneda lemma and the universal property of $P_d$. In particular, $R(d)\cong \pi_0\left(P_dh_{\bbS}(\bbS)\right)$, and the homomorphism $\pi_0(P_{d-1})\colon R(d) \to R(d-1)$ can be identified with the natural homomorphism $\pi_0(p_{d-1})\colon \pi_0(P_{d}h_{\bbS}(\bbS))\to \pi_0(P_{d-1}h_{\bbS}(\bbS)).$

	Now consider the fiber sequence
	\begin{equation}\label{eq:split}
		D_d h_{\bbS}(\bbS)\to P_{d}h_{\bbS}(\bbS)\to P_{d-1}h_{\bbS}(\bbS).
	\end{equation}
	Recall (for example via \Cref{prop:n-homogeneous-functors}) that $D_d h_{\bbS}(X)\simeq X^{\otimes d}_{h\Sigma_d}$. Taking $X=\bbS$, we have $D_d h_{\bbS}(\bbS)\simeq \Sigma^\infty {B\Sigma_d}_+$ and $\pi_0\left(D_d h_{\bbS}(\bbS)\right)\cong \bbZ$. Therefore, applying $\pi_0$ to the fibration sequence~\eqref{eq:split} yields an exact sequence
	\begin{equation}\label{eq:SES}
		\bbZ\to R(d)\xrightarrow{\pi_0(p_{d-1})} R(d-1).
	\end{equation}
	We claim that this sequence is in fact short exact. It will be the left column of~\eqref{diag:R(n)}.

	We will provide two proofs that the sequence above is short exact. Recall that $h_\bbS\simeq \Sigma^\infty\Omega^\infty$, so $h_\bbS(\bbS)\simeq \Sigma^\infty\Omega^\infty\Sigma^\infty S^0$. Since $\Sigma^\infty$ is a left adjoint, composing with~$\Sigma^\infty$ commutes with $P_d$. It follows that the Taylor tower of the functor $h_{\bbS}$ evaluated at $\bbS$ is the same as the Taylor tower of the functor $\Sigma^\infty\Omega^\infty\Sigma^\infty$ evaluated at the space $S^0$.

The following lemma is well-known, but we did not find an explicit statement of it in the literature. It was in a preprint version of~\cite{Goodwillie03}, but was not included in the published version (but see~\cite[Example 1.20]{Goodwillie03}, which is closely related).
\begin{Lem}
The Taylor tower of the functor $\Sigma^\infty
\Omega^\infty\Sigma^\infty$ splits. This means that for all $d$, the natural map $P_d\Sigma^\infty\Omega^\infty\Sigma^\infty \to P_{d-1}\Sigma^\infty\Omega^\infty\Sigma^\infty$ has a section.  
\end{Lem}
\begin{proof}
The lemma is a consequence of the Snaith splitting, so let us recall what that is. It is well-known that for connected pointed spaces $X$ there is a natural equivalence
    \begin{equation}\label{eq:Snaith}
    \Sigma^\infty\Omega^\infty\Sigma^\infty X \simeq \bigvee_{n=1}^\infty (\Sigma^\infty X)^{\otimes n}_{h\Sigma_n}.
    \end{equation}
    This is known as the Snaith splitting. The original references are~\cite{Snaith74,Kahn1978}, a more general version is proved in~\cite{CohenMayTaylor78}, and a different, elegant proof was given in~\cite{CohenR80}. More precisely, the references above construct a certain functor~$C$, equipped with a natural transformation $C\to \Omega^\infty\Sigma^\infty$ that is an equivalence on connected spaces, and such that $\Sigma^\infty C(X)$ is naturally equivalent to the right-hand side of \eqref{eq:Snaith} for all $X$. This implies that the Taylor tower of $\Sigma^\infty C$ splits. The fact that the map $\Sigma^\infty C(X)\to \Sigma^\infty\Omega^\infty X$ is an equivalence for connected $X$ implies that the induced map of Taylor approximations $P_d\Sigma^\infty C\to P_d\Sigma^\infty\Omega^\infty\Sigma^\infty $ is an equivalence of functors (the proof is similar to~\cite[Proposition 5.1]{Goodwillie91}, but easier). It follows that the Taylor tower of $\Sigma^\infty\Omega^\infty\Sigma^\infty$ splits as well.
\end{proof}

    It follows that the fiber sequence~\eqref{eq:split} splits, and therefore it induces a split short exact sequence on $\pi_0$.

	We now give another proof that \eqref{eq:SES} is short exact which will, moreover, establish that the homomorphism at the top of \eqref{diag:R(n)} is multiplication by $d!$. 

	The functor $\partial_d\colon \Fun(\Sp^c, \Sp)\to \Sp$ factors through the category of spectra with an action of $\Sigma_d$. It follows that the map 
	\[
		\partial_d\colon \Hom(P_dh_{\bbS}, P_dh_{\bbS})\to \Hom_{\Sp}(\partial_dh_\bbS,\partial_dh_\bbS)\simeq\Hom_{\Sp}(\bbS,\bbS)\simeq\bbS
	\]
	factors through $\Hom_{\Sp}(\bbS,\bbS)^{h\Sigma_d}$. The map fits into the following diagram
	\[\begin{tikzcd}[ampersand replacement=\&]
		{\bbS_{h\Sigma_d}} \& {\bbS_{h\Sigma_d}} \\
		{\Hom(P_dh_{\bbS}, P_dh_{\bbS})} \& {\Hom_{\Sp}(\bbS,\bbS)^{h\Sigma_d}} \& {\Hom_{\Sp}(\bbS,\bbS)\simeq\bbS} \\
		{\Hom(P_{d-1}h_{\bbS}, P_{d-1}h_{\bbS})} \& {\Hom_{\Sp}(\bbS,\bbS)^{t\Sigma_d}}
		\arrow[from=1-1, to=2-1]
		\arrow[from=2-1, to=3-1]
		\arrow[from=2-1, to=2-2]
		\arrow[from=2-2, to=2-3]
		\arrow[from=2-2, to=3-2]
		\arrow[from=3-1, to=3-2]
		\arrow["\simeq", from=1-1, to=1-2]
		\arrow[from=1-2, to=2-2]
	\end{tikzcd}\]
	Here the square at the bottom is a special case of the Kuhn--McCarthy pullback square (\Cref{prop:tatesquare}) applied to the functor $P_dh_{\bbS}$ and evaluated at $\bbS$. To see this, use the identification $\Hom(P_dh_{\bbS}, P_dh_{\bbS})\simeq P_dh_{\bbS}(\bbS)$, and see~\cite[Proposition~4.14]{AroneChing15} which, together with its proof, shows that the top map in the Kuhn--McCarthy square is indeed induced by our map $\partial_d$.

	It follows that the map $\bbS_{h\Sigma_d}\to \bbS$ from the top left corner to the far right spot in the diagram is the transfer map, whose effect on $\pi_0$ is the homomorphism $\bbZ\xrightarrow{d!}\bbZ$. In particular, this homomorphism is injective. It follows that the map $S_{h\Sigma_d}\to \Hom(P_dh_{\bbS},P_dh_{\bbS})$ induces a monomorphism on $\pi_0$. But this is the same as the map $D_dh_{\bbS}(\bbS)\to P_dh_{\bbS}(\bbS)$ in \eqref{eq:split}. Note also that $\pi_{-1}(\bbS_{h\Sigma_d})=0$, and it follows that the map $p_{d-1}\colon P_dh_{\bbS}(\bbS)\to P_{d-1}h_{\bbS}(\bbS)$ induces a surjective homomorphism on $\pi_0$. We have proved (for the second time) that the fibration sequence \eqref{eq:split} induces a short exact sequence on $\pi_0$. Furthermore, we have proved that the map $\partial_d\colon \Hom(P_dh_{\bbS},P_dh_{\bbS})\to \bbS$ induces multiplication by $d!$ from the kernel of $\pi_0(p_{d-1})$ (as defined in \eqref{diag:R(n)}) to $\pi_0(\Hom_{\Sp}(\partial_dh_{\bbS}, \partial_dh_{\bbS}))\cong \bbZ$.

	We have shown that the left column of \eqref{diag:R(n)} is a short exact sequence and that the upper horizontal homomorphism is multiplication by $d!$. By induction on $d$ it follows that $\theta_d$ is a monomorphism for all $d$, and $R(d)\cong \bbZ^d.$ For the basis of the induction one can take the case $d=0$, in which case $\theta_0$ is the homomorphism from the trivial group to itself. Or one can begin with $d=1$, in which case $\theta_1$ is the isomorphism $\pi_0(\bbS)\xrightarrow{\cong}\bbZ$ given by degree. 
\end{proof}

\begin{Rem}
	Now that we have shown that $R(d)\cong \bbZ^d$, our next task is to find an explicit basis for $R(d)$. We have just defined an injective homomorphism $\prod_{i=1}^d\theta_d^i\colon R(d)\to \bbZ^d$. The following lemma will help us to identify a basis of $R(d)$.
\end{Rem}

\begin{Lem}\label{lem:basis_recognition}
	Suppose that $b_1, \ldots, b_d$ are elements of $R(d)$ that satisfy the following for all $i, j\le d$:
	\[
		\theta_d^i(b_j)=\left\{\begin{array}{cc} i! & i= j \\ 0 & i< j\end{array}\right. .
	\]
	Then $b_1, \ldots, b_d$ is a basis of $R(d)$.
\end{Lem}

\begin{proof}
	By \Cref{prop:injective} we have a split short exact sequence
	\[
		0\to \bbZ \to R(d)\xrightarrow{\pi_0(P_{d-1})} R(d-1) \to 0.
	\]
	It is enough to show that $b_d$ is in the image of a generator of $\bbZ$ and the images of $b_1, \ldots, b_{d-1}$ form a basis of $R(d-1)$. Let $\bar b_i$ be the image of $b_i$ in $R(d-1)$. By the commutativity of the bottom half of \eqref{diag:R(n)}, $\theta_{d-1}^i(\bar b_j)=\theta_d^i(b_j)$ for $i, j\le d-1$. Arguing by induction on $d$, we can conclude that $\bar b_1, \ldots, \bar b_{d-1}$ form a basis for $R(d-1)$. Once again, for a basis of the induction one can take the case $d=0$ or $d=1$, which are both immediate.

	Now consider $b_d$. By assumption, $\theta^i_d(b_d)=0$ for $i<d$. It follows that $b_d\in \ker(\pi_0(p_{d-1}))$ and therefore $b_d$ is in the image of the homomorphism $\bbZ\to R(d)$, that is, in the upper left corner of \eqref{diag:R(n)}. Furthermore, the assumption $\theta^d(b_d)=d!$ means that $b_d$ is the image of a generator of $\bbZ$. 
\end{proof}

\begin{Rem}
	Now we are ready to construct an explicit basis of $R(d)$. Recall that $h_{\bbS}(m)$ is the value of the $m$-th cross-effect (or equivalently the $m$-th co-cross-effect) of the contravariant functor $x\mapsto h_x$ at $(\bbS, \ldots,\bbS)$; see \cref{sec:cross-rep}.
\end{Rem}

\begin{Def}
    Define $\lambda_m\colon h_{\bbS}\to h_{\bbS}$ to be the following composition
	\[
		h_{\bbS}\to h_{\bbS^{\oplus m}}\to h_{\bbS}(m)\to h_{\bbS^{\oplus m}}\to h_{\bbS}.
	\]
	Here the first and the last map are induced by the fold and the diagonal maps between $\bbS$ and $\bigoplus_m\bbS$. The second and the third map are a special case of the natural map from $F(x_1\oplus\cdots\oplus x_m)$ to $\crosseffect_m F(x_1, \ldots, x_m)$ and back, with $F(x)=h_x$; see \eqref{eq:crosstococross}.
\end{Def}

\begin{Rem}\label{rem:lambda}
    By \Cref{lem:cross-effect}, $\lambda_m$ is equivalent to the composition
    \[
		h_{\bbS}\to h_{\bbS^{\oplus m}}\xrightarrow{\underset{t\in [m]}{\bigcirc}(1-h_{\psi([m]\setminus \{t\})})}  h_{\bbS^{\oplus m}} \to h_{\bbS}.
	\]
	Let us remind the reader that $h_{\psi([m]\setminus \{t\})}$ is the idempotent map $h_{\bbS^{\oplus m}}\to h_{\bbS^{\oplus m}}$ induced by collapsing to a point the $t$-th summand of $\underbrace{\bbS\oplus\cdots\oplus\bbS}_{m}$. 
\end{Rem}

\begin{Rem}
	The maps $\lambda_m$ are natural transformations from $h_\bbS$ to itself. To shorten notation, let us also denote by $\lambda_m$ the induced map $P_d(\lambda_m)\colon P_dh_{\bbS}\to P_dh_{\bbS}$ for all $d\ge m$. Let us denote by $[\lambda_m]$ the homotopy class of $\lambda_m$, considered as an element of $R(d)$.

	We want to show that $[\lambda_1], \ldots, [\lambda_d]$ form a basis of $R(d)$. It follows from \Cref{lem:basis_recognition} that to do this we need to analyze the induced maps $\partial_i\lambda_m \colon \partial_i h_{\bbS}\to \partial_ih_{\bbS}$. We already observed that $\partial_ih_{\bbS}\simeq \mathbb D(\bbS^{\otimes i})\simeq \bbS$ and therefore the homotopy type of~$\partial_i\lambda_m$ is determined by a single integer --- the degree.
\end{Rem}

\begin{Lem}\label{lem:derivatives}
	The degree of $\partial_i \lambda_m$ is $|\surj(i, m)|$. In particular, $\partial_i \lambda_m$ is null for $i<m$ and $\partial_i \lambda_i$ has degree $i!$.
\end{Lem}

\begin{proof}
	Using \Cref{rem:lambda} and the identification $\partial_i h_x\simeq \mathbb D (x^{\otimes i})$, we see that the map $\partial_i\lambda_m\colon h_{\bbS}\to h_{\bbS}$ is the Spanier--Whitehead dual of the following composition of maps
	\begin{equation}\label{eq:d_ilambda_m}
		\bbS^{\otimes i}\xrightarrow{\Delta_m^{\otimes i}} (\bbS^{\oplus m})^{\otimes i} \xrightarrow{\underset{t\in [m]}{\bigcirc}\left(1-\psi([m]\setminus\{t\})^{\otimes i}\right)} (\bbS^{\oplus m})^{\otimes i} \xrightarrow{\nabla_m^{\otimes i}} \bbS^{\otimes i}
	\end{equation}
	where $\Delta_m^{\otimes i}$ is the $i$-th smash power of the diagonal map $\bbS\to \bigoplus_m \bbS$, $\nabla_m^{\otimes i}$ is the $i$-th smash power of the fold map $\bigoplus_m \bbS \to \bbS$, and the map in the middle is the composition of maps of the form $1-\psi([m]\setminus\{t\})^{\otimes i}$, where $1$ is the identity on $(\bbS^{\oplus m})^{\otimes i}$ and $\psi([m]\setminus\{t\})^{\otimes i}$ is the $i$-th smash power of the self map of $\bbS^{\oplus m}$ that collapses the $t$-th copy of $\bbS$ to a point.

	There is an obvious identification $(\bbS^{\oplus m})^{\otimes i}\cong \bbS^{\oplus m^i}$, where we identify $m^i$ with the set of functions $\num{i}\to \num{m}$. Under this identification, the map $\Delta_m^{\otimes i}$ in \eqref{eq:d_ilambda_m} becomes the diagonal map $\bbS\to \bbS^{\oplus m^i}$ and the map $\nabla_i^{\otimes m}$ in the same line becomes the folding map $\bbS^{\oplus m^i}\to \bbS$. Given a $t\in [m]$, the map $\psi([m]\setminus\{t\})^{\otimes i}$ becomes the map $\bbS^{\oplus m^i}\to \bbS^{\oplus m^i}$ that collapses to a point all the copies of $\bbS$ that are labeled by maps $\num{i}\to\num{m}$ whose image contains $t$. It follows that $1-\psi([m]\setminus\{t\})^{\otimes i}$ is equivalent to the map $\bbS^{\oplus m^i}\to \bbS^{\oplus m^i}$ that collapses to a point all the copies of $\bbS$ that are labeled by maps $\num{i}\to\num{m}$ whose image does not contain $t$. It follows that $\underset{t\in [m]}{\bigcirc}\left(1-\psi([m]\setminus\{t\})^{\otimes i}\right)$ is equivalent to the map $\bbS^{\oplus m^i}\to \bbS^{\oplus m^i}$ that collapses to a point all the copies of $\bbS$ that are labeled by non-surjective maps $\num{i}\to\num{m}$.

	We conclude that the composition of maps on line~\eqref{eq:d_ilambda_m} is equivalent to the composition
	\[
		\bbS \to \bigoplus_{m^i}\bbS \to \bigoplus_{m^i} \bbS \to \bbS
	\]
	where the first map is the diagonal, the middle map collapses to a point all copies of~$\bbS$ labeled by non-surjective maps from $\num{i}$ to $\num{m}$, and the third map is the fold map. It is clear that this composition has degree $|\surj(i, m)|$ and therefore $\partial_i\lambda_m$, which is the Spanier--Whitehead dual of this composition, also has degree $|\surj(i, m)|$.
\end{proof}

The motivation for defining the elements $[\lambda_i]$ stems from the following lemma. 

\begin{Lem}\label{lem:basis}
	The elements $[\lambda_1]\ldots, [\lambda_d]$ form an additive basis of $R(d)$.    
\end{Lem}

\begin{proof}
	By \Cref{lem:basis_recognition} it is enough to prove that $\lambda_j$ induces multiplication by~$j!$ on~$\partial_j$ and induces the zero map on $\partial_i$ for $i<j$. This is a special case of \Cref{lem:derivatives}.
\end{proof}

Now we are ready to prove the main result of this section. Recall that $\mu(i, j, l)$ is the number of good subsets of $\num{i}\times\num{j}$ of cardinality $l$.

\begin{Thm}\label{thm:burnside-ring-isomorphism}
    The ring $R(d)$ is isomorphic to the quotient of the polynomial ring $\bbZ[[\lambda_2], \ldots, [\lambda_d]]$ by the relations
    \[
		[\lambda_i][\lambda_j]=\sum_{l=1}^d\mu(i,j,l)[\lambda_i][\lambda_j].
    \]    
    In particular, there is an isomorphism $R(d) \cong A(d)$ between $R(d)$ and the Goodwillie--Burnside ring given by sending $[\lambda_i]$ to $x_i$. 
\end{Thm}

\begin{proof}
	We will show that there is an isomorphism $R(d)\xrightarrow{\cong} A(d)$ that takes $[\lambda_i]$ to~$x_i$ for $i=1, \ldots, d$. The first part of the theorem then follows from \Cref{cor:polynomial}.

	By \Cref{lem:basis}, $R(d)$ is the free abelian group on the set $[\lambda_1], \ldots, [\lambda_d]$. Let $\bbZ^d$ denote the product ring, with generators $y_1, \ldots, y_d$. Goodwillie derivatives induce a ring homomorphism $R(d)\to \bbZ^d$. By \Cref{lem:derivatives}, this homomorphism sends $[\lambda_j]$ to $\sum_{i=1}^d |\surj(i, j)| y_j$. By \Cref{thm:fundamental-ses-a(n)}, there is a unique ring structure on $R(d)$ that makes this map into a ring homomorphism, and this ring structure is isomorphic to~$A(d)$.
\end{proof}

\begin{Rem}
    Because of this theorem, we will now cease to use $R(d)$ and instead only use the notation $A(d)$. 
\end{Rem}

\begin{Prop}\label{prop:comparison-map}
	For each $1 \le k \le d$, the comparison map of \Cref{def:comparison-map}
	\[
		\rho \colon \Spc(\Exc{d}(\Sp^c,\Sp)^c) \to \Spec(A(d))
	\]
	sends $\cat P(\num{k},0,1)$ to $\mathfrak p(\num{k},0)$ and $\cat P(\num{k},p,h)$ to $ \mathfrak p(\num{k},p)$ for all $h>1$. 
\end{Prop}

\begin{proof}
	It follows from \Cref{lem:derivatives} that the functor $\partial_k\colon \Exc{d}(\Sp^c, \Sp)\to \Sp$ induces the homomorphism $\phi_k\colon A(d)\to \bbZ$ that satisfies $\phi_k([\lambda_m])=|\surj(k, m)|$, which was used to define the prime ideal $\mathfrak p(\num{k},(p))$; see \Cref{def:prime_ideals}. 

	By naturality of the comparison map (\cite[Theorem 5.3(c)]{Balmer10b}) we deduce that the diagram 
	\[\begin{tikzcd}[column sep = 4em, ampersand replacement=\&]
		{\Spc(\Sp^c)} \& {\Spc(\Exc{d}(\Sp^c,\Sp)^c)} \\
		{\Spec(\bbZ)} \& {\Spec(A(d))}
		\arrow["\rho"', from=1-1, to=2-1]
		\arrow["{\Spc(\partial_k)}", from=1-1, to=1-2]
		\arrow["\rho", from=1-2, to=2-2]
		\arrow["{\Spec(\phi_k)}"', from=2-1, to=2-2]
	\end{tikzcd}\]
	commutes. The result then follows from \Cref{exa:comparison-map-for-spectra} and the definitions of the various prime ideals. 
\end{proof}

\newpage
\part{The spectrum of \texorpdfstring{$d$}{d}-excisive functors}\label{part:III}

We determined the underlying set of the Balmer spectrum $\Spc(\Exc{d}(\Sp^c,\Sp)^c)$ for any $d \ge 1$ in \cref{thm:spec-as-a-set}. Our next goal is to completely compute its topology. Recall that each prime tt-ideal $\cat P \in \Spc(\Exc{d}(\Sp^c,\Sp)^c)$ is of the form $\cat P=\cPd(\num{k},p,h)$ for some triple $(k,p,h)$ consisting of an integer $1 \le k \le d$, a prime number $p$ or $p=0$, and a chromatic height $1 \le h \le \infty$. As we will see in \cref{prop:spc_posetreduction}, the topology of the Balmer spectrum is completely determined by the inclusions among prime tt-ideals
    \[
        \cPd(\num{k},p,h) \overset{?}{\subseteq} \cPd(\num{l},q,h').
    \]
Consequently, our primary goal is to describe precisely when such an inclusion occurs in terms of a numerical formula involving the above variables.

We achieve this in three main steps, each of which requires a different set of techniques:
    \begin{enumerate}
        \item First, we study the chromatic blueshift behaviour of the Tate-derivatives $\partial_k t_d i_d\colon \Sp\to\Sp$ introduced in \cref{def:tate-derivatives}. Using results from \cite{AroneChing15}, this problem can be translated into an analogous blueshift question in stable equivariant homotopy theory for the family $\Fnt$ of non-transitive subgroups of products of symmetric groups $\Sigma_{n}$. While the spectrum $\Spc(\Sp_{\Sigma_{n}}^c)$ is only known for $n \leq 7$ by work of Balmer--Sanders \cite{BalmerSanders17} ($n \leq 3$) and Kuhn--Lloyd \cite{KuhnLloyd2024} ($n\leq 7$), we are able to give a complete answer (\cref{thm:tateblueshift}) by reducing to the main result of \cite{BHNNNS19}.
        \item We then explain how the blueshift behaviour of the Tate-derivatives provides ``elementary'' inclusions among the prime tt-ideals of $\Exc{d}(\Sp^c,\Sp)^c$; see \cref{lem:basic-tate-inclusion}. Combining general tt-geometric techniques developed in~\cite{BalmerSanders17} together with our computation of the spectrum of the Goodwillie--Burnside ring (\cref{thm:burnside-goodwillie-spectrum}), we then show in \cref{thm:formula-for-geombluetake2} that all inclusions among prime tt-ideals are given by certain minimal chains of elementary inclusions.
        \item It then remains to understand the nature of these minimal chains, a combinatorial problem that is---in light of \cref{rem:spectral-mackey}---ultimately controlled by the structure of the poset $\Epi{d}$. This in turn translates into an elementary number-theoretic problem about the existence of $p$-power partitions, whose solution is given in \cref{prop:ppowerchains}. The final answer which precisely describes the inclusions among the prime tt-ideals of $\Exc{d}(\Sp^c,\Sp)^c$ is then assembled in \cref{thm:posetstructure}.
    \end{enumerate}

In \cref{sec:applications}, we express  the topology of $\Spc(\Exc{d}(\Sp^c,\Sp)^c)$ more explicitly in terms of type functions and state the resulting classification of tt-ideals (\cref{thm:tt-ideal-classification}). We then conclude with two further results: a calculus analogue (\cref{cor:floydinequality}) of transchromatic Smith--Floyd theory inspired by \cite{KuhnLloyd2024} and a computation (\cref{thm:spc_integral}) of the Balmer spectrum of the category $\Exc{d}(\Sp^c,\Mod_{\HZ})$ obtained by changing coefficients along $\Sp\to\Mod_{\HZ}$ which is inspired by \cite{PatchkoriaSandersWimmer22}.

\section{Blueshift of Tate-derivatives}\label{sec:tateblueshift}

In this section we study the chromatic behaviour of the Tate-derivatives on the category of $d$-excisive functors, which were introduced in \Cref{def:tate-derivatives}:
    \begin{equation}\label{eq:dti}
        \partial_lt_di_d\colon \Sp \to \Sp, \quad A \mapsto \partial_l(X \mapsto (A \otimes X^{\otimes d})^{t\Sigma_d}).
    \end{equation}
Our goal is to determine the precise blueshift at each prime $p$, i.e., the extent to which the functor \eqref{eq:dti} shifts chromatic height. We can summarize our approach in two steps:
    \begin{enumerate}
        \item We use results in \cite{AroneChing15} to express the Tate-derivatives in terms of geometric fixed point functors for families of non-transitive subgroups of products of symmetric groups. This translates our question to a problem in stable equivariant homotopy theory.
        \item We reduce the problem further to the case of cyclic groups, for which we can apply the main theorems of \cite{Kuhn04} and \cite{BHNNNS19} to resolve the analogous blueshift question.
    \end{enumerate}
The answer is given in \cref{thm:tateblueshift} below.

\subsection*{Preliminaries on stable equivariant homotopy theory}\label{ssec:seht}

\begin{Not}
	For a finite group $G$, we denote the symmetric monoidal stable $\infty$-category of \emph{genuine $G$-spectra} by $\Sp_G$. It is rigidly-compactly generated by the orbits $\Sigma^\infty G/H_+$ associated to the (conjugacy classes of) subgroups $H \le G$. For a construction of $\Sp_G$ and further discussion of its basic properties, see \cite{MathewNaumannNoel17} and \cite{BalmerSanders17}. Each group homomorphism $f\colon H\to G$ induces a geometric functor $f^*\colon \Sp_G\to\Sp_H$ which has a lax symmetric monoidal right adjoint $f_*$. In particular, for the unique homomorphism $p\colon G \to e$ to the trivial group, we have the \emph{inflation} functor $\infl_e^G\coloneqq p^*\colon \Sp\to\Sp_G$ whose right adjoint $(-)^G\coloneqq p_*$ is the \emph{(categorical) fixed points} functor.
\end{Not}

\begin{Def}\label{def:families}
    Let $G$ be a finite group. A \emph{family of subgroups} $\cF$ is a collection of subgroups of $G$ which is closed under conjugation and passage to subgroups.
\end{Def}

\begin{Exa}\label{ex:nontransitivefamily}
    Let $G = \Sigma_d$ be the symmetric group on $d$ letters. A subgroup $H \subseteq \Sigma_d$ is called \emph{non-transitive} if the induced permutation $H$-action on a set of $d$-elements is non-transitive. The collection of non-transitive subgroups of $\Sigma_d$ forms a family $\Fnt(d)$ which will play a distinguished role in what follows.
\end{Exa}

\begin{Rem}
	Let $d=\sum_{i=0}^m a_i p^i$ be the expansion of $d$ in base $p$, so that $0 \le a_i < p$ for all $i$. Let $C_p^{\wr i} \coloneqq C_p \wr \ldots \wr C_p$ denote the $i$-fold wreath product of cyclic groups of order $p$. The $p$-Sylow subgroups of $\Sigma_d$ are of the form
    \[
        S_{d,p} \cong \prod_{i=0}^m (C_p^{\wr i})^{\times a_i}.
    \]
    (See \cite[p.~176]{Rotman95} or \cite{Kaloujnine48}.) It follows that $\Sigma_d$ has a transitive $p$-Sylow subgroup if and only if $d$ is a power of $p$. This is the underlying reason for the sparsity of blueshift numbers we will encounter in the description of the topology later on.
\end{Rem}

\begin{Rem}\label{rem:closure-prop-for-families}
	We can pull back a family $\cF$ of subgroups of $G$ along any homomorphism $f\colon H \to G$ to obtain a family
    \[
        f^{-1}\cF \coloneqq \SET{K \subseteq H}{f(K) \in \cF}
    \]
	of subgroups of $H$. One also readily observes that a union or intersection of a collection of families of subgroups of $G$ is again a family of subgroups of $G$.
\end{Rem}

\begin{Def}\label{def:universalspaces}
	We let $E\cF$ denote the universal $G$-space for the family $\cF$, which is characterized (up to homotopy equivalence) by 
    \[
        E\cF^H \simeq 
            \begin{cases}
                \ast & H \in \cF \\
                \varnothing & H \notin \cF.
            \end{cases}
    \]
	We define the pointed $G$-space $\tilde{E}\cF$ via the cofiber sequence of pointed spaces
    \begin{equation}\label{eq:EF-cofiber}
        E\cF_+ \to S_G^0 \to \tilde{E}\cF.
    \end{equation}
	It is characterized by
    \[
		\tilde{E}\cF^H   \simeq   
            \begin{cases}
                S^0 & H \notin \cF \\
                \ast & H \in \cF.
            \end{cases}
    \]
\end{Def}

\begin{Exa}\label{ex:universalspaces}
     If $\Ftriv$ is the family consisting of only the trivial subgroup of $G$, then we obtain the familiar universal $G$-space $E\Ftriv \simeq EG$. If $\Fall$ is the family of all subgroups of $G$, then $E\Fall \simeq S_G^0$ and $\tilde{E}\Fall \simeq \ast$.
\end{Exa}

\begin{Exa}\label{exa:universal-space-for-sigma-d}
   Let $\overline \rho_d$ denote the reduced standard $(d-1)$-dimensional real representation of $\Sigma_d$ and write $S^{\overline \rho_d}$ for the corresponding $\Sigma_d$-representation sphere, i.e., the one-point compactification of $\overline \rho_d$ with induced $\Sigma_d$-action. Let $S^{n\overline{\rho}_d}$ be its $n$-fold smash power. The spaces $S^{n\overline{\rho}_d}$ form a directed system as $n$ varies, and 
    \[
		S^{\infty \overline \rho_d} \coloneqq\colim_n S^{n\overline{\rho}_d}
    \]
    is a model for the $\Sigma_d$-space $\tilde E \Fnt(d)$ for the family $\Fnt(d)$ of \cref{ex:nontransitivefamily}.
\end{Exa}

\begin{Lem}\label{lem:universalspace_properties}
	Let $G$ and $H$ be finite groups.
    \begin{enumerate}
        \item Let $f\colon H \to G$ be a group homomorphism and $\cF$ a family of subgroups of $G$. Then there are equivalences of $H$-spaces
            \[
                f^*(E\cF_+) \simeq (Ef^{-1}\cF)_+ \quad \text{and} \quad f^*(\tilde{E}\cF) \simeq (\tilde{E}f^{-1}\cF).
            \]
        \item Consider two families $\cF_1,\cF_2$ of subgroups of $G$. Then there are equivalences of $G$-spaces
            \[
                (E\cF_1)_+ \wedge (E\cF_2)_+ \simeq E(\cF_1 \cap \cF_2)_+ \quad \text{and} \quad (\tilde{E}\cF_1) \wedge (\tilde{E}\cF_2) \simeq \tilde{E}(\cF_1 \cup \cF_2).
            \]
    \end{enumerate}
\end{Lem}

\begin{proof}
    Both pairs of identities can be readily checked by testing on fixed points. %
\end{proof}

\begin{Rem}\label{rem:EF-idempotent-triangle}
	Applying $\Sigma^\infty$ to the cofiber sequence \eqref{eq:EF-cofiber} we obtain an exact triangle
	\[
		\Sigma^\infty E\cF_+ \to \Sigma^\infty S^0_G \to \Sigma^\infty \tilde{E}\cF 
	\]
	in $\Sp_G$ which can be identified with the idempotent triangle (\cref{rem:balmer-favi-properties})
	\[
		e_Y \to \unit \to f_Y 
	\]
	associated to the Thomason subset $Y\coloneqq\bigcup_{H \in \cF} \supp(\Sigma^\infty G/H_+) \subseteq \Spc(\Sp_G^c)$; cf.~\cite[Example 5.14]{BalmerSanders17}. By slight abuse of notation, we will drop the $\Sigma^\infty$ and write $E\cF_+$ and $\tilde{E}\cF$ for the corresponding suspension spectra in $\Sp_G$. Note that the stable versions of the identities of \cref{lem:universalspace_properties} hold as well, since $\Sigma^\infty$ commutes with the smash product $\wedge$ and change of groups $f^*$. Alternatively, they follow from the general behaviour of Balmer--Favi idempotents using the above identification; cf.~\cref{prop:idempotents-under-tt-functors} and \cite[Lemma 1.27]{bhs1}.
\end{Rem}

\begin{Rem}
	We now define geometric fixed points with respect to a family of subgroups. When applied to the family $\cF=\cP(G)$ of all proper subgroups of $G$, the definition recovers the usual geometric fixed points for $G$.
\end{Rem}

\begin{Def}\label{def:geometricfixedpoints}
    The \emph{geometric fixed points} of a $G$-spectrum $X$ with respect to a family $\cF$ of subgroups of $G$ is the non-equivariant spectrum
        \[
            \Phi^{\cF}(X) \coloneqq (X \otimes \tilde{E}\cF)^{G}.
        \]
\end{Def}

\begin{Lem}\label{lem:geomfixedpoints_changeoffamily}
    Suppose $\cF_1 \subseteq \cF_2$ are families of subgroups of $G$. There is a monoidal natural transformation $\Phi^{\cF_1} \to \Phi^{\cF_2}$ of lax symmetric monoidal functors.
\end{Lem}

\begin{proof}
    The desired natural transformation is obtained from the map $\tilde{E}\cF_1 \to \tilde{E}\cF_2$ of equivariant commutative ring spectra by applying the lax symmetric monoidal functor of fixed points. 
\end{proof}

\begin{Def}\label{def:tateconstruction}
    Let $\cF$ be a family of subgroups of $G$. The \emph{$\cF$-Tate construction} on a $G$-spectrum $X$ is defined as the $\cF$-geometric fixed points
    \[
		X^{t\cF} \coloneqq \Phi^{\cF}(\underline{X})
	\]
    of the \emph{Borel-completion} $\underline{X}\coloneqq F(EG_+,X)$ of $X$. For a non-equivariant spectrum~$Y$, we write
    \[ Y^{t\cF} \coloneqq (\infl_e^G Y)^{{t\cF}}\]
    for the $\cF$-Tate construction on $Y$ equipped with a trivial $G$-action via inflation.
\end{Def}

\begin{Exa}\label{exa:usual-tate}
	Taking the trivial family $\cF = \Ftriv$, we obtain the usual Tate construction $X^{tG} \coloneqq (\underline{X} \otimes \tilde{E}G)^G$.
\end{Exa}

\begin{Lem}\label{lem:geomfixedpoints_functoriality}
	Let $f \colon H \to G$ be a group homomorphism and $\cF$ a family of subgroups of $G$. Then:
	\begin{enumerate}
	\item There is a monoidal natural transformation
			\[
				\Phi^{\cF} \to \Phi^{f^{-1}\cF}\circ f^*
			\]
		of symmetric monoidal functors $\Sp_G \to \Sp$. 
	\item There is a monoidal natural transformation
		\[
		(-)^{t\cF} \to f^*(-)^{t(f^{-1}\cF)}
		\]
		of lax symmetric monoidal functors $\Sp_G \to \Sp$.
	\item In particular, precomposing with $\infl_e^G$, there is a monoidal natural transformation
		\[
			(-)^{t\cF}  \to (-)^{t(f^{-1}\cF)}
		\]
		of lax symmetric monoidal functors $\Sp \to \Sp$.
	\end{enumerate}
\end{Lem}

\begin{proof}
    The natural transformation in part (a) is given by the composite
	\[
		\Phi^{\cF}(X) = (\tilde{E}\cF \otimes X)^G \to (f_*f^*(\tilde{E}\cF \otimes X))^G \simeq (f^*(\tilde{E}\cF) \otimes f^*(X))^H \simeq \Phi^{f^{-1}\cF}(f^*(X))
	\]
    where the middle map is the unit of the $(f^*,f_*)$-adjunction, the first equivalence follows from the relation $(f_*X)^G \simeq X^H$ for any $X \in \Sp_H$, and the last from \cref{lem:universalspace_properties}(a). For part (b), let $N\coloneqq \ker f$ denote the kernel of $f$ and observe that $f^*(EG_+) \simeq E\cF[{\le}N]_+$ is the universal $H$-space for the family of all subgroups contained in $N$, again by \cref{lem:universalspace_properties}(a). Since $\Ftriv \subseteq \cF[{\le}N]$, \cref{lem:universalspace_properties}(b) implies there is a canonical morphism $\theta\colon EH_+ \to E\cF[{\le}N]_+$ given by
	\[
		EH_+ \simeq EH_+ \otimes E\cF[{\le}N]_+ \to \unit \otimes E\cF[{\le}N]_+ \simeq E\cF[{\le}N]_+
	\]
	which induces a map
	\[
		\ihom{E\cF[{\le}N]_+,Y} \xrightarrow{\ihom{\theta,1}} \ihom{EH_+,Y}
	\]
    for any $H$-spectrum $Y$. The natural transformation in (b) is then induced from the one in part (a) using the natural transformation
    \begin{equation}\label{eq:closed-aux}
		f^*\ihom{EG_+,X} \to \ihom{f^*(EG_+),f^*(X)}\xrightarrow{\ihom{\theta,1}} \ihom{EH_+,f^*(X)}.
    \end{equation}
    We note in passing that when $f\colon H\hookrightarrow G$ is the inclusion of a subgroup, the first map in \eqref{eq:closed-aux} is an equivalence by \cite[(3.12)]{BalmerDellAmbrogioSanders16} and the second map is also an equivalence since in this case $f^*(EG_+)\simeq EH_+$ (i.e.,~$\cF[{\le}N]=\Ftriv$). Part (c) follows from part (b) since $f^* \circ \infl_e^G \simeq \infl_e^H$.
\end{proof}

\begin{Rem}\label{rem:transferargument}
    Let $\cF$ be a family of subgroups of $G$ and consider some $H \in \cF$ of index $[G:H]$ in $G$. If $X$ is a $G$-spectrum on which $[G:H]$ is invertible, then a standard transfer argument shows that $X^{t\cF} = 0$. Indeed, we may apply \cite[Corollary A.9]{GlasmanLawson20pp} with $\cF' = \Fall$ to see that $X^{t\cF} \simeq X^{t\Fall} = 0$. For example, if $X$ is a $p$-local $G$-spectrum (such as $X=\infl_e^G Y$ for a $p$-local spectrum $Y$) and $p$ does not divide $[G:H]$ then $X^{t\cF}=0$.
\end{Rem}

\subsection*{Derivatives of the Tate construction}\label{ssec:tatederivatives}

Let $d>l>0$ be positive integers. Our next goal is to express the Tate-derivative $\partial_l t_d i_d\colon \Sp\to\Sp$ of \cref{def:tate-derivatives} in terms of the equivariant Tate constructions of \cref{def:tateconstruction}. This will be based on a description of the Tate-derivatives due to \cite{AroneChing15}. We first introduce some notation.

\begin{Def}\label{def:generalsymmetricgroup}
    For a partition $\lambda = (d_1,\ldots,d_l) \vdash d$ of $d$ of length $l$, we let 
    \[
		\Sigma_{\lambda} \coloneqq \Sigma_{d_1} \times \ldots \times \Sigma_{d_l}
    \]
    denote the corresponding product of symmetric groups. The projection onto the $i$-th factor will be denoted by $\pi_i \colon \Sigma_{\lambda} \twoheadrightarrow \Sigma_{d_i}$. Furthermore, we write 
	\begin{equation}\label{eq:nontransitivefamily}
		\Fnt(\lambda) \coloneqq \pi_1^{-1}\Fnt(d_1) \cup \ldots \cup \pi_l^{-1}\Fnt(d_l)
	\end{equation}
    for the family of subgroups $K$ of $\Sigma_{\lambda}$ with the property that at least one of the projections $\pi_iK \subseteq \Sigma_{d_i}$ is non-transitive. Note that $\Fnt(\lambda)$ does indeed form a family by \Cref{rem:closure-prop-for-families}.
\end{Def}

\begin{Exa}\label{ex:nontransitive}
    Consider $\lambda = (2,2) \vdash 4$. Then $\Fnt(\lambda)$ consists of the three subgroups $e, \Sigma_2 \times e, e \times \Sigma_2$ in $\Sigma_2 \times \Sigma_2$. In particular, the diagonal $\Sigma_2$ is not part of the family. 
\end{Exa}

We now recall a construction from \cite{AroneChing15} that appears in their description of the Tate-derivatives.

\begin{Def}\label{def:aronechingspaces}
    Let $\lambda = (d_1,\ldots, d_l) \vdash d$ be a partition of $d$ of length $l$. Continuing \cref{exa:universal-space-for-sigma-d}, we define the $\Sigma_{\lambda}$-space
        \[  
            S^{\infty(d-l)}_{(\lambda)} \coloneqq \colim_{L \to \infty} S^{L\overline \rho_{d_1}} \wedge \ldots \wedge S^{L\overline \rho_{d_l}}
        \]
    where $\Sigma_{\lambda}$ acts on each smash factor $S^{L\overline \rho_{d_i}}$ through its projection to $\Sigma_{d_i}$. We also write $S^{\infty(d-l)}_{(\lambda)}$ for the associated $\Sigma_{\lambda}$-suspension spectrum.
\end{Def}

\begin{Lem}\label{lem:universalspace_identification} 
    Let $\lambda = (d_1,\ldots,d_l) \vdash d$ be a partition of $d$ of length $l$. Then we have an equivalence of $\Sigma_{\lambda}$-spectra 
        \[ 
            \textstyle S^{\infty(d-l)}_{(\lambda)} \simeq \bigotimes_{i=1}^l S^{\infty \overline \rho_{d_i}} \simeq \tilde{E}\Fnt(\lambda) 
        \] 
    where $\Fnt(\lambda)$ is the family \eqref{eq:nontransitivefamily} of $\Sigma_{\lambda}$. 
\end{Lem} 

\begin{proof}
	The first equivalence holds because the equivariant suspension spectrum functor is symmetric monoidal. The second equivalence then follows from \cref{exa:universal-space-for-sigma-d} and \cref{lem:universalspace_properties}. 
\end{proof}

\begin{Prop}\label{prop:tatederivatives_equivariantformula}
    Let $A$ be a spectrum. Then there is an equivalence
        \[
            \partial_l(X \mapsto (A \otimes X^{\otimes d})^{h\Sigma_d}) \simeq \prod_{\lambda \vdash d\colon |\lambda|=l} A^{t\Fnt(\lambda)}
        \]
    where the product ranges over partitions of $d$ of length $l$ and $\Fnt(\lambda)$ is the family defined in \eqref{eq:nontransitivefamily}.
\end{Prop}

\begin{proof}
    This result is essentially a reformulation of \cite[Proposition 5.2 and Remark 5.3]{AroneChing15}. There the authors write $K_lA$ for the $l$-th derivative of the functor $X \mapsto (A\otimes X^{\otimes d})^{h\Sigma_d}$ and describe it by the formula
        \[
            \partial_l(X \mapsto (A \otimes X^{\otimes d})^{h\Sigma_d})  \simeq \prod_{\lambda \vdash d\colon |\lambda|=l} \colim_{L \to \infty}(A \otimes S^{L(d-l)})^{h\Sigma_{\lambda}}.
        \]
    Here, $S^{L(d-l)}$ is the spectrum with $\Sigma_d$-action constructed in \cite[Definition 5.1]{AroneChing15}. Restricting the action along $\Sigma_{\lambda} \subseteq \Sigma_d$ for some partition $\lambda = (d_1,\ldots,d_l)\vdash d$, we may identify $S^{L(d-l)}$ with $S^{L\overline \rho_{d_1}} \otimes \ldots \otimes S^{L\overline \rho_{d_l}}$ as spectra with $\Sigma_{\lambda}$-action. To obtain the proposition, we can thus rewrite the factors in the product as follows:
        \begin{align*}
            \colim_{L \to \infty}(A \otimes S^{L(d-l)})^{h\Sigma_{\lambda}} & \simeq \colim_{L \to \infty}\ihom{E\Sigma_{\lambda+}, (\infl_e^{\Sigma_{\lambda}} A) \otimes \textstyle\bigotimes_{i=1}^lS^{L\overline \rho_{d_i} }}^{\Sigma_{\lambda}} \\
            & \simeq \colim_{L \to \infty}(\ihom{E\Sigma_{d+},\infl_e^{\Sigma_{\lambda}} A} \otimes \textstyle\bigotimes_{i=1}^lS^{L\overline \rho_{d_i} })^{\Sigma_{\lambda}} \\
            & \simeq (\ihom{E\Sigma_{d+},\infl_e^{\Sigma_{\lambda}} A} \otimes S^{\infty(d-l)}_{(\lambda)})^{\Sigma_{\lambda}} \\
            & \simeq (\underline{A} \otimes \tilde{E}\Fnt(\lambda))^{\Sigma_{\lambda}} \\
            & \simeq \Phi^{\Fnt(\lambda)}(\underline{A}) = A^{t\Fnt(\lambda)}
        \end{align*}
    where the penultimate equivalence substitutes the identification of \cref{lem:universalspace_identification}.
\end{proof}

\subsection*{Blueshift}\label{ssec:blueshift}

If $M$ is a spectrum, we can restrict the homology theory represented by~$M$ to $p$-local finite spectra. Its kernel will then be a thick subcategory of $\Sp_{(p)}^c$, so we can use the thick subcategory theorem to define a notion of height for $M$:

\begin{Def}\label{def:height}
    If $M$ is a spectrum, then its \emph{height} at the prime $p$ is defined as
        \[
            \height_p(M) \coloneqq \inf{\SET{h \in \bbN \cup\{-1,\infty\}}{M \otimes \cat C_{p,h+1} = 0}}
        \]
    where $\cat C_{p,h+1}$ is the prime given by the vanishing of $K(p,h)$, as in \cref{exa:balmer-spectrum-spectra}. 
\end{Def}

\begin{Exa}\label{ex:edgeheight}
    By definition, we have $\height_p(0) = -1$ and $\height_p(\bbS) = \infty$ for all primes $p$. In fact, any spectrum that is not $p$-local has infinite $p$-height using the convention that $\inf(\varnothing) = \infty$.
\end{Exa}

\begin{Rem}
    Let $\smash[b]{L_{p,h}^f}$ be the finite localization on $\Sp$ which localizes away from~$\cat C_{p,h+1}$ (\cref{exa:chromatic-truncation}). If we further set $L_{p,-1}^f X \coloneqq 0$ and $L_{p,\infty}^f X \coloneqq X_{(p)}$ for all $X \in \Sp$, we can rewrite \cref{def:height} as
        \[
            \height_p(M) \coloneqq \inf{\SET{h \in \bbN \cup\{-1,\infty\}}{M \simeq L_{p,h}^f(M)}}.
        \]
    This implies that the notion of height defined here is compatible with, and in fact extends, the one appearing in \cite{BHNNNS19} using the fact that the telescope conjecture holds for ring spectra (\cite[Lemma 2.3]{land2022purity}). In particular, we see that $\height_p(L_{p,h}^f\bbS) = h$ and $\height_p(K(p,h)) = h$ for all $h$.
\end{Rem}

\begin{Lem}\label{lem:height_properties}
	Let $p$ be a prime. Height has the following properties:
	\begin{enumerate} 
		\item If $R \to S$ is a map of ring spectra, then $\height_p(R) \ge \height_p(S)$.
		\item If $M, N$ are spectra, then $\height_p(M \oplus N) = \sup(\height_p(M),\height_p(N))$.
	\end{enumerate}
\end{Lem}

\begin{proof}
    Both statements follow from unwinding the definitions. For instance, for the first one, observe that $S$ is an $R$-module, so $R\otimes X = 0$ implies $S \otimes X = 0$ for any spectrum $X$.  
\end{proof}

\begin{Def}\label{def:heightidempotent}
	Let $d \ge 1$ be an integer and $p$ a prime. For any $h \ge 0$, we define a $d$-excisive functor via inflation (\cref{def:inflation}) from the finite local sphere of height~$h$:
	\[
		\cL_{p,h}^f \coloneqq i_dL_{p,h}^f\bbS.
	\]
\end{Def}

\begin{Def}\label{def:ppowerpartition}
	A \emph{$p$-power partition} of $d$ of length $l$ is a partition 
 	\[
		\lambda = (d_1,\ldots,d_l) \vdash d
	\]
	of $d$ of length $l$ such that each $d_i$ is a power of $p$. In this definition, we include $p^0=1$ as a power of $p$. We will write $\cP_p(d;l)$ for the set of such partitions.
\end{Def}

We are ready to state and prove the main theorem of this section, which captures the chromatic behaviour of the Tate-derivatives. 

\begin{Thm}\label{thm:tateblueshift}
    Let $p$ be a prime and fix integers $d \geq l > 0$ and $h \geq 0$. Then $\partial_lt_d\cL_{p,h}^f$ is contractible unless $d>l$ and there exists a $p$-power partition of $d$ of length $l$. If such a partition exists, then 
	\[
		\height_p(\partial_lt_d\cL_{p,h}^f) = h-1.
	\]
\end{Thm}

\begin{proof}
    Let $\lambda = (d_1,\ldots,d_l)$ be a partition of $d$ of length $l$. The Tate construction~$t_d(F)$ of any $d$-excisive functor $F$ is $(d-1)$-excisive, so the theorem holds for~$d=l$. Therefore, we may assume $l<d$ for the remainder of the proof. The vertical maps in the Tate square of \cref{prop:tatesquare} are $\partial_i$-equivalences for all $i<d$, so we may replace the Tate fixed points by homotopy fixed points in our formula. In light of \cref{prop:tatederivatives_equivariantformula} and \cref{lem:height_properties}(b), it then suffices to verify the claim for
    \[
		(L_{p,h}^f\bbS)^{t\Fnt(\lambda)} \coloneqq \Phi^{\Fnt(\lambda)}(\underline{\infl_e^{\Sigma_\lambda}L_{p,h}^f\bbS})
	\]
    in place of the Tate-derivative $\partial_lt_d\cL_{p,h}^f$. 
    
    We begin with the vanishing claim. To this end, suppose $\lambda$ is not a $p$-power partition, so that there exists some $d_i$ which is not a power of $p$, say $d_1$. Let $S$ be a $p$-Sylow subgroup of $\Sigma_{d_1}$, which is then a non-transitive subgroup of $\Sigma_{d_1}$ by \cref{ex:nontransitivefamily}. It follows that $S'\coloneqq S \times \prod_{i=2}^l\Sigma_{d_i} \subseteq \Sigma_{\lambda}$ is contained in $\Fnt(\lambda)$. Since $L_{p,h}^f \bbS$ is $p$-local and $[\Sigma_{\lambda}:S']$ is prime to $p$, it follows from \cref{rem:transferargument} that $(L_{p,h}^f\bbS)^{t\Fnt(\lambda)} = 0 $. This proves the vanishing part of the theorem.
    
    For the rest of this proof, we will assume that $\lambda$ is a $p$-power partition of $d$ of length $l$, say $d_i = p^{e_i}$ for some~$e_i \geq 0$. Note that, since $l<d$ by assumption, at least one of the~$e_i$ must be positive. Write $\Ftriv$ for the family on $\Sigma_{\lambda}$ consisting only of the trivial subgroup and recall that $X^{t\Ftriv} = X^{t\Sigma_{\lambda}}$ is the usual Tate construction (\cref{exa:usual-tate}). \Cref{lem:geomfixedpoints_changeoffamily} therefore supplies a map of commutative ring spectra
	\[
		(L_{p,h}^f\bbS)^{t\Sigma_{\lambda}} = (L_{p,h}^f\bbS)^{t\Ftriv} \to (L_{p,h}^f\bbS)^{t\Fnt(\lambda)}.
	\]
    Kuhn's blueshift theorem \cite[Theorem 1.5]{Kuhn04} shows that the height of the domain is at most $h-1$, so via \cref{lem:height_properties}(a) we get the upper bound
	\begin{equation}\label{eq:tateheight_upperbound}
		\height_p((L_{p,h}^f\bbS)^{t\Fnt(\lambda)}) \leq h-1.
	\end{equation}
    As for the lower bound, let $\sigma_i$ be the long cycle in $\Sigma_{p^{e_i}}$ and consider the cyclic $p$-subgroup $\langle \sigma \rangle$ of $\Sigma_{\lambda}$ generated by $\sigma = (\sigma_1,\ldots,\sigma_l) \in \Sigma_{\lambda}$. In particular, the order of $\langle \sigma \rangle$ is the maximum of the $p^{e_i}$. By construction, $\pi_i\langle \sigma \rangle = \langle \sigma_i \rangle$ is a transitive subgroup in $\Sigma_{p^{e_i}}$ and, as at least one of the $e_i$'s is positive, $\langle \sigma \rangle$ is not the trivial group. Therefore, $\langle \sigma \rangle \notin \Fnt(\lambda)$.

    With this preparation in hand, we can apply \cref{lem:geomfixedpoints_functoriality} once more to obtain a map of commutative ring spectra
	\[
		(L_{p,h}^f\bbS)^{t\Fnt(\lambda)} \to (L_{p,h}^f\bbS)^{t(\Fnt(\lambda)\cap \langle \sigma \rangle)}.
	\]
    The projection of every proper subgroup of $\langle \sigma \rangle$ has order less than $p^{e_i}$ and thus cannot act transitively on $\Sigma_{p^{e_i}}$, so it is not a member of  $\Fnt(\lambda)$. We conclude that $\Fnt(\lambda) \cap \langle \sigma \rangle$ is precisely the family of proper subgroups of $\langle \sigma \rangle$. Hence
    \[
		(L_{p,h}^f\bbS)^{t(\Fnt(\lambda)\cap \langle \sigma \rangle)} = \Phi^{\langle \sigma \rangle}(\underline{\infl_e^{\langle \sigma \rangle}L_{p,h}^f\bbS}).
	\]
    By \cite[Theorem 1.5]{BHNNNS19}, the height of the latter ring spectrum is $h-1$, so 
	\begin{equation}\label{eq:tateheight_lowerbound}
		\height_p((L_{p,h}^f\bbS)^{t\Fnt(\lambda)}) \geq h-1
	\end{equation}
    using \cref{lem:height_properties}(a) again. Combining~\eqref{eq:tateheight_upperbound} and~\eqref{eq:tateheight_lowerbound}, we obtain the desired formula for the height of $(L_{p,h}^f\bbS)^{t\Fnt(\lambda)}$ and hence for the height of $\partial_lt_d\cL_{p,h}^f$. 
\end{proof}

\begin{Rem}\label{rem:heightinsensitive}
	The theorem shows that the height-shifting behaviour of $\partial_lt_di_d$ is independent of the input height $h$.
\end{Rem}

\section{The topology of the Balmer spectrum}\label{sec:topology}

We now determine the topology of the spectrum of $\Exc{d}(\Sp^c,\Sp)^c$. The architecture of our proof is modelled on the strategy given in the context of equivariant stable homotopy theory in \cite{BalmerSanders17}. First we establish that the topology is entirely determined by the inclusions among prime tt-ideals and then proceed to understanding this poset structure. The key idea is that these inclusions are controlled by the blueshift behaviour of the Tate-derivatives computed in the previous section, together with the combinatorics of chains of $p$-power partitions. 

\begin{Prop}\label{prop:spc_posetreduction}
    Every closed subset of $\Spc(\Exc{d}(\Sp^c,\Sp)^c)$ is a finite union of irreducible closed subsets. In particular, the topology of $\Spc(\Exc{d}(\Sp^c,\Sp)^c)$ is determined by the inclusions among prime tt-ideals. 
\end{Prop}

\begin{proof}
	The proof that every closed subset is a finite union of irreducible closed subsets is similar to the proof of \cite[Proposition~6.1]{BalmerSanders17}. The key ingredient (\cref{lem:cover}) is that $\Spc(\Exc{d}(\Sp^c,\Sp)^c)$ is covered by finitely many spectra for which the statement is correct; see \cite[Corollary~9.5(e)]{Balmer10a}. The second statement follows from the first because the irreducible closed subsets correspond to the tt-primes and are given by $\overline{\{\cat P\}} = \SET{\cat Q}{\cat Q \subseteq \cat P}$.
\end{proof}

\begin{Rem}
	In the proofs of the following results, we will repeatedly use the fact that the map $\Spc(F)$ on Balmer spectra induced by a tt-functor $F$ preserves inclusions among prime tt-ideals; see \cref{Rem:basic-balmer-properties}.
\end{Rem}

\begin{Prop}\label{prop:verticalinclusions}
	For every pair of integers $1 \le l \le d$, primes $p,q$, and chromatic integers $1 \le h,h' \le \infty$, the following are equivalent:
	\begin{enumerate}
		\item $\cPd(\num{l},q,h') \subseteq \cPd(\num{l},p,h)$;
		\item $\cat C_{q,h'} \subseteq \cat C_{p,h};$
		\item $h' \ge h$ and, either $h = 1$ or $q = p$. 
	\end{enumerate}
\end{Prop}

\begin{proof}
	$(a) \iff (b)$ follows from \cref{lem:injectivity-of-partial-k} and \cref{lem:Spc(i_d)} while $(b) \iff (c)$ follows from the computation of the spectrum of $\Sp^c$.
\end{proof}

\begin{Cor}\label{cor:inclusion_chromatic}
	Suppose $\cPd(\num{k},q,h') \subseteq \cPd(\num{l},p,h)$ for integers $1\le l,k \le d$, primes $p,q$ and chromatic integers $1 \le h,h' \le \infty$. Then $\cat C_{q,h'} \subseteq \cat C_{p,h}$ in $\Sp^c$ and so $h' \ge h$. If $h > 1$, then $p = q$. 
\end{Cor}

\begin{proof}
	This is an immediate consequence of \cref{lem:Spc(i_d)} and \cref{prop:verticalinclusions}.
\end{proof}

\begin{Rem}\label{rem:pprimary}
	For any pair of primes $p,q$, we have $\cPd(\num{l},p,1)=\cPd(\num{l},q,1)$ in $\Exc{d}(\Sp^c,\Sp)^c$ since $\cat C_{p,1}=\cat C_{q,1}$ in $\Sp^c$. Hence, to determine the inclusions
	\[
		\cPd(\num{k},q,h') \subseteq \cPd(\num{l},p,h)
	\]
	among tt-primes, \cref{cor:inclusion_chromatic} allows us to restrict attention to the case $p=q$.
\end{Rem}

\begin{Rem}
	We now bring to bear the constraints on such inclusions supplied by the comparison map to the spectrum of the Goodwillie--Burnside ring (studied in \cref{sec:goodwillie-burnside,sec:comparison-map}):
\end{Rem}

\begin{Prop}\label{prop:inclusions_formalproperties}
	Fix a prime number $p$. Suppose that $\cPd(\num{k},p,h') \subseteq \cPd(\num{l},p,h)$ for integers $1 \le k,l \le d$ and chromatic integers $1 \le h,h' \le \infty$. Then:
	\begin{enumerate}
		\item $p - 1 \mid k-l \ge 0$; and
		\item if $h' = 1$ then $h = 1$ and $k = l$, so that the two tt-primes are equal. 
	\end{enumerate}
\end{Prop}

\begin{proof}
	\Cref{cor:inclusion-layers} implies $k \ge l$ while \cref{prop:comparison-map} implies that we have an inclusion $\mathfrak p(\num{k},p) \supseteq \mathfrak p(\num{l},q)$ in the Goodwillie--Burnside ring and hence $p-1\mid k-l$ by \cref{thm:burnside-goodwillie-spectrum}. This verifies (a). For (b), note that $\cPd(\num{k},p,1) \subseteq \cPd(\num{l},p,h)$ implies $h=1$ by \cref{cor:inclusion_chromatic}. It follows from \cref{prop:comparison-map} that $\mathfrak p(\num{k},0) \supseteq \mathfrak p(\num{l},0)$ and hence $k=l$ by \Cref{thm:burnside-goodwillie-spectrum}.
\end{proof}

\begin{Rem}\label{rem:spc_reduction}
	In summary, to determine the topology of $\Spc(\Exc{d}(\Sp^c,\Sp)^c)$, it is enough to determine the minimal number $\beth \ge 0$ such that
	\begin{equation}\label{eq:blue-inclusion}
		\cPd(\num{k},p,h+\beth) \subseteq \cPd(\num{l},p,h)
	\end{equation}
    whenever $p-1 \mid k-l \ge 0$ and for any $1 \leq h \leq \infty$. Moreover, since the map 
	\[
		\Spc(P_k)\colon \Spc(\Exc{k}(\Sp^c,\Sp)) \hookrightarrow \Spc(\Exc{d}(\Sp^c,\Sp))
	\]
	is an open embedding (\cref{prop:pn-induced-spectra}), the inclusion \eqref{eq:blue-inclusion} is equivalent to the inclusion
	\[
		\cat P_k(\num{k},p,h+\beth) \subseteq \cat P_k(\num{l},p,h).
	\]
	 This leads to:
\end{Rem}

\begin{Def}\label{def:geomblue}
	For each $p-1 \mid k-l\ge 0$ and $h \ge 1$, define
	\[
		\blue_p(k,l;h) \coloneqq \min\SET{\beth}{\cPd(\num{k},p,h+\beth) \subseteq \cPd(\num{l},p,h)}.
	\]
	It will follow from \cref{prop:sum-inclusion} below that the collection of such inclusions is nonempty, so that it is a well-defined natural number. Moreover, as noted above, it doesn't depend on the ambient $d \ge k$. These \emph{geometric blueshift numbers} determine the topology of the spectrum.
\end{Def}

\begin{Exa}\label{exa:blue-k=l}
	It follows from \cref{prop:verticalinclusions} that $\blue_p(l,l;h)=0$ for all $l$.
\end{Exa}

\begin{Rem}
	Our next goal is to relate the geometric blueshift to the blueshift of the Tate-derivatives.
\end{Rem}

\begin{Def}\label{def:tate-blue}
	For each prime number $p$, integers $k\ge l\ge 1$ and $h \ge 1$, the \emph{Tate blueshift number} is defined as
	\[ 
		\tblue_{p}(k,l;h) \coloneqq \height_p(\cL_{p,h-1}^f) - \height_p(\partial_lt_{k}\cL_{p,h-1}^f) 
	\]
	with $\height_p(0)=-1$. (Recall \cref{def:height}.) Thus:
	\[
		\tblue_{p}(k,l;h) 
		=        \begin{cases}
                            (h -1)- \height_p(\partial_lt_{k}\cL_{p,h-1}^f) & \text{if } \partial_lt_{k}\cL_{p,h-1}^f \neq 0; \\
                            h & \text{if } \partial_lt_{k}\cL_{p,h-1}^f = 0.
                 \end{cases}
	\]
	By definition, we have
	\begin{equation}\label{eq:tate-kernel}
		\cat C_{p,h-\tblue_p(k,l;h)} =
		\ker(-\otimes \partial_l t_k \cL_{p,h-1}^f : \Sp_{(p)}^c \to \Sp)
	\end{equation}
	with the convention that $\cat C_{p,0}$ is the whole category of finite $p$-local spectra. 
\end{Def}

The next lemma establishes the first relation between the two types of blueshift:

\begin{Lem}\label{lem:basic-tate-inclusion}
	If $h > \tblue_p(k,l;h)$, then 
	\[
		\cPd(\num{k},p,h) \subseteq \cPd(\num{l},p,h-\tblue_p(k,l;h)).
	\]
\end{Lem}

\begin{proof}
	Let $x \in \cPd(\num{k},p,h)$, i.e., $\partial_k(x) \in \cat C_{p,h}$, meaning $K(p,h-1)_*\partial_k(x) = 0$. Equivalently, we have
    \[
        0 = L_{p,h-1}^f\bbS \otimes \partial_k(x) \simeq \partial_k(i_d(L_{p,h-1}^f\bbS) \circledast x),
    \]
	hence $D_k\unit \circledast i_d(L_{p,h-1}^f\bbS) \circledast x = 0$. Since $t_k(-) \simeq \ihom{P_{k-1}\unit,\Sigma D_k\unit \circledast -}$, we deduce that $t_k (i_d(L_{p,h-1}^f\bbS) \circledast x) = 0$. Consequently,
    \[
        0 = \partial_lt_k (i_d L_{p,h-1}^f\bbS \circledast x) \simeq \partial_l(t_k(i_dL_{p,h-1}^f\bbS) \circledast x) \simeq \partial_l(t_k(i_dL_{p,h-1}^f\bbS)) \otimes \partial_l(x).
    \]
	It follows that $\partial_l(x) \in \cat C_{h-\tblue_p(k,l;h)}$ by the definition of Tate blueshift numbers~\eqref{eq:tate-kernel}. In other words, $x \in \cPd(\num{l},h-\tblue_p(k,l;h)).$
\end{proof}

\begin{Rem}\label{rem:tate-were-computed}
	We computed the Tate blueshifts in \cref{thm:tateblueshift}; namely,
	\[
		\tblue_p(k,l;h) = \begin{cases}
		1 & \text{if there exists a $p$-power partition of $k$ of length $l$}\\
		h & \text{otherwise}
		\end{cases}
	\]
	In particular, it does not depend on $h$ when there exists a $p$-power partition of $k$ of length $l$ and, in such a situation, we will drop the $h$ from the notation. Although the reader may wish to substitute the value $1$ in what follows, we have opted to keep the discussion in terms of $\tblue_p(k,l)$ for conceptual clarity. In any case, in order to amplify the connection between the geometric blueshift and the Tate blueshift we need to introduce the following auxiliary notion:
\end{Rem}

\begin{Def}
    Let $p$ be a prime number and pick integers $k > l \ge 1$. A \emph{chain of $p$-power partitions} between $k$ and $l$ is a sequence of integers
	\[
		(k=l_s > l_{s-1} > \cdots > l_1 > l_0 = l) 
	\]
	with $s>0$ such that for each consecutive pair $(l_\alpha,l_{\alpha-1})$ there exists a $p$-power partition of $l_\alpha$ of length $l_{\alpha-1}$ in the sense of \cref{def:ppowerpartition}. We call $s$ the \emph{length} of the chain. We write $\Chp(k,l)$ for the set of chains of $p$-power partitions between~$k$ and $l$; in symbols:
	\[
		\Chp(k,l) = \SET{k = l_s > l_{s-1} > \ldots > l_1 > l_0 = l}{s >0 \text{ and } \forall \alpha\colon \cP_p(l_\alpha;l_{\alpha-1})\neq \emptyset}.
	\]
\end{Def}

\begin{Rem}\label{rem:p-power-and-divisibility}
	If there exists a $p$-power partition of $k$ of length $l$ then $p-1 \mid k-l > 0$. The converse is not true, in general. (Consider, for example, $p=2$, $k=3$, $l=1$.) However, if $p-1 \mid k-l > 0$ then there exists a chain of $p$-power partitions from $k$ to~$l$, that is, $\Chp(k,l) \neq \emptyset$. Indeed, the statement is true when $k-l = p-1$, since we can write $k=p+\sum_{\alpha=1}^{l-1} p^0$. The general case follows from this, by using the chain
	\[
		(k > k-(p-1) > \ldots > k-a(p-1) = l) \in \Chp(k,l)
	\]
    where
    $a$ the natural number such that $k-l = a(p-1)$.
\end{Rem}

\begin{Prop}\label{prop:sum-inclusion}
	If $(k=l_s > \cdots > l_0 = l) \in \Chp(k,l)$ is a chain of $p$-power partitions, then
	\[
		\cPd(k,p,h+\sum_{\alpha=1}^s \tblue_p(l_\alpha,l_{\alpha-1})) \subseteq \cPd(l,p,h)
	\]
	for each $h \ge 1$.
\end{Prop}

\begin{proof}
	By definition, for each $\alpha=1,\ldots,s$, there exists a $p$-power partition of $l_{\alpha}$ of length $l_{\alpha-1}$, hence by \cref{thm:tateblueshift}, $\tblue_p(l_{\alpha},l_{\alpha-1};h)$ is independent of $h$; see \cref{rem:tate-were-computed}. It then follows from \cref{lem:basic-tate-inclusion} that we have an inclusion of prime ideals
	\[
		\cPd(l_\alpha,p,h' + \tblue_p(l_\alpha,l_{\alpha-1})) \subseteq \cPd(l_{\alpha-1},p,h')
	\]
	for all $1 \le \alpha \le s$ and all heights $h' \ge 1$. Iteratively combining these inclusions, we deduce that
	\[
		\cPd(k,p,h+\sum_{\alpha=1}^s \tblue_p(l_\alpha,l_{\alpha-1})) \subseteq \cPd(l,p,h)
	\]
	as desired.
\end{proof}

\begin{Rem}\label{rem:geom-well-defined}
	It follows from \cref{rem:p-power-and-divisibility} and \cref{prop:sum-inclusion} that the set of inclusions in \cref{def:geomblue} is nonempty, hence the geometric blueshift is well-defined for all $h \ge 1$. Although we have not defined the geometric blueshift for $h=\infty$, it follows that $\cPd(\num{k},p,n) \subseteq \cPd(\num{l},p,\infty)$ if and only if $n=\infty$ and $p-1 \mid k-l \ge 0$.
\end{Rem}

\begin{Def}\label{def:p-distance}
	Define the \emph{$p$-distance} $\delta_p(k,l)$ between $k>l$ as the smallest length of a chain of $p$-power partitions between $k$ and $l$, or $\infty$ if no such chain exists; in symbols:
    \[
        \delta_p(k,l) \coloneqq \inf(\SET{s}{(l_s>\ldots>l_0) \in \Chp(k,l)}).
    \]
	We also set $\delta_p(l,l)\coloneqq 0$.
\end{Def}

\begin{Thm}\label{thm:formula-for-geombluetake2}
	For each $p-1 \mid k-l\ge 0$ and $h \ge 1$, we have
	\begin{equation}\label{eq:formula-for-geombluetake2}
		\blue_p(k,l;h) = \delta_p(k,l).
	\end{equation}
	In particular, the geometric blueshift $\blue_p(k,l;h)$ does not depend on $h$.
\end{Thm}

\begin{proof}
	The $k=l$ case is immediate from \cref{exa:blue-k=l}, so we may assume $k > l$. Recall from \cref{rem:p-power-and-divisibility} that $p-1\mid k-l>0$ implies $\Chp(k,l)\neq \emptyset$, so $\delta_p(k,l)$ is a well-defined natural number. By \cref{prop:sum-inclusion} and \cref{thm:tateblueshift}, we have
	\[
		\blue_p(k,l;h) \le \sum_{\alpha=1}^s \tblue_p(l_\alpha,l_{\alpha-1}) = s
	\]
	for any $(l_s > \cdots > l_0)\in \Chp(k,l)$. Hence
	\[
		\blue_p(k,l;h) \le \delta_p(k,l).
	\]
	It remains to prove that $\blue_p(k,l;h) \ge \delta_p(k,l)$ or, in other words, to prove that if $n<\delta_p(k,l)$ then $\cPd(\num{l},p,h+n)\not\subseteq \cPd(\num{k},p,h)$. Due to the vertical inclusions among tt-primes (\cref{prop:verticalinclusions}), it suffices to consider the extremal case $n = \delta_p(k,l)-1$. Moreover, recall that we can assume $d=k$ (\cref{rem:spc_reduction}). In other words, it suffices to prove the following claim:
	\begin{itemize}
		\item[(P($d$))] 
		For all $p-1 \mid d-l \ge 0$ and $1 \le h <\infty$,
		\[
			\cPd(\num{d},p,h+\delta_p(d,l)-1) \not\subseteq \cPd(\num{l},p,h).
		\]
	\end{itemize}

	We prove this by strong induction on $d$. The base case $d=1$ holds trivially, so let $d>1$ and suppose (P($i$)) has been verified for all $1 \le i < d$. Set $n\coloneqq \delta_p(d,l)-1$ for notational convenience and write $\Exc{d} = \Exc{d}(\Sp^c,\Sp)$. Recall that $\cL^f_{h+n-1} \coloneqq i_dL^f_{p,h+n-1}\bbS \in \Exc{d}$ is the right idempotent for the finite localization
	\begin{equation}\label{eq:truncation}
		(-)_{\le h+n}\colon\Exc{d} \to \cL^f_{h+n-1}\circledast \Exc{d} =: (\Exc{d})_{\le h+n}
	\end{equation}
	which truncates below height $h+n$; see \cref{exa:chromatic-truncation} and \cref{def:inflation}. That is, 
	\[
		\Spc( (\Exc{d})_{\le h+n}^c) \hookrightarrow \Spc(\Exc{d}^c)
	\]
	is a homeomorphism onto $V_{\le h+n} \coloneqq \SET{\cPd(\num{k},p,m)}{1 \le m \le h+n} \subseteq \Spc(\Exc{d}^c)$. Since the space $V_{\le h+n}$ is finite, it follows that there exists $z \in \Exc{d}^c$ such that 
	\begin{equation}\label{eq:auxiliaryobject}
		\supp(z) \cap V_{\le h+n} = \overbar{\{ \cPd(\num{l},p,h)\}} \cap V_{\le h+n}.
	\end{equation}
    Indeed, the complement $W$ of $\overbar{\{ \cPd(\num{l},p,h)\}} \cap V_{\le h+n}$ in $V_{\le h+n}$ is finite and open, hence also finite and open when viewed as a subset of $\Spc(\Exc{d}^c)$. It follows from \cite[Proposition 2.14]{Balmer05a} that there exists $z \in \Exc{d}^c$ with $\supp(z) = W^c$ in $\Spc(\Exc{d}^c)$. Intersecting with $V_{\le h+n}$ then gives \cref{eq:auxiliaryobject}.

    In order to prove (P($d$)) we will first prove the following auxiliary claim:
	\begin{equation}\label{eq:geometricblueshift_auxclaim}
		t_d(\cL_{h+n-1}^f \circledast z) = 0.
	\end{equation}
	We can test this on derivatives (\cref{lem:conservativity}). In fact, since the Tate construction~$t_d$ of a $d$-excisive functor is $(d-1)$-excisive, it is enough to establish 
	\[
		\partial_i(t_d(\cL^f_{h+n-1}\circledast z))=0
	\]
	for $1 \le i \le d-1$. Since $\partial_i$ is symmetric monoidal and $z$ is dualizable, we need to show
	\begin{equation}\label{eq:partialvanishing}
		\partial_i(t_d(\cL^f_{h+n-1}))\otimes \partial_i z = 0
	\end{equation}
	in $\Sp$ for each $1\le i \le d-1$. This is trivially true if $d$ does not have a $p$-power partition of length $i$ by \cref{thm:tateblueshift}. Thus, assume $d$ has a $p$-power partition of length $i$. Then, also invoking \cref{thm:tateblueshift}, we have
	\[
		\ker(-\otimes \partial_i t_d \cL^f_{n+h-1})\cap \Spc(\Sp_{(p)}^c) = \cat C_{p,n+h-\tblue(d,i)} = \cat C_{p,n+h-1}
	\]
	and~\eqref{eq:partialvanishing} would follow from the claim that
	\begin{equation}\label{eq:partialvanishing2}
		\partial_i z \in \cat C_{p,n+h-1}.
	\end{equation}
	In order to verify \eqref{eq:partialvanishing2} and hence \eqref{eq:geometricblueshift_auxclaim}, we distinguish two cases. Suppose first that $\cPd(\num{i},p,n+h) \not\subseteq \cPd(\num{l},p,h)$. Then 
		\[
			\cPd(\num{i},p,n+h) \notin \overbar{\{ \cPd(\num{l},p,h)\}} \cap V_{\le h+n} = \supp(z) \cap V_{\le h+n},
		\]
	so in particular $\cPd(\num{i},p,n+h) \notin \supp(z)$. This is equivalent to the statement that $z \in \cPd(\num{i},p,n+h)$, hence $\partial_i z \in \cat C_{p,h+n} \subseteq \cat C_{p,h+n-1}$, as desired. 
        
	On the other hand, if $\cPd(\num{i},p,n+h)\subseteq \cPd(\num{l},p,h)$ then \cref{prop:inclusions_formalproperties} implies $p-1 \mid i-l \ge 0$. Then, invoking the  inductive  hypothesis (P($i$)),  we have
	\begin{equation}\label{eq:ind-invocation}
		\cat P_i(\num{i},p,h+\partial_p(i,l)-1) \not\subseteq \cat P_i(\num{l},p,h).
	\end{equation}
	The desired claim~\eqref{eq:partialvanishing2} is equivalent to $z \in \cPd(\num{i},p,n+h-1)$ or, in other words, to $\cPd(\num{i},p,n+h-1) \not\in \supp(z)$. If this were not the case then we would have
	\[
		\cPd(\num{i},p,n+h-1) \in \supp(z) \cap V_{\le n+h}
	\]
	which would imply
	\[
		\cat P_i(\num{i},p,n+h-1) \subseteq \cat P_i(\num{l},p,h).
	\]
	This would then imply that 
	\[
		\cat P_i(\num{i},p,h+\delta_p(i,l)-1) \not\subseteq \cat P_i(\num{i},p,n+h-1)
	\]
	because of~\eqref{eq:ind-invocation}, so that $h+\delta_p(i,l)-1 < n+h-1$ by \cref{prop:verticalinclusions}. Substituting the definition $n=\delta_p(d,l)-1$ this would mean that
	\[
		\delta_p(i,l)+1 < \delta_p(d,l) \le \delta_p(d,i)+\delta_p(i,l).
	\]
	Hence $1 < \delta_p(d,i)=1$ which is a contradiction. This completes the verification of~\eqref{eq:geometricblueshift_auxclaim}.

	Recall that the Tate construction $t_d\colon \Exc{d}\to \Exc{d}$ is the Tate construction associated to the Thomason closed subset $\supp(P_dh_{\bbS}(d))$ whose open complement is $X_{d-1} \coloneqq \SET{\cPd(\num{k},p,m)}{1\le k \le d-1}\subseteq \Spc(\Exc{d}^c)$. We can also consider the Tate construction $t_Y\colon (\Exc{d})_{\le h+n} \to (\Exc{d})_{\le h+n}$ on the truncated category associated to the preimage $Y\coloneqq \supp(P_dh_\bbS(d))\cap V_{\le h+n} = \supp(P_dh_\bbS(d)_{\le h+n})$. Since the right adjoint of the localization~\eqref{eq:truncation} is conservative, the Tate vanishing \eqref{eq:geometricblueshift_auxclaim} implies the vanishing
	\begin{equation}\label{eq:tatevanishing2}
		t_Y(z_{\le h+n})=0
	\end{equation}
	by \cref{rem:Utate}. By \cite[Proposition 2.29]{PatchkoriaSandersWimmer22}, the Tate vanishing \eqref{eq:tatevanishing2} translates into a disconnectedness statement for $\supp(z_{\le h+n})$: namely, we have a decomposition
	\[
		\supp(z_{\le h+n}) = Z_1 \sqcup Z_2
	\]
	in $\Spc((\Exc{d})_{\le h+n}^c)\cong V_{\le h+n}$ with $Z_1  \cap  X_{d-1} =\emptyset$ and $Z_2 \subseteq X_{d-1}$. Therefore, we have $\cPd(\num{l},p,h) \in Z_2$. If $\cPd(\num{d},p,n+h) \subseteq \cPd(\num{l},p,h)$ then we would have $\cPd(\num{d},p,n+h) \in Z_2$ which contradicts $Z_2 \subseteq X_{d-1}$. This establishes (P($d$)) for each $l < d$, and it is trivially true for $l=d$.
\end{proof}
    
\subsection*{The combinatorics of $p$-power partition chains}

\begin{Rem}
	Armed with \cref{thm:formula-for-geombluetake2}, we see that the task of giving an explicit description of the topology of the Balmer spectrum reduces to understanding the combinatorics of chains of $p$-power partitions. We have already seen (\cref{rem:p-power-and-divisibility}) that such a chain between $k>l$ exists (that is, $\Chp(k,l)\neq \emptyset$) if and only if $k-l$ is divisible by $p-1$. We will discover the surprising fact that in this case there always exists a chain of length at most $2$ (see \cref{prop:ppowerchains} below).
\end{Rem}

\begin{Def}
	Let $p$ be a prime and consider integers $k \geq l \ge 1$. We define the \emph{weight} of a $p$-power partition $k = \sum_{i=0}^ra_ip^i$ of $k$ (with $a_i \geq 0$ for all $i$) as the sum of the coefficients: $w(a) \coloneqq \sum_{i=0}^ra_i$ for $a = (a_0,\ldots,a_r)$.
\end{Def}

\begin{Rem}\label{rem:weight-properties}
	The weight has the following two properties:
    \begin{enumerate}
		\item\label{it:weight-congruence} Since $p^i\equiv 1$ mod $(p-1)$, it follows that $\sum_{i=0}^ra_ip^i\equiv w(a)$ mod $(p-1)$.
		\item\label{it:weight-minimizing} Among all $p$-power partitions of $k$, the expansion in base $p$ is the unique one that minimizes its weight. Indeed, suppose $k = \sum_{i=0}^ra_ip^i$ is some $p$-power partition of $k$. If one of the coefficients $a_j \geq p$, then the $p$-power partition $k = \sum_{i=0}^ra_i'p^i$ with 
            \[
				a_i' \coloneqq 
					\begin{cases}
						a_j-p & \text{if } i = j \\
						a_j+1 & \text{if } i = j+1 \\
						a_i & \text{otherwise}
					\end{cases}
            \]
        has smaller weight. Because the base $p$ expansion of any integer is unique, the claim follows.
    \end{enumerate}
	Motivated by the second property, we define:
\end{Rem}

\begin{Def}\label{def:spk}
	For any prime $p$ and integer $k \ge 1$, let $s_p(k)$ be the weight of the base $p$ expansion of $k$, i.e., the sum of the coefficients of the base $p$ expansion of $k$.
\end{Def}

\begin{Rem}
	The following two lemmas provide the number-theoretic input to our formula for the geometric blueshift numbers $\blue_{p}(k,l)$, treating separately the cases when $l \geq s_p(k)$ and $l < s_p(k)$. Moreover, the first lemma completes our observations in \cref{rem:p-power-and-divisibility}.
\end{Rem}

\begin{Lem}\label{lem:subclaimA}
	There exists a $p$-power partition of $k$ of length $l$ (that is, one can write $k$ as a sum of $l$ powers of $p$) if and only if $p-1 \mid k-l \ge 0$ and $l \ge s_p(k)$.
\end{Lem}

\begin{proof}
	($\Rightarrow$): Suppose $k$ can be written as a sum of $l$ powers of $p$, say $k = \sum_{j=1}^lp^{e(j)}$ for some $e(j) \geq 0$. We have $k \geq l$ and $\smash{k-l = \sum_{j=1}^l(p^{e(j)}-1)}$ is divisible by $p-1$. Property~(b) of \cref{rem:weight-properties} implies that $l \geq s_p(k)$.

	($\Leftarrow$): For the converse, let $k = \sum_{i=0}^r a_ip^i$ be the expansion of $k$ in base $p$. Property~(a) of \cref{rem:weight-properties} shows that $0 \equiv k-l \equiv s_p(k) - l$ mod $(p-1)$. If $l = s_p(k)$, then the expansion in base $p$ gives a $p$-power expansion of length $l$. Consider then the case $l > s_p(k)$. Starting from the expansion of $k$ in base $p$, we have to argue that we can increase the length of the corresponding $p$-power partition 
    \[
        (\underbrace{p^r,\ldots,p^r}_{a_r},\underbrace{p^{r-1},\ldots,p^{r-1}}_{a_{r-1}}, \ldots, \underbrace{p^0,\ldots,p^0}_{a_0})  \vdash k
    \]
    by multiples of $(p-1)$ until we reach $l$. The case $k=l$ is clear, so assume $k>l$ and pick some $j>0$ with $a_j >0$, which exists because $k > p-1$ by our divisibility assumption on $k-l$. Setting 
    \[
        a_i' \coloneqq 
            \begin{cases}
                a_j-1 & \text{if } i = j \\
                a_{j-1}+p & \text{if } i = j-1 \\
                a_i & \text{otherwise}
            \end{cases}
    \]
	gives a presentation $k = \sum_{i=0}^ra_i'p^i$ of $k$ whose corresponding $p$-power partition is of length $s_p(k) + (p-1)$. Continuing this process if necessary until we reach weight $l$, we obtain the claim.
\end{proof}

\begin{Lem}\label{lem:subclaimB}
    If $0< l < s_p(k)$ and $p-1$ divides $k-l$, then there exists an integer $i<k$ that is congruent to $k$ and $l$ modulo $p-1$ and which satisfies $i \geq s_p(k)$ and $l \geq s_p(i)$.
\end{Lem}

\begin{proof}
    Consider $k = \sum_{i=0}^r a_ip^i$, the expansion of $k$ in base $p$ with $a_r \neq 0$. If $k < p$, then $s_p(k) = k$ and $p-1$ cannot divide $k-l$, so we may assume $k \geq p$, i.e., $r >0$. Let $1 \le x \leq p-1$ be the integer that is congruent to $k$ modulo $p-1$ and set $i \coloneqq p^r + (x-1)$. We will verify that this choice of $i$ satisfies the conditions given in the statement. Since $r > 0$, we have
        \[
            i \equiv x \equiv k \equiv l \mod p-1.
        \]
    Next, $s_p(i) = x$ is the smallest positive integer congruent to $l$ modulo $p-1$, hence $l \geq s_p(i)$. By construction, $k \geq i$; in fact, since $k \neq i$ as $s_p(k) > l \geq s_p(i)$, we get $k > i$. It remains to see that $i \geq s_p(k)$, for which we introduce a case distinction. Suppose first that $r \geq 2$. Then $i \geq p^r \geq (r+1)(p-1) \geq s_p(k)$. If $r = 1$ instead, then $s_p(k) \leq 2(p-1)$. If $s_p(k) < p$, then $x = s_p(k)$ and thus $i = p + x > s_p(k)$. Finally, if $p \leq s_p(k) \leq 2(p-1)$, then $x = s_p(k) - (p-1)$, so $i = p + (x-1) = s_p(k)$. We have shown that $i \geq s_p(k)$ in all cases.
\end{proof}

\begin{Prop}\label{prop:ppowerchains}
    Let $p$ be a prime and consider integers $k \geq l \geq 1$. Then:
        \[
            \delta_p(k,l) = 
                \begin{cases}
					0 & \text{if } k=l;\\
                    1 & \text{if } p-1 \mid k-l>0 \text{ and } l \geq s_p(k); \\
                    2 & \text{if } p-1 \mid k-l>0 \text{ and } l < s_p(k); \\
                    \infty & \text{otherwise}.
                \end{cases}
        \]
\end{Prop}

\begin{proof}
	This follows from \cref{lem:subclaimA} and \cref{lem:subclaimB}.
\end{proof}

Combining everything together we have:

\begin{Thm}\label{thm:posetstructure}
	Let $p,q$ be prime numbers, $1\le k,l\le d$ integers, and suppose $1\leq h,h' \leq \infty$. There is an inclusion $\cPd(\num{k},p,h') \subseteq \cPd(\num{l},q,h)$ if and only if the following conditions hold:
        \begin{enumerate}
			\item $p-1\mid k-l\geq 0$; 
            \item $h' \ge h + \delta_p(k,l)$; and
            \item if $h > 1$, then $p=q$.
        \end{enumerate}
\end{Thm}

\begin{proof}
	This is the culmination of \cref{thm:formula-for-geombluetake2}, \cref{prop:ppowerchains}, \cref{cor:inclusion_chromatic}, and \cref{prop:inclusions_formalproperties} bearing in mind  \cref{rem:geom-well-defined} for the $h=\infty$ case.
\end{proof}

\begin{Rem}
	A complete depiction of the spectrum for $d=3$ is displayed in \cref{fig:exc3} on page \pageref{fig:exc3}. For larger $d$ such pictures start to become unwieldy and are more easily drawn one prime at a time. For example, the $d=4$ case is displayed in \cref{fig:exc4} below. The mathematical justification for working $p$-locally is as follows: Recall from \cref{rem:p-local-derivative} that the $p$-localization $\Exc{d}(\Sp^c,\Sp)\to\Exc{d}(\Sp^c,\Sp)_{(p)}$ is a finite localization which induces an identification
	\[
		\Spc(\Exc{d}(\Sp^c,\Sp)_{(p)}^c) \xrightarrow{\sim} \SET{\cPd(\num{k},p,n)}{\text{all } k, n} \subseteq \Spc(\Exc{d}(\Sp^c,\Sp)^c).
	\]
	Moreover, since all nontrivial inclusions occur $p$-locally (\cref{rem:pprimary}), the spectra of the $p$-localizations completely describe the spectrum of the unlocalized category.
\end{Rem}

\begin{figure}[h!]
\includegraphics{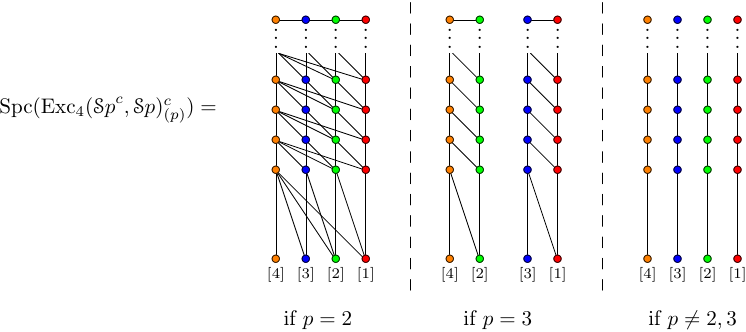}\caption{The $p$-local part of the Balmer spectrum of 4-excisive functors from finite spectra to spectra.}\label{fig:exc4}
\end{figure}

\section{Applications}\label{sec:applications}

\subsection*{The classification of tt-ideals}

We can now give the classification of thick ideals in $\Exc{d}(\Sp^c,\Sp)^c$. Throughout this section we fix $d \ge 1$ and abbreviate $\Exc{d}(\Sp^c,\Sp)^c$ by $\Exc{d}^c$ when convenient.

\begin{Prop}\label{prop:thomason-closed}
	Let $Z =\overline{\{\cat P_1\}} \cup\cdots\cup \overline{\{\cat P_N\}}$ be a closed subset of $\Spc(\Exc{d}^c)$ (see \cref{prop:spc_posetreduction}). Assume that this is irredundant, that is, $\cat P_i \not\subseteq \cat P_j$ for all $1 \le i \neq j \le N$. Let $\cat P_i = \cPd(\num{k_i},p_i,m_i)$ for all $i=1,\ldots,N$. Then the complement of $Z$ is quasi-compact if and only if all the chromatic integers $m_1,\ldots,m_N$ are finite.
\end{Prop}

\begin{proof}
	We follow the approach of \cite[Proposition 10.1]{BalmerSanders17}.
	
	($\Leftarrow$) Suppose the $m_1,\ldots,m_N$ are finite. We claim that $Z$ has a quasi-compact complement. Since the supports of compact objects form a basis of closed sets, it suffices to prove that whenever $Z=\cap_{i\in I}\supp(x_i)$ for a collection $\{x_i\}_{i \in I}$ of objects in $\Exc{d}^c$, there exists a finite subset $J \subseteq I$ such that $Z=\cap_{i\in J}\supp(x_i)$. For each $l\in [d]=\{1, \ldots, d\}$, let $\varphi_l\colon\Spc(\Sp^c)\to\Spc(\Exc{d}^c)$ be the map on spectra induced by the $l$th derivative $\partial_l \colon \Exc{d}\to\Sp$. Since $\Spc(\Exc{d}^c)=\cup_{1\le l\le d} \im \varphi_l$ by \cref{lem:cover}, it suffices to prove that for each $l$ there exists a finite subset $J_l \subseteq I$ such that 
	\[
		\cap_{j \in J_l} \supp(x_j) \cap \im \varphi_l \subseteq Z\cap \im \varphi_l
	\]
	(since we can then take $J \coloneqq \cup_{1 \le l \le d} J_l$). Moreover, since $\varphi_l$ is an embedding of $\Spc(\Sp^c)$ onto $\im \varphi_l$ (\cref{lem:injectivity-of-partial-k}), this reduces to the claim that
	\[
		\varphi_l^{-1}(\cap_{j \in J_l} \supp(x_j)) \subseteq \varphi_l^{-1}(Z).
	\]
	Now it follows from \cref{thm:posetstructure} that $\varphi_l^{-1}(\overline{\{\cPd(\num{k_i},p_i,m_i)\}})$ is either empty or of the form $\overline{\{\cat C_{p_i,n_i}\}}$ for some finite $n_i \ge m_i$. This closed subset of $\Spc(\Sp^c)$ has quasi-compact complement by \cite[Corollary 9.5(d)]{Balmer10b}. We conclude that indeed there is a finite $J_l\subseteq I$ such that \[ \cap_{j \in J_l} \supp(\partial_l x_j) \subseteq \varphi_l^{-1}(Z)\] as desired.

	($\Rightarrow$) Conversely, suppose that one of the $m_i=\infty$. Write $Z=Z' \cup \overline{\{\cPd(\num{k},p,\infty)\}}$, regrouping the other irreducibles under $Z'$. Since $\cat C_{p,\infty} = \cap_{1\le n < \infty} \cat C_{p,n}$, it follows that $\cPd(\num{k},p,\infty) =\cap_{1\le n < \infty}\cPd(\num{k},p,n)$. Hence $Z=\cap_{1\le n <\infty}(Z' \cup \overline{\{\cPd(\num{k},p,n)\}})$. If the complement of $Z$ is quasi-compact then there is a finite $n$ such that $Z=Z' \cup \overline{\{\cPd(\num{k},p,n)\}}$. But since $\cPd(\num{k},p,n) \not\subseteq \cPd(\num{k},p,\infty)$, it follows that $\cPd(\num{k},p,n) \in Z'$ and hence that $\cPd(\num{k},p,\infty)\in Z'$. But this implies that $\cPd(\num{k},p,\infty)$ is contained in one of the other irreducibles of $Z$, contradicting the assumption that the $\cat P_1,\ldots,\cat P_N$ were irredundant.
\end{proof}

\begin{Cor}\label{cor:desc-of-thomasons}
	The Thomason subsets of $\Spc(\Exc{d}^c)$ are the arbitrary unions of $\overline{\{\cPd(\num{k},p,m)\}} = \SET{\cat Q}{\cat Q \subseteq \cPd(\num{k},p,m)}$ for arbitrary $1\le k\le d$, $p$ prime and \emph{finite} chromatic integer $1 \le m <\infty$.
\end{Cor}

\begin{proof}
	This is an immediate consequence of \cref{prop:thomason-closed}, \cref{prop:spc_posetreduction} and \cref{thm:spec-as-a-set}.
\end{proof}

\begin{Not}\label{not:PNinf}
	Let $\mathbb{P}$ denote the set of prime numbers and let
	\[
		\mathbb{N}_{\infty} \coloneqq \{0,1,2,\ldots\}\cup\{\infty\}
	\]
	with the natural strict ordering $<$ with the understanding that $\infty \nless \infty$ and $\infty + n =\infty$ for any finite $n$.
\end{Not}

\begin{Def}\label{def:admissible}
    Let $d \ge 1$ be an integer. Say that a function $f\colon [d] \times \mathbb{P} \to \mathbb{N}_{\infty}$ is \emph{admissible} if it satisfies the following two conditions:
	\begin{enumerate}
		\item $f(k,p) \le \delta_p(k,l) + f(l,p)$ for all $p-1 \mid k-l\ge 0$; and
		\item if $f(k,p)=0$ for some prime $p$ then $f(k,q)=0$ for every prime $q$.
	\end{enumerate}
	Recall that the function $\delta_p(k,l)$ was computed in \cref{prop:ppowerchains}.
\end{Def}

\begin{Thm}[Classification of tt-ideals]\label{thm:tt-ideal-classification}
	Let $d \ge 1$ be an integer. There is a one-to-one correspondence between the set of admissible functions $[d]\times \mathbb{P} \to \mathbb{N}_\infty$ and the Thomason subsets of $\Spc(\Exc{d}(\Sp^c,\Sp)^c)$ which maps $f$ to
	\[
		Y_f \coloneqq \SET{\cPd(\num{k},p,m)}{m > f(k,p)}.
	\]
	Consequently, there is a one-to-one correspondence between admissible functions and tt-ideals in $\Exc{d}(\Sp^c,\Sp)^c$ which maps $f$ to
	\[
		\cat I_f \coloneqq \SET{x \in \Exc{d}^c}{\partial_k(x) \in \cat C_{p,f(k,p)} \text{ for each } (k,p)\in [d]\times\mathbb{P} \text{ such that } f(k,p)>0}.
	\]
\end{Thm}
\begin{proof}
	We will use repeatedly that Thomason subsets are closed under specialization. With this in mind, let $Y\subseteq \Spc(\Exc{d}^c)$ be a Thomason subset. Then 
	\begin{equation}\label{eq:Ydecomp}
		Y=\bigcup_{(k,p)\in[d]\times \mathbb{P}} X(k,p)
	\end{equation}
	where $X(k,p)\coloneqq \SET{\cPd(\num{k},p,n)}{ \cPd(\num{k},p,n) \in Y}$. The nontrivial inclusion~$\subseteq$ in~\eqref{eq:Ydecomp} follows from \cref{thm:spec-as-a-set}. Then consider a pair $(k,p) \in {\num{d}\times \mathbb{P}}$. If $X(k,p)$ is nonempty then $\cPd(\num{k},p,\infty) \in X(k,p)$. Since $Y$ is Thomason, \cref{cor:desc-of-thomasons} implies that $\cPd(\num{k},p,\infty) \subseteq \cPd(\num{l},q,h) \in Y$ for some $(l,q)\in \num{l}\times\mathbb{P}$ and \emph{finite} $h$. It then follows from \cref{thm:posetstructure} that $\cPd(\num{k},p,h+\delta_p(k,l)) \subseteq \cPd(\num{l},q,h)$ and, moreover, that $h+\delta_p(k,l)$ is finite. This establishes that if $X(k,p)$ is nonempty then there is a finite $n\ge 1$ such that $\cPd(\num{k},p,n)\in X(k,p)$. It follows that every nonempty $X(k,p)$ is of the form $\SET{\cPd(\num{k},p,n)}{n_{k,p} \le n \le \infty}$ for some (uniquely determined) finite $n_{k,p} \ge 1$.

	We then define a function $f\colon \num{d}\times \mathbb{P}\to \bbN_{\infty}$ by setting $f(k,p)=\infty$ when $X(k,p)=\emptyset$ and $f(k,p)=n_{k,p}-1$ when $X(k,p)$ is nonempty. With this definition, $X(k,p) = \SET{\cPd(\num{k},p,n)}{n > f(k,p)}$. We claim that the function $f$ is admissible. Indeed, condition (b) follows simply from the fact that $\cPd(\num{k},p,1)=\cPd(\num{k},q,1)$ for any two primes $p$ and $q$. On the other hand, condition~(a) follows from \cref{thm:posetstructure}, including the cases where one side of the inequality is equal to infinity. 
 
	In this way we obtain an injective map from the set of Thomason subsets to the set of admissible functions. We claim that every admissible function arises in this way. Indeed, given an admissible function $f$, consider the subset
	\[
		Y \coloneqq \bigcup_{\substack{(l,q)\in \num{d}\times \mathbb{P}:\\f(l,q)<\infty}} \overline{\{\cPd(\num{l},q,f(l,q)+1)\}}.
	\] 
	It is Thomason by \cref{cor:desc-of-thomasons}. We claim that $\cPd(\num{k},p,n)\in Y$ if and only if $n > f(k,p)$. The $(\Leftarrow)$ direction is immediate from the definition of $Y$. On the other hand, if $\cPd(\num{k},p,n)\in Y$ then $\cPd(\num{k},p,n) \subseteq \cPd(\num{l},q,f(l,q)+1)$ for some $(l,q)$ such that $f(l,q)<\infty$. \Cref{thm:posetstructure} then implies that $p-1\mid k-l \ge 0$ and that $n > f(l,q)+\delta_p(k,l)$ and, moreover, that $p=q$ if $f(l,q) \neq 0$. Since $f$ is admissible, it follows that $n > f(k,p)$ as desired.

	This establishes the claimed bijection between admissible functions and Thomason subsets. The translation from Thomason subsets to tt-ideals is \cite[Theorem~4.10]{Balmer05a} just unravelling the definitions.
\end{proof}

\begin{Rem}
	The above bijective correspondence is order-preserving when the set of admissible functions $[d]\times \mathbb{P} \to \mathbb{N}_{\infty}$ is endowed with the pointwise order induced from the order on $\mathbb{N}_{\infty}$.
\end{Rem}

\begin{Rem}\label{rem:typefunctions}
	Following the perspective of \cite{bgh_balmer}, we can define the \emph{type function} of a compact $d$-excisive functor $x \in \Exc{d}(\Sp^c,\Sp)^c$ as the function which collects the types of the derivatives $\partial_k x \in \Sp^c$ of $x$ at each prime $p$, i.e., the function $\type_x\colon\num{d}\times\mathbb{P} \to \mathbb{N}_\infty$ defined by
        \[
            \type_x(k,p) \coloneqq \inf\{h \in \mathbb{N}_\infty \mid K(p,h)_*\partial_k x \neq 0\}
        \]
    with the convention that $\inf \emptyset = \infty$. We can extend the notion from objects $x$ to tt-ideals $\cat I \subseteq \Exc{d}(\Sp^c,\Sp)^c$ as follows:
        \[
            \type_{\cat I}\colon [d]\times \mathbb{P} \to \mathbb{N}_\infty, \quad (l,p) \mapsto \inf\{\type_x(l,p)\mid x\in \cat I\}.
        \]
	One readily checks from the proof of \cref{thm:tt-ideal-classification} that the function $\cat I \mapsto \type_{\cat I}$ which sends a tt-ideal to its type function is the inverse of the bijection $f \mapsto \cat I_f$.
\end{Rem}

\begin{Def}\label{def:padmissible}
	The classification of tt-ideals simplifies $p$-locally. Say that a function $f\colon [d]\to \mathbb{N}_{\infty}$ is \emph{$p$-admissible} if 
	\[
		f(k) \le \delta_p(k,l) + f(l)
	\]
    for all $p-1\mid k-l \ge 0$. 
\end{Def}

\begin{Cor}\label{cor:plocal-tt-ideal-classification}
	Let $d \ge 1$ be an integer and $p$ a prime. There is a one-to-one correspondence between the set of $p$-admissible functions $[d] \to \mathbb{N}_{\infty}$ and the Thomason subsets of $\Spc(\Exc{d}(\Sp^c,\Sp)^c_{(p)})$ which maps $f$ to $Y_f \coloneqq \SET{\cPd(\num{k},p,m)}{m > f(k)}$.
\end{Cor}

\begin{Exa}
	Note that the empty Thomason subset  corresponds to the constant $p$-admissible function $f \equiv \infty$ while the whole space  corresponds to the constant function $f \equiv 0$.
\end{Exa}

\begin{Rem}
	In the situation of \cref{cor:plocal-tt-ideal-classification}, we observe that Thomason subsets are automatically closed. Hence every every tt-ideal of $\smash{\Exc{d}(\Sp^c,\Sp)^c_{(p)}}$ is generated by a single compact object; see \cite[Lemma~2.3]{Sanders13}. Since \cref{rem:targetcat} supplies a geometric equivalence
		\[
			\Exc{d}(\Sp^c,\Sp)_{(p)} \simeq \Exc{d}(\Sp^c,\Sp_{(p)})
		\]
	we thus obtain the statement of our main theorem on page \pageref{thmmain}.
\end{Rem}

\begin{Rem}
	The behaviour of the function $\delta_p(k,l)$ depends on the behaviour of the function $s_p(k)$ which can be quite erratic. Further information about the latter function can be found in \cite[Section 3.2]{AlloucheShallit03}. Its generating function has a remarkable closed form formula; see \cite{AdamsWattersRuskey09}.
\end{Rem}

\subsection*{Transchromatic Smith and Floyd theory in functor calculus}
A famous theorem of P.A.~Smith \cite{Smith41} states that if $H$ is a subgroup of a $p$-group $G$ and $X$ is a finite-dimensional based $G$-CW-complex, then  
	\[
		\tilde H_*(X^H;\Fp) = 0 \implies \tilde H_*(X^G;\Fp) = 0. 
	\]
This was later extended by Floyd \cite{Floyd52}, who showed that if $H$ is a subgroup of a $p$-group $G$ and $X$ is a finite-dimensional $G$-CW complex with $\dim_{\Fp}H_*(X^H;\Fp)$ finite, then 
	\[
		\dim_{\Fp}H_*(X^H;\Fp) \ge \dim_{\Fp}H_*(X^G;\Fp).
	\]
Noting that mod $p$ homology corresponds to $K(p,\infty)$, i.e., Morava $K$-theory at height $\infty$, Hausmann and Kuhn--Lloyd \cite{KuhnLloyd2024} considered versions of the Smith and Floyd theorems for Morava $K$-theory $K(p,n)$. They showed that understanding such ``transchromatic'' Smith and Floyd theorems is equivalent to understanding the topology of the Balmer spectrum of the category of compact $G$-spectra. Our aim in this section is to relate the topology of the Balmer spectrum of compact $d$-excisive functors to analogs of the Smith and Floyd inequalities in functor calculus.

\begin{Def}\label{def:smithandfloyd}
	Let $d \ge 1$ be an integer and $p$ a prime. For each $1 \le k,l\leq d$ and $0\le n,h \leq \infty$, we say that
    \begin{itemize}
		\item $\Smith_{d,p}(k,l;n,h)$ holds if: for all $x \in \Exc{d}(\Sp^c,\Sp)^{c}$, we have
		\[
			K(p,n)_*(\partial_kx) =0\implies K(p,h)_*(\partial_lx) =0;
		\]
		\item $\Floyd_{d,p}(k,l;n,h)$ holds if: for all $x \in \Exc{d}(\Sp^c,\Sp)^{c}$, we have 
		\[
			\dim_{K(p,n)_*}K(p,n)_*(\partial_kx) \geq \dim_{K(p,h)_*}K(p,h)_*(\partial_lx).
		\]
    \end{itemize}
\end{Def}

\begin{Rem}\label{rem:smith-floyd-p-local}
	Recall from \cref{rem:p-local-derivative} that  $p$-localization 
		\[
			\Exc{d}(\Sp^c,\Sp)\to\Exc{d}(\Sp^c,\Sp)_{(p)}
		\]
	is the finite localization associated to $i_d(f_{p,\infty})$. Hence \cref{lem:partial-splits-inflation} implies that the Goodwillie derivative $\partial_k$ commutes with $p$-localization. Since~$K(p,n)$ is itself $p$-local, it follows that $K(p,n)_*(\partial_k x)\simeq K(p,n)_*(\partial_k(x_{(p)}))$ where~$x_{(p)}$ denotes the $p$-localization of~$x$. Thus, the statements in \cref{def:smithandfloyd} could be equivalently formulated by instead letting $x$ range over all compact $p$-local functors $x \in \Exc{d}(\Sp^c,\Sp)_{(p)}^c$. To see this, just recall that if $x \in \Exc{d}(\Sp^c,\Sp)_{(p)}^c$ then $x \oplus \Sigma x$ is the $p$-localization of a compact $x' \in \Exc{d}(\Sp^c,\Sp)^c$.
\end{Rem}

\begin{Prop}\label{prop:smithfloyd_equivalence}
    For all $1\le k,l\leq d$ and $0 \le n,h \leq \infty$, we have
		\[
			\Smith_{d,p}(k,l;n,h) \iff \Floyd_{d,p}(k,l;n,h).
		\]
\end{Prop}

\begin{proof}
    The ($\Longleftarrow$) implication is clear. In order to establish the converse, we will prove the contrapositive: Suppose there exists some $x \in \Exc{d}(\Sp^c,\Sp)^{c}$ with 
        \[
			\dim_{K(p,n)_*}K(p,n)_*(\partial_kx) < \dim_{K(p,h)_*}K(p,h)_*(\partial_lx).
        \]
    We may replace $x$ with its $p$-localization (\cref{rem:smith-floyd-p-local}). We will employ the methods of J.~Smith \cite[Appendix C]{Ravenel92}, as expanded upon in the work of Kuhn--Lloyd \cite{KuhnLloyd2024}, to construct a compact $p$-local $d$-excisive functor $y$ from $x$ which falsifies the corresponding statement $\Smith_d(k,l;n,h)$. Note that, for any $m \ge 1$, the symmetric group acts on $x^{\circledast m}$ by permuting the factors, giving rise to a map of rings
        \[
            \Z_{(p)}[\Sigma_m] \to \pi_0\End(x^{\circledast m}).
        \]
    In particular, any idempotent $e \in \Z_{(p)}[\Sigma_m]$ induces a self-map $e\colon x^{\circledast m} \to x^{\circledast m}$, and we write $ex^{\circledast m}$ for the corresponding retract of $x^{\circledast m}$. The monoidality of~$\partial_l(-)$ and $K(p,h)_*(-)$ imply that this construction is compatible with its algebraic counterpart: 
        \[
            K(p,h)_*(\partial_l(ex^{\circledast m})) \cong e(K(p,h)_*(\partial_lx))^{\otimes m}.
        \]
    Set $V_* \coloneqq K(p,h)_*(\partial_lx)$ as a graded module over $K(p,h)_*$ and similarly $W_* \coloneqq K(p,n)_*(\partial_kx)$ as a graded $K(p,n)_*$-module. The results of \cite[Section 6.4]{KuhnLloyd2024} provide a choice of integer $m$ and idempotent $e_m \in \Z_{(p)}[\Sigma_m]$ such that\footnote{There is a subtlety at the prime $p=2$ related to the exotic multiplicative structure of $K(n)$, which however can be `filtered away' as in \cite[Section 6.3]{KuhnLloyd2024}.} 
        \[
            e_mV_*^{\otimes m} \neq 0 \quad \text{and} \quad e_mW_*^{\otimes m} = 0.
        \]
    It follows that $y\coloneqq e_mx^{\circledast m}$ witnesses the failure of $\Smith_d(k,l;n,h)$, as desired. 
\end{proof}

\begin{Rem}
   The relation with the topology of the Balmer spectrum is the following: 
\end{Rem}

\begin{Prop}\label{prop:smith-topology}
	Let $1 \le k,l \le d$ and $0 \le n,h \le \infty$. Then $\Smith_{d,p}(k,l;n,h)$ holds if and only if $\cPd(\num{k},p,n+1) \subseteq \cPd(\num{l},p,h+1)$ in $\Exc{d}(\Sp^c,\Sp)^c$.
\end{Prop}

\begin{proof}
	Unravelling the definitions one sees that $\Smith_{d,p}(k,l;n,h)$ is equivalent to the statement that for every compact $x$, $\cPd(\num{l},p,h+1) \in \supp(x)$ implies $\cPd(\num{k},p,n+1) \in \supp(x)$. This is equivalent to $\cPd(\num{k},p,n+1)\subseteq\cPd(\num{l},p,h+1)$ since the supports of compact objects form a basis of closed sets for the Balmer topology.
\end{proof}

\begin{Cor}\label{cor:floydinequality}
    Consider integers $1 \leq l \leq k \leq d$ satisfying $p-1 \mid k-l$, a height $0\le h\le \infty$, and let $n \geq h +\delta_p(k,l)$. For any compact $x \in \Exc{d}(\Sp^c,\Sp)^c$, the following inequality holds:
        \[
            \dim_{K(p,n)_*}K(p,n)_*(\partial_kx) \geq \dim_{K(p,h)_*}K(p,h)_*(\partial_lx).
        \]
\end{Cor}

\begin{proof}
	This is a direct consequence of \cref{prop:smithfloyd_equivalence}, \cref{prop:smith-topology} and \cref{thm:posetstructure}.
\end{proof}

\subsection*{Calculus with coefficients}\label{ssec:calculuswithcoefficients}

The aim of this subsection is to explain how our computation of the spectrum of $\Exc{d}(\Sp^c,\Sp)^c$ can be extended to other  coefficient categories $\cat D$ with a particular focus on the case $\cat D = \Mod_{\HZ}$. In the context of stable equivariant homotopy theory, this is the subject of \cite{PatchkoriaSandersWimmer22,bhs1,BCHNP2023Quillen}.

\begin{Rem}\label{rem:generalcoefficients}
    Let $\cat D$ be a rigidly-compactly generated tt-$\infty$-category and fix some integer $d \ge 1$. The category $\Exc{d}(\Sp^c,\cat D)$ of reduced $d$-excisive functors from finite spectra to $\cat D$ may then again be equipped with the symmetric monoidal structure afforded by localized Day convolution, and as such forms a rigidly-compactly generated \mbox{tt-$\infty$-category}. As explained in \cref{rem:targetcat}, there is a canonical symmetric monoidal equivalence $\Exc{d}(\Sp^c,\cat D) \simeq \Exc{d}(\Sp^c,\Sp) \otimes \cat D$ and we obtain $\cat D$-linear derivatives $\partial_k^{\cat D} \colon \Exc{d}(\Sp^c,\cat D) \to \cat D$ for all $k \in [d]$. 
    
    Let $i_d \colon \Sp \to \Exc{d}(\Sp^c,\Sp)$ be the functor defined in \cref{def:inflation} and write~$i_d^{\cat D}$ for its $\cat D$-linear analogue. For any $1 \leq k \leq d$, we thus obtain a commutative diagram 
	\[\begin{tikzcd}
		\Sp \ar[d,"F"'] \ar[r,"i_d"] & \Exc{d}(\Sp^c,\Sp) \ar[d,"F_d"'] \ar[r,"\partial_k"] & \Sp \ar[d,"F"] \\
		\cat D \ar[r,"i_d^{\cat D}"] & \Exc{d}(\Sp^c,\cat D) \ar[r,"\partial_k^{\cat D}"] & \cat D
		\end{tikzcd}\]
    in which the vertical functors are induced by base-change along the canonical geometric functor $F\colon \Sp \to \cat D$ (\cref{rem: F_A}). Since the top composite is an equivalence (\cref{lem:partial-splits-inflation}), by base-change the bottom horizontal composite is one as well. This setup allows us to generalize \cref{thm:spec-as-a-set}:
\end{Rem}

\begin{Prop}\label{prop:setspc_coefficients}
    Suppose $\cat D$ is a rigidly-compactly generated tt-$\infty$-category. The ($\cat D$-linear) derivatives induce a bijection
        \[
            \varphi^{\cat D}\coloneqq \Spc((\partial_k^{\cat D})_{1\leq k \leq d})\colon \coprod_{1\leq k \leq d}\Spc(\cat D^c) \xrightarrow{\sim} \Spc(\Exc{d}(\Sp^c,\cat D)^c)
        \]
    whose restriction to each component is an embedding.
\end{Prop}

\begin{proof}
    By induction on the $\cat D$-linear Taylor tower, we see that the functors $\partial_k^{\cat D}$ are jointly conservative, so $\varphi^{\cat D}$ is surjective by \cite[Theorem 1.3]{barthel2023surjectivity}. Consider the following commutative diagram 
	\[\begin{tikzcd}
		\coprod_{1\leq k \leq d}\Spc(\cat D^c) \ar[r,"\varphi^{\cat D}"] & \Spc(\Exc{d}(\Sp^c,\cat D)^c) \ar[d,"\Spc(i_d^{\cat D})"] \\
		\Spc(\cat D^c) \ar[r,equals] \ar[u,"j_k"'] & \Spc(\cat D^c).
		\end{tikzcd}\]
    Here $j_k$ denotes the inclusion of the $k$-th component and the bottom map is the identity because $\partial_k^{\cat D} \circ i_d^{\cat D} \simeq \id$ by \cref{rem:generalcoefficients}. This implies that $\Spc(\partial_k^{\cat D}) = \varphi^{\cat D} \circ j_k$ is an embedding (as in the proof of \cref{lem:injectivity-of-partial-k}) and it remains to show that the images of the different components are disjoint.

    To this end, consider the following commutative diagram
	\begin{equation}\label{eq:dlinearsquare}
        \begin{tikzcd}
		\coprod_{1\leq k \leq d}\Spc(\cat D^c) \ar[r,"\varphi^{\cat D}"] \ar[d,"\coprod\Spc(F)"'] & \Spc(\Exc{d}(\Sp^c,\cat D)^c) \ar[d,"\Spc(F_d)"] \\ 
		\coprod_{1\leq k \leq d}\Spc(\Sp^c) \ar[r,"\varphi","\sim"'] & \Spc(\Exc{d}(\Sp^c,\Sp)^c).
		\end{tikzcd}
        \end{equation}
    Since the composite through the lower left corner separates components, so does~$\varphi^{\cat D}$ as desired. 
\end{proof}

\begin{Rem}\label{rem:noetherian}
	Since a finite union of noetherian subspaces is a noetherian space, \cref{prop:setspc_coefficients} implies that $\Spc(\Exc{d}(\Sp^c,\cat D)^c)$ is noetherian whenever $\Spc(\cat D^c)$~is. 
\end{Rem}

\begin{Rem}\label{rem:integralcoefficients}
    We now completely determine the topology of $\Spc(\Exc{d}(\Sp^c,\cat D)^c)$ in the special case that $\cat D = \Mod_{\HZ}$. This argument is meant as a proof of concept; we will return to a more systematic discussion elsewhere. 
\end{Rem}

\begin{Rem}
    Repeating the arguments given in, and leading up to, \cite[Corollary 4.11]{PatchkoriaSandersWimmer22}, one can show that there is a geometric equivalence
    \[
		\Exc{d}(\Sp^c,\Mod_{\HZ}) \simeq \Mod_{\Exc{d}(\Sp^c,\Sp)}(i_d\HZ),
    \]
    induced by base-change along $\bbS \to \HZ$.
\end{Rem}

\begin{Not}
    Henceforth, in functors such as $\partial$ and $i_d$, we will abbreviate superscripts `$\Mod_{\HZ}$' by `$\Z$'. Let $1 \leq k \leq d$ and write $\cPd^{\Z}(\num{k},\mathfrak p)$ for the prime tt-ideal in $\Exc{d}(\Sp^c,\Mod_{\HZ})^c$ that corresponds to $\mathfrak p \in \Spec(\Z) \cong \Spc(\Mod_{\HZ}^c)$ under the bijection of \cref{prop:setspc_coefficients}. Explicitly, this means that 
	\[
		\cPd^{\Z}(\num{k},\mathfrak p) = \{x \in \Exc{d}(\Sp^c,\Mod_{\HZ})^c \mid \kappa(\mathfrak p) \otimes \partial_k^{\Z}(x) = 0\},
	\]
    where $\kappa(\mathfrak p)$ denotes the residue field of $\Z$ at $\mathfrak p$.
\end{Not}

\begin{Thm}\label{thm:spc_integral}
    Let $1 \le k,l \le d$ be integers and consider two prime ideals $\mathfrak p, \mathfrak q \in \Spec(\Z)$. Then there is an inclusion $\cPd^{\Z}(\num{k},\mathfrak p) \subseteq \cPd^{\Z}(\num{l},\mathfrak q)$ if and only if one of the following two conditions is satisfied:
        \begin{enumerate}
            \item $\mathfrak p = (p)$ for some prime $p$, $\mathfrak q = (p)$ or $\mathfrak q = (0)$, and $p-1 \mid k-l \geq 0$;
            \item $\mathfrak p = (0) = \mathfrak q$ and $k=l$.
        \end{enumerate}
    Moreover, $\Spc(\Exc{d}(\Sp^c,\Mod_{\HZ})^c)$ is noetherian, so the topology is determined by these inclusions. Finally, base-change $\Sp\to\Mod_{\HZ}$ induces a map
		\[
			\Spc(\Exc{d}(\Sp^c,\Mod_{\HZ})^c) \to \Spc(\Exc{d}(\Sp^c,\Sp)^c)
		\] 
   which is a homeomorphism onto its image. It maps $\cPd^{\Z}(\num{k},(p))$ to $\cPd(\num{k},p,\infty)$ and maps $\cPd^{\Z}(\num{k},(0))$ to $\cPd(\num{k},0,1).$
\end{Thm}

\begin{proof}
    The spectrum of $\Exc{d}(\Sp^c,\Mod_{\HZ})^c$ is noetherian by \cref{rem:noetherian}, so in order to understand its topology, it is enough to determine all inclusions among prime tt-ideals. Suppose first that we are given an inclusion $\cPd^{\Z}(\num{k},\mathfrak p) \subseteq \cPd^{\Z}(\num{l},\mathfrak q)$. The commutative square \eqref{eq:dlinearsquare} shows that $\Spc(F) \colon \Spc(\Mod_{\HZ}^c) \to \Spc(\Sp^c)$ satisfies 
        \[
            \Spc(F)(\cPd^{\Z}(\num{k},\mathfrak p)) = \cPd(\num{k},\Spc(F)(\mathfrak p)),
        \]
    where $\Spc(F)(\mathfrak p)$ can be identified with a pair $(p,h) \in \mathbb{P} \times \{1,\infty\}$. Since $\Spc(F_d)$ preserves inclusions (\cref{Rem:basic-balmer-properties}), we thus obtain an inclusion
        \[
			\cPd(\num{k},\Spc(F)(\mathfrak p)) \subseteq \cPd(\num{l},\Spc(F)(\mathfrak q))
        \]
    of prime tt-ideals in $\Exc{d}(\Sp^c,\Sp)^c$. Our classification theorem\footnote{In fact, we do not require the full strength of this theorem; it is enough to combine \cref{cor:inclusion_chromatic,rem:pprimary,prop:inclusions_formalproperties}.} \cref{thm:posetstructure} then implies that either condition $(a)$ or $(b)$ must hold. 

	For the converse, assume first that $(b)$ is satisfied. In this case, we have an equality $\cPd^{\Z}(\num{k},\mathfrak p)  = \cPd^{\Z}(\num{l},\mathfrak q)$ as desired. Similarly, if $(a)$ holds with $\mathfrak q = (0)$, then $\cPd^{\Z}(\num{l},\mathfrak q') \subseteq \cPd^{\Z}(\num{l},\mathfrak q)$ for any $\mathfrak q' \in \Spec(\Z)$, in particular for $\mathfrak p = (p)$. In order to show $\cPd^{\Z}(\num{k},\mathfrak p) \subseteq \cPd^{\Z}(\num{l},\mathfrak q)$, we are therefore reduced to verifying the following claim:
        \[
            p-1 \mid k-l \geq 0 \implies \cPd^{\Z}(\num{k},\mathfrak p) \subseteq \cPd^{\Z}(\num{l},\mathfrak p).
        \]
	To this end, suppose first that there exists a $p$-power partition $\lambda \vdash k$ of length $l$. Consider some $x \in \cPd^{\Z}(\num{k},\mathfrak p)$, i.e., a compact functor $x \in \Exc{d}(\Sp^c,\Mod_{\HZ})$ with
        \begin{equation}\label{eq:spc_integral1}
            0 = \Fp \otimes \partial_k^{\Z}(x) \simeq \partial_k^{\Z}(i_d^{\Z}(\Fp) \circledast x).
        \end{equation}
    Let $t_k^{\Z}$ be the Tate construction (\cref{def:tate-functor}) internal to $\Exc{d}(\Sp^c,\Mod_{\HZ})$ with respect to the subset 
        \[
            Y_k \coloneqq \SET{\cPd^{\Z}(\num{i},\mathfrak p)}{i \geq k} = \Spc(F_d)^{-1}(\SET{\cPd(\num{i},p,h)}{i \geq k}).
        \]
    Note that $Y_k$ is a Thomason subset since it is the pullback of a Thomason subset along a spectral map. As in the proof of \cref{lem:basic-tate-inclusion}, we then deduce formally from~\eqref{eq:spc_integral1} that $t_k^{\Z}(i_d^{\Z}(\Fp) \circledast x) = 0$. Using the compatibility of the Tate construction with base-change (\cref{prop:idempotents-under-tt-functors}) as well as \cref{rem:generalcoefficients} and the dualizability of~$x$, we compute
        \begin{align*}
            0 = \partial_l^{\Z}t_k^{\Z}(i_d^{\Z}(\Fp) \circledast x) & \simeq (\partial_l^{\Z}t_k^{\Z}i_d^{\Z}(\Fp)) \otimes \partial_l^{\Z}(x) \\
            & \simeq (\partial_l^{\Z}t_k^{\Z}i_d^{\Z}(\Z \otimes \bbS/p))) \otimes \partial_l^{\Z}(x) \\
            & \simeq (\partial_l^{\Z}t_k^{\Z}F_di_d(\bbS/p)) \otimes \partial_l^{\Z}(x) \\
            & \simeq F(\partial_lt_k(i_d(\bbS/p))) \otimes \partial_l^{\Z}(x) \\
            & \simeq (\Fp \otimes \partial_lt_ki_d(\bbS)) \otimes \partial_l^{\Z}(x).
        \end{align*}
    In order to show that $x \in \cPd^{\Z}(\num{l},\mathfrak p)$, it thus suffices to prove that $\Fp \otimes \partial_lt_ki_d(\bbS)$ is non-trivial, so that it contains $\Fp$ as a retract. To this end, observe that we have ring maps
        \[
            \partial_lt_ki_d(\bbS) \to \partial_lt_ki_d(\Fp) \quad \text{and} \quad  \Fp^{t\Fnt(\lambda)} \to \Fp^{tC},
        \]
    where $C$ is a non-trivial cyclic $p$-group and the second map is the one we constructed in the proof of \cref{thm:tateblueshift}. On the one hand, the target of the second map is non-zero, hence $\Fp^{t\Fnt(\lambda)}$ is also non-zero. On the other hand, \cref{prop:tatederivatives_equivariantformula} implies that $\Fp^{t\Fnt(\lambda)}$ is a retract of $\partial_lt_ki_d(\Fp)$, so the target of the first map is non-zero and $\Fp$-linear, and we conclude that $\Fp \otimes \partial_lt_ki_d(\bbS) \neq 0$, as claimed. 
 
    This finishes the proof under the assumption that there is a $p$-power partition of $k$ of length $l$. In general, our assumption that $p-1 \mid k-l \geq 0$ guarantees the existence a chain of $p$-power partitions (see \cref{rem:p-power-and-divisibility}) and we conclude by concatenating the inclusions just established.
\end{proof}

\begin{Rem}
	The $d=3$ case is shown in \Cref{fig:zarizki_a3z}. Observe that it is homeomorphic to the top and bottom chromatic layers of \cref{fig:exc3}. The relations between these spaces are completely analogous to the situation in equivariant stable homotopy theory established in \cite{PatchkoriaSandersWimmer22}.
\end{Rem}

\begin{figure}[h!]%
\includegraphics{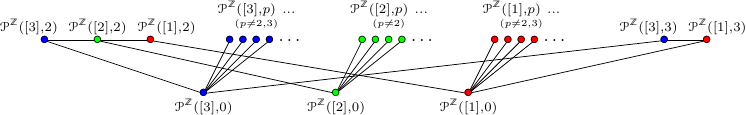}\caption{The Balmer spectrum $\Spc(\Exc{3}(\Sp^c,\Mod_{\HZ})^c)$.}\label{fig:zarizki_a3z}
\end{figure}

\begin{Cor}\label{cor:spc_integral}
	Let $d \ge 1$. Say that a subset $Y \subseteq \num{d}\times \Spec(\Z)$ is \emph{admissible} if it satisfies the following closure properties:
	\begin{enumerate}
		\item If $(l,0)\in Y$ then $(k,p)\in Y$ for all $p-1\mid k-l \ge 0$.
		\item If $(l,p) \in Y$ then $(k,p)\in Y$ for all $p-1\mid k-l \ge 0$.
	\end{enumerate}
	There is an inclusion-preserving bijection between the set of admissible subsets of $\num{d}\times\Spec(\Z)$  and the collection of tt-ideals of $\Exc{d}(\Sp^c,\Mod_{\HZ})^c$ given by 
	\begin{align*}
		Y &\mapsto \SET{ x \in \Exc{d}(\Sp^c,\Mod_{\HZ})^c}{\partial^{\Z}_k(x) \in \mathfrak{p} \text{ if } (k,\mathfrak{p}) \not\in Y} \\
	\intertext{with inverse}
		\cat C &\mapsto \SET{(k,\mathfrak{p}) \in \num{d}\times\Spec(\Z)}{\partial^{\Z}_k(x) \not\in \mathfrak{p} \text{ for some } x \in \cat C}.
	\end{align*}
\end{Cor}

\begin{proof}
	As a set, we can identify $\Spc(\Exc{d}(\Sp^c,\Mod_\HZ)^c)$ with $[d]\times \Spec(\bbZ)$ by \cref{prop:setspc_coefficients}. Since the space is noetherian, a subset is Thomason if and only if it is specialization closed. According to \cref{thm:spc_integral}, a subset is specialization closed precisely when it has the closure properties $(a)$ and $(b)$. In other words, the admissible subsets are precisely the Thomason subsets. With this in hand, the rest is just a formulation of the correspondence between Thomason subsets and tt-ideals \cite[Theorem~4.10]{Balmer05a}.
\end{proof}

\begin{Rem}\label{rem:integral_bistratification}
    By base-change along the initial geometric functor $F\colon \Sp \to \Mod_{\HZ}$, we have that $P_{d-1}^{\Z}\colon \Exc{d}(\Sp^c,\Mod_{\HZ}) \to \Exc{d-1}(\Sp^c,\Mod_{\HZ})$ is a finite localization, while $\partial_d^{\Z}\colon \Exc{d}(\Sp^c,\Mod_{\HZ}) \to \Mod_{\HZ}$ is finite \'etale (this follows either by direct verification, or from \Cref{rem:top-derivative-etale} and \cite[Example 5.7]{Sanders22}). We are thus in a situation where we can apply the techniques of \cite{bhs1} and \cite{BCHS2023pp} to show that $\Spc(\Exc{d}(\Sp^c,\Mod_{\HZ})^c)$ is stratified and costratified over its spectrum and that the telescope conjecture holds for $\Exc{d}(\Sp^c,\Mod_{\HZ})$.
\end{Rem}

\begin{Exa}
	We can also consider $\Exc{d}(\Sp^c,\Mod_{\HbbF})$ for a field $\mathbb{F}$. It follows from the above techniques that $\Spc(\Exc{d}(\Sp^c,\HQ)^c)$ is the discrete space with $d$ points, and can be identified via $\Z\to\mathbb{Q}$ with the subspace of $\Spc(\Exc{d}(\Sp^c,\Mod_{\HZ})^c)$ consisting of $\SET{\cPd^\Z(\num{k},0)}{1\le k \le d}$. On the other hand, $\Spc(\Exc{d}(\Sp^c,\Mod_{\HFp})^c)$ identifies with the subspace $\SET{\cPd^\Z(\num{k},p)}{1\le k\le d}$ of $\Spc(\Exc{d}(\Sp^c,\Mod_{\HZ})^c)$. In other words, it is the set $\{1,2,\ldots,d\}$ equipped with the topology given by $k \in \smash{\overbar{\{l\}}}$ if and only if $p-1 \mid k-l \ge 0$. Finite spectral spaces correspond to finite posets (via the map which sends a space to its specialization poset, see \cite[1.1.16]{DickmannSchwartzTressl19}). From this perspective, it is the poset $(\num{d},\le_p)$ where $k\le_p l$ means that $p-1 \mid k-l \ge 0$.
\end{Exa}

\newpage
\addtocontents{toc}{\vspace{\normalbaselineskip}}
\begin{landscape}
\appendix
\section{Equivariant homotopy and Goodwillie calculus dictionary}\label{appendix}
For readers familiar with equivariant homotopy theory, we include here a `dictionary' for translating between equivariant homotopy theory and Goodwillie calculus. When $d=2$ there is an equivalence of categories $\Sp_{C_2} \simeq \Exc{2}(\Sp^c,\Sp)$ \cite{Glasman18pp} and in this case the analogies described by the table are genuine correspondences.

\smallskip

\begin{adjustbox}{width=1.36\textwidth}
\begin{tabular}[b!]{@{}l|l@{}}
\toprule
Equivariant homotopy & Goodwillie calculus \\ \midrule
$\Sp_G$ & $\Exc{d}(\Sp^c,\Sp)$\\ \hdashline
Burnside ring $A(G)$ & Goodwillie--Burnside ring  $A(d) $ \\ \hdashline
Conjugacy class of subgroups $H \le G$ & Integer $1 \le i \le d$ \\\hdashline
Compact generator $\Sigma^{\infty}G/H_+$ & Compact generator $P_dh_s(i)$ \\\hdashline
Fixed points functor $(-)^H \colon \Sp_G \to \Sp$ & Cross-effect $\crosseffect_i(-)(\bbS,\ldots,\bbS) \colon \Exc{d}(\Sp^c,\Sp) \to \Sp$ \\\hdashline
Restriction to the trivial group $\res^G_e\colon \Sp_G \to \Sp$ & $d$-th cross-effect $\crosseffect_d(-)(\bbS,\ldots,\bbS) \colon \Exc{d}(\Sp^c,\Sp) \to \Sp$  \\\hdashline
`Partial' geometric fixed points $\tilde \Phi^H \colon \Sp_G \to \Sp_{W_G H}$ &    $P_i \colon \Exc{d}(\Sp^c,\Sp) \to \Exc{i}(\Sp^c,\Sp)$ \\\hdashline
Geometric fixed points $\Phi^H = \res^{W_G H}_e \circ \tilde \Phi^H \colon \Sp_G \to \Sp$ & $i$-th derivative $\partial_i = \crosseffect_i P_i(-)(\bbS,\ldots,\bbS) \colon \Exc{d}(\Sp^c,\Sp) \to \Sp$ \\\hdashline
Inflation $\inf_1^G \colon \Sp \to \Sp_G$ & $i_d \colon \Sp \to \Exc{d}(\Sp^c,\Sp)$ \\\hdashline
Free $G$-spectra & $d$-homogeneous functors \\\hdashline
Tom Dieck splitting: $\displaystyle (\mathbb{S}_G^0)^G \simeq \bigoplus_{(H)} \Sigma^{\infty}(S^0)_{hW_GH}$ & Snaith splitting: $\displaystyle \crosseffect_1P_dh_{\bbS} \simeq \bigoplus_{1 \le i \le d} \Sigma^{\infty}(S^0)_{h\Sigma_i}$ \\\hdashline
   \makecell[l]{Double coset formula:\\ $ \displaystyle
    \Sigma^{\infty}G/H_+ \otimes \Sigma^{\infty}G/K_+ \simeq \bigoplus_{g \in H \backslash G/K} \Sigma^{\infty}G/(H^g \cap K)_+$}
     &  \makecell[l]{Good subset decomposition:\\$\displaystyle P_dh_{\bbS}(i) \circledast P_dh_{\bbS}(j) \simeq \bigoplus_{\substack{\mathcal{U} \subseteq \num{i} \times \num{j} \\ \
\mathcal{U} \text{ good} \\|\mathcal{U}| \le d}} P_dh_s(|U|)$} \\ \hdashline
\makecell[l]{Tate square: \\
$\begin{tikzcd}[ampersand replacement=\&]
	X \& {\tilde EG \otimes X} \\
	{F(EG_+,X)} \& {\tilde EG \otimes F(EG_+,X)}
	\arrow[from=1-1, to=1-2]
	\arrow[from=1-1, to=2-1]
	\arrow[from=2-1, to=2-2]
	\arrow[from=1-2, to=2-2]
 \arrow["\ulcorner"{anchor=center, pos=0.125}, draw=none, from=1-1, to=2-2]
\end{tikzcd}$}& \makecell[l]{Kuhn--McCarthy square:\\ %
$\begin{tikzcd}[ampersand replacement=\&]
	{P_dF(X)} \& {(\partial_dF \otimes X^{\otimes d})^{h\Sigma_d}} \\
	{P_{d-1}F(X)} \& {(\partial_dF \otimes X^{\otimes d})^{t\Sigma_d}}
	\arrow[from=1-1, to=2-1]
	\arrow[from=1-1, to=1-2]
	\arrow[from=2-1, to=2-2]
	\arrow[from=1-2, to=2-2]
	\arrow["\ulcorner"{anchor=center, pos=0.125}, draw=none, from=1-1, to=2-2]
\end{tikzcd}$}\\
\bottomrule
\end{tabular}
\end{adjustbox}

\end{landscape}

\newpage
\bibliographystyle{alpha}\bibliography{tt-geo}
\end{document}